\documentclass[10pt]{article}
\usepackage[english]{babel}
\usepackage{natbib}
\usepackage{url}
\usepackage{inputenc}
\usepackage{amsfonts}
\usepackage{amsmath}
\usepackage{empheq}
\usepackage{graphicx}
\usepackage{subcaption}
\graphicspath{{images/}}
\usepackage{parskip}
\usepackage{fancyhdr}
\usepackage{bbold}
\usepackage{amssymb}
\usepackage{vmargin}
\usepackage{array}
\usepackage{amsthm}
\usepackage[ruled,vlined]{algorithm2e}
\usepackage{hyperref}
\usepackage{caption}
\usepackage{float}
\usepackage{fourier-orns}
\usepackage{listings}
\usepackage{epsfig}
\usepackage{amsmath}             \usepackage[scr=boondoxo,scrscaled=1.05]{mathalfa}

\usepackage[shortlabels]{enumitem}
\mathtoolsset{showonlyrefs=true}

\usepackage{wasysym}
\usepackage{tikz}
\newcommand{\parallelogram}{%
  \tikz[scale=0.2, baseline=-0.5ex]{%
    \draw (0,0) -- (1,0) -- (1.5,1) -- (0.5,1) -- cycle;
  }%
}

\newtheorem{theorem}{Theorem}[section]
\newtheorem{corollary}{Corollary}[section]
\newtheorem{lemma}[theorem]{Lemma}
\newtheorem{remark}{Remark}[section]
\theoremstyle{definition}

\usepackage{listings}

\newcommand{\bydef}{\stackrel{\mbox{\tiny\textnormal{\raisebox{0ex}[0ex][0ex]{def}}}}{=}}

\usepackage[dvipsnames]{xcolor}

\makeatletter
\newcommand\subsubsubsection{\@startsection{paragraph}{4}{\z@}{-2.5ex\@plus -1ex \@minus -.25ex}{1.25ex \@plus .25ex}{\normalfont\normalsize\bfseries}}
\newcommand\subsubsubsubsection{\@startsection{subparagraph}{5}{\z@}{-2.5ex\@plus -1ex \@minus -.25ex}{1.25ex \@plus .25ex}{\normalfont\normalsize\bfseries}}
\makeatother
\title{Proving periodic solutions and branches in the 2D Swift Hohenberg PDE with hexagonal and triangular symmetry}
\author{
Dominic Blanco
\footnote{McGill University, Department of Mathematics and Statistics, 805 Sherbrooke Street West, Montreal, QC, H3A 0B9, Canada. {\tt dominic.blanco@mail.mcgill.ca}}}
\begin{document}

\maketitle
\begin{abstract}
    In this article, we enforce space group symmetries in Fourier series to rigorously prove the existence of smooth, periodic solutions in partial differential equations (PDEs) with hexagonal and triangular symmetries. In particular, we provide the necessary analytical and numerical tools to construct Fourier series of functions on the hexagonal lattice. This allows one to build approximate solutions that are periodic. Moreover, to generate the periodic tiling, we can use one symmetric hexagon for $D_6$ symmetry and two symmetric triangles for $D_3$ symmetry. We derive a Newton-Kantorovich approach based on the construction of an approximate inverse around an approximate solution, $\overline{u}$. More specifically, we verify a condition based on the computation of explicit bounds. The strategy for constructing $\overline{u}$, the approximate inverse, and the computation of these bounds will be presented. We demonstrate our approach on the 2D Swift-Hohenberg PDE by proving the existence of $D_3$ and $D_6$ periodic solutions. We then perform proofs of branches of solutions by using Chebyshev series. The algorithmic details to perform the proof can be found on Github.
    \end{abstract}

\begin{center}
{\bf \small Key words.} 
{ \small Periodic Solutions, Branches of Solutions, Swift Hohenberg,  Hexagonal Lattice Symmetries, Dihedral symmetry, Computer-Assisted Proofs}
\end{center}
\section{Introduction}
In this paper, we develop a methodology for constructively proving the existence (and local uniqueness) of solutions to partial differential equations (PDEs) with symmetries representable on the hexagonal lattice. More specifically, we will focus on $D_3$ and $D_6$-symmetric solutions. By $D_j$, we mean the dihedral group of order $2j$ which is the symmetry group of the $j$-gon. These groups have presentation
\begin{align}
    D_j = < r,s ~ | ~ r^{j} = s^2 = 1, ~ rs = sr^{-1} >.\label{def : Dj presentation}
\end{align}
We will illustrate such an approach in the case of the 2D Swift-Hohenberg (SH) PDE
\begin{align}\label{original SH}
    -u_t &= (I_d + \nabla^2)^2 u + \mu u - \gamma u^2 + u^3, ~~ u = u(x,t).
\end{align}
We wish to establish steady states of \eqref{original SH} with dihedral symmetries, namely
\begin{align}\label{eq : swift_hohenberg}
     &0 = (I_d + \nabla^2)^2 u + \mu u - \gamma u^2 + u^3, ~~ u = u(x),~ u ~ \text{is such that} ~ u(x) = u(g \cdot x) ~ \text{for} ~ \text{all} ~ g \in D_j.
\end{align}
\begin{figure}[t]
\centering
 \begin{minipage}{.33\linewidth}
  \centering\epsfig{figure=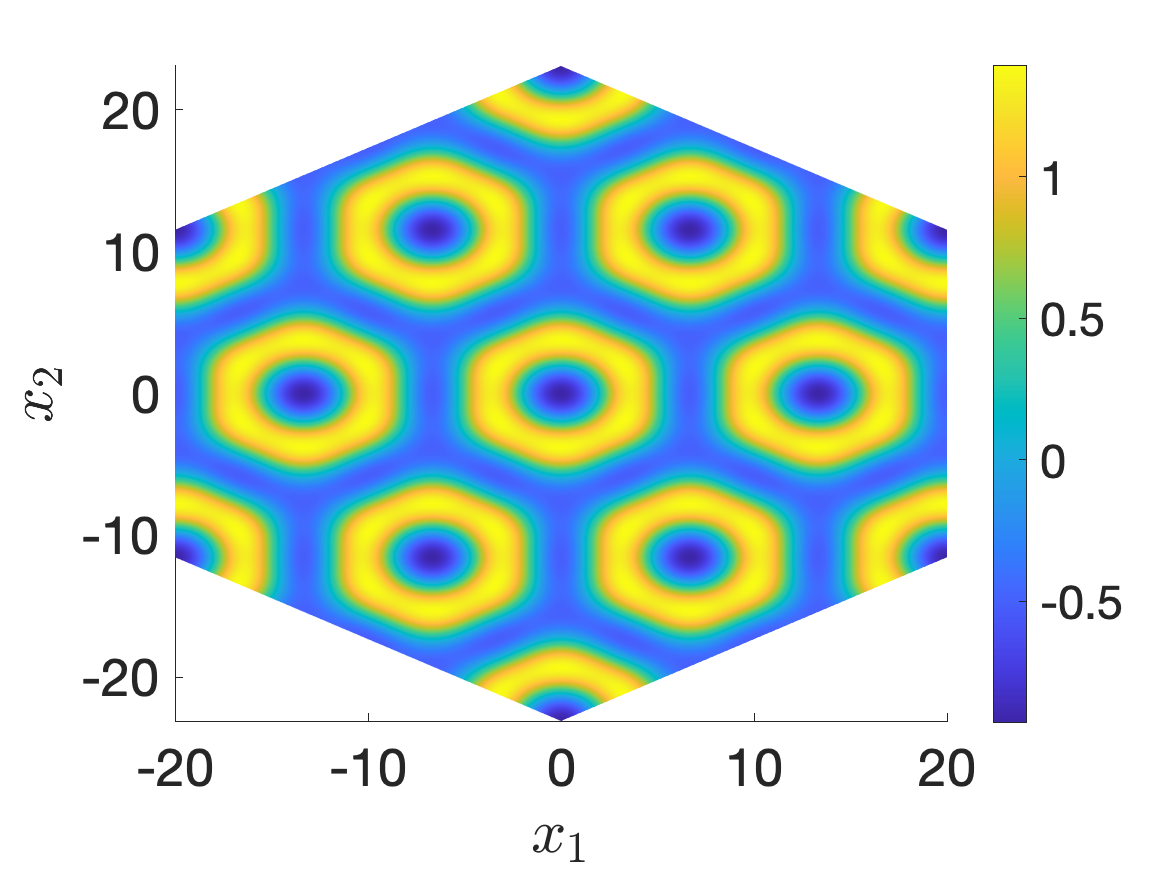,width=\linewidth}
  \end{minipage}%
 \begin{minipage}{.33\linewidth}
  \centering\epsfig{figure=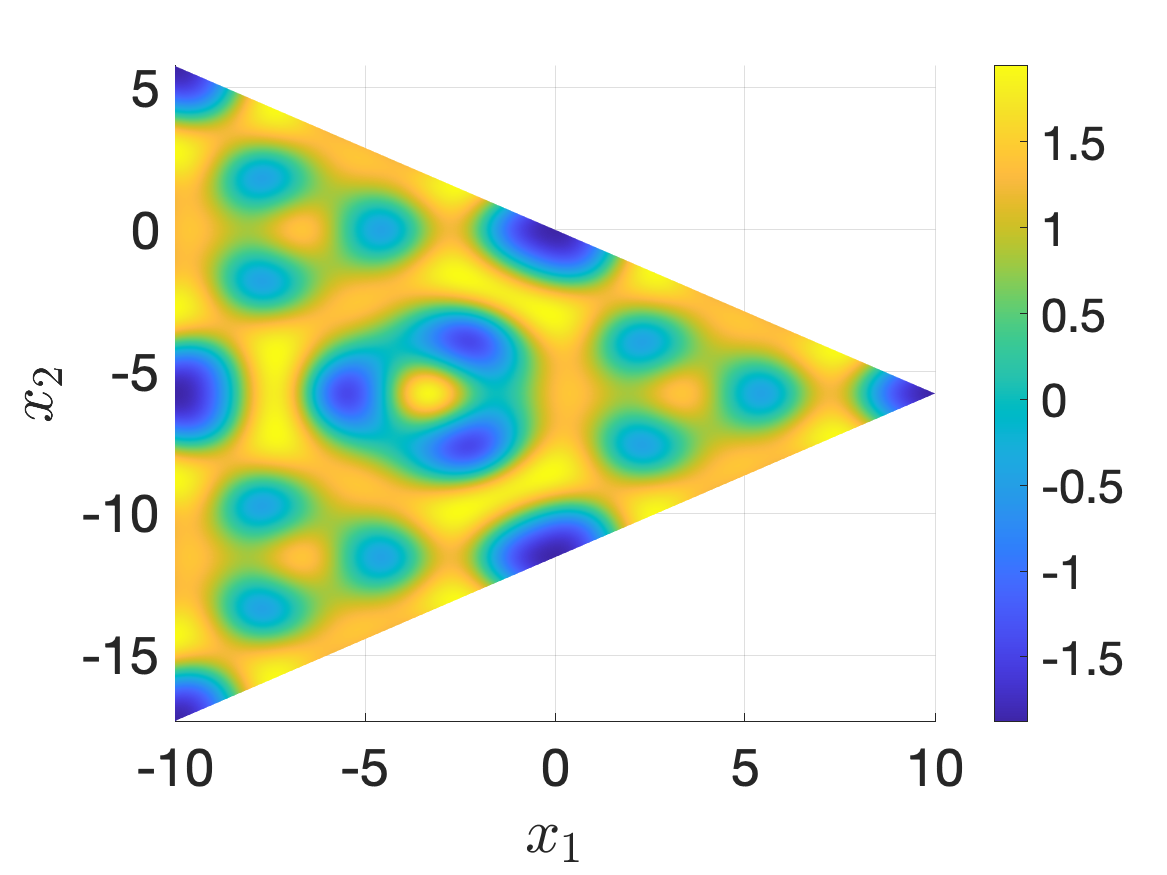,width=\linewidth}
 \end{minipage} 
 \begin{minipage}{.33\linewidth}
  \centering\epsfig{figure=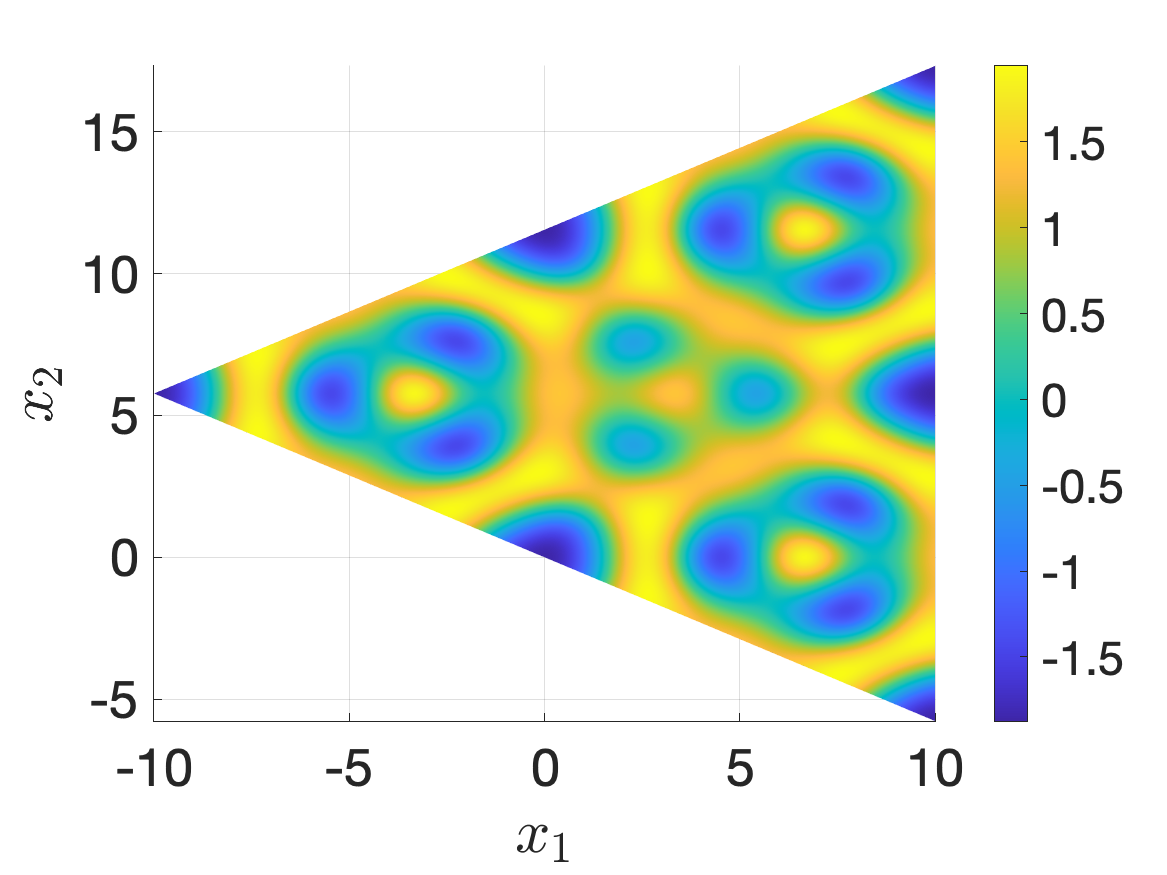,width=\linewidth}
 \end{minipage} 
 \caption{An approximate solution on the branch presented in Figure \ref{fig : th4_1} (green point) when $\mu \approx 0.01$ (left). An approximate solution on the branch presented in Figure \ref{fig : th3_1} (green point) when $\mu \approx -0.3$ on $\Delta_1$ (center) and $\Delta_2$ (right). See Theorem \ref{th : triangle periodic} for a definition.}\label{fig : intro}
 \end{figure}%
Note that we have set $u_t = 0$ as we are looking for steady states.
We wish to prove the existence, local uniqueness, and symmetry of solutions to \eqref{eq : swift_hohenberg}. The Swift-Hohenberg PDE has been investigated in a wide variety of studies. In particular, \eqref{eq : swift_hohenberg} is known to exhibit a rich variety of well documented dihedral solutions. Various numerical and analytical studies of SH have led to a deeper understanding of dihedral patterns such as hexagonal \cite{hexagon2021lloyd,hexagon2008} and square \cite{squareSakaguchi_1997} patterns. More generally, the works of \cite{jason_review_paper}, \cite{jason_spot_paper}, and \cite{jason_ring_paper} demonstrate an approach to numerically compute dihedral solutions. By using radial coordinates, the authors provide a method to construct approximate solutions using a Galerkin projection starting with $\mu$ small. By then performing continuation, one can obtain localized planar patterns of various dihedral symmetries. These can form both spot patterns (see \cite{jason_spot_paper}) and ring patterns (see \cite{jason_ring_paper}). Computer assisted proofs have been used in 2D and 3D to obtain existence results on rectangular domains in \cite{lessard_2,lessard_1}.

\par In this paper, we demonstrate how we can use symmetries on the hexagonal lattice for proving the existence of periodic solutions to \eqref{eq : swift_hohenberg}. As there is still a heavy interest in periodic solutions to \eqref{eq : swift_hohenberg}, we would like to provide a general approach to prove the solution and its symmetry for any value of $\mu$. Note that we will be able to prove solutions when $\mu < 0$ as we are interested in periodic solutions, and our approach will not require the linear part to be invertible. To begin, we will need a way to represent symmetries on the hexagonal lattice in Fourier series. We are interested in using approaches which reduce the computational complexity of the problem at hand. This is especially important as our problems will be posed in 2D, so performing proofs without any reduction would be difficult if not impossible. To see the benefit, consider even symmetry in 1D. When a periodic solution has even symmetry, it can be written as a Cosine Fourier series. Such a series only has positive Fourier coefficients, meaning the number of coefficients we need to store  in memory is nearly halved. For space groups in higher dimensions with more symmetries, the reduction becomes even greater. Even when memory is not a concern, enforcing the symmetry is of interest when one wishes to find a solution with a specific symmetry. Without enforcing it, there is no way to guarantee that the solution obtained has the desired symmetry. Such techniques on symmetric Fourier coefficients have been deeply developed in \cite{JB_symmetries_1} and \cite{JB_symmetries_2}. More specifically, the authors describe a way to build a Fourier series that possesses the desired symmetry.  The methods developed in \cite{JB_symmetries_1} and \cite{JB_symmetries_2} are applicable to any space group in any dimension. The authors use an algorithm to compute a reduced set of Fourier indices under the symmetry of a space group. Furthermore, the method allows one to determine the specific relations amongst the coefficients. In particular, it relies on computing a set which contains one element from each orbit (see \cite{gallianalgebra} for a definition) of the given space group. Such a set is called a fundamental domain, denoted $\mathcal{Z}_{\mathrm{dom}}(\mathcal{G})$. To demonstrate the use of $\mathcal{Z}_{\mathrm{dom}}(\mathcal{G})$, suppose $u$ is a Fourier series with Fourier coefficients $(u_n)_{n \in \mathbb{Z}^m}$. If an index, $n$, is in the fundamental domain, then the other elements in its orbit (in the sense of group action) will we related to one another. That is, for any $g \in \mathcal{G}$, we can find a relation between $u_n$ and  $u_{g \cdot n}$ where $\cdot$ denotes the group action. This relation already reduces the number of Fourier indices one must store. This is especially important when performing computer assisted proofs in higher dimensions, such as in \cite{van2021spontaneous}. Additionally, the authors of \cite{JB_symmetries_1} and \cite{JB_symmetries_2} define two additional sets. The first set, $\mathcal{Z}_{\mathrm{triv}}(\mathcal{G})$, is called the trivial set. This set contains all Fourier indices which have values of $0$ by symmetry. We emphasize that these are $0$ due to the symmetry itself. In other words, it is possible that the are other indices that are $0$ by the nature of the solution. Such $0$ indices are not included in the trivial set. The indices which are $0$ depend on the symmetry being enforced.  The second is defined as
\begin{align}
    \mathcal{Z}_{\mathrm{sym}}(\mathcal{G}) \bydef \mathbb{Z}^m \setminus \mathcal{Z}_{\mathrm{triv}}(\mathcal{G}).
\end{align}
Since the trivial set contains all indices corresponding to coefficients which are $0$, the second set, $\mathcal{Z}_{\mathrm{sym}}(\mathcal{G})$, contains all indices corresponding to coefficients which are non-zero. Hence, we would like to disregard the indices in $\mathcal{Z}_{\mathrm{triv}}$ numerically and only consider those in $\mathcal{Z}_{\mathrm{sym}}$. Once $\mathcal{Z}_{\mathrm{dom}}(\mathcal{G})$ and $\mathcal{Z}_{\mathrm{sym}}(\mathcal{G})$ are obtained, one can define the reduced set of Fourier coefficients
 \begin{align}
     \mathcal{Z}_{\mathrm{red}}(\mathcal{G}) \bydef \mathcal{Z}_{\mathrm{dom}}(\mathcal{G}) \cap \mathcal{Z}_{\mathrm{sym}}(\mathcal{G}).\label{def : reduced set}
 \end{align}
 Indeed, the relations determined by $\mathcal{Z}_{\mathrm{dom}}(\mathcal{G})$ and $\mathcal{Z}_{\mathrm{triv}}(\mathcal{G})$ allow one to write a Fourier series only featuring indices in $\mathcal{Z}_{\mathrm{red}}(\mathcal{G})$. Such a Fourier series will enforce the needed relations numerically. This allows one to build an approximate solution that is necessarily of the symmetry enforced. Using this approach in \cite{JB_symmetries_1}, the authors obtain rigorous proofs of existence and local uniqueness of periodic solutions to 3D Ohta-Kawasaki with a SG229 and SG230 symmetry. In \cite{JB_symmetries_2}, they generalize the approach and allow for one to attempt rigorous proofs for any space group symmetry. Since the rigorous approach used relies on a contraction argument, it can be shown that the true solution also has the symmetry of the approximate solution. The method was also applied in \cite{van2021spontaneous}, where symmetries on the range of the function were also considered. The same algorithm can be applied and the computational details are outlined in Sections 5 and 6 of the aforementioned paper. 
 \par In this paper, we will focus on $D_3$ and $D_6$ symmetry, which are both expressible on the hexagonal lattice. We wish to do so using computer-assisted proofs (CAPs). The use of validated numerics and CAPs to study problems in differential equations and dynamical systems has become increasingly popular. While there are many approaches for CAPs, we will focus on one which relies on a Newton-Kantorovich type contraction argument. More specifically, one constructs a fixed point operator whose fixed points correspond to the zeros of the problem being studied. The goal is to then show that this operator is a contraction on some ball around a well-chosen numerical approximation of the solution. This makes the approach constructive in nature as it is based on building such an approximate solution, which can provide some desirable properties of the solution (its shape, location in the function space, among others). In order to show that the fixed point operator is contracting on such a ball, we will need to estimate certain bounds partially by hand and partially on the computer using interval arithmetic (cf. \cite{julia_interval}). As we are interested in hexagonal lattice symmetries, we will need to provide the necessary estimates to perform a computer-assisted proof using Fourier series posed on the hexagonal lattice. As the methods developed by the authors of \cite{JB_symmetries_1,JB_symmetries_2} are done generally for any space group, their techniques readily apply. However, we would like to use RadiiPolynomial.jl in Julia \cite{julia_olivier} to perform our rigorous computations. This means we need to develop the Julia code to enforce $D_3$ and $D_6$ symmetries. Indeed, \cite{julia_olivier} only has the ability to enforce $D_2$-symmetry. One can also enforce $D_4$-symmetry by using the code provided at \cite{dominic_D_4_julia}. We wish to generalize the code of \cite{dominic_D_4_julia} to provide a complete package of all dihedral space groups. The implementation of such a package is nontrivial in a variety of aspects. Firstly, one must compute the reduced set of Fourier indices and store it in a vector. This means we need to find a mapping between the entries of a vector and the reduced set of indices which are tuples since we are in 2D. Following this, we need to develop a way to produce the necessary objects for our CAP. This includes addition, convolution, ways to build multiplication operators, ways to evaluate the sequence with this symmetry, norms, and more. We describe these in more detail in Section \ref{sec : dihedral.jl}. Moreover, since we have $D_3$ and $D_6$-symmetry, we would like to demonstrate how we can generate the periodic tiling on $\mathbb{R}^2$ using triangles and hexagons respectively. This is of interest as Fourier series on the hexagonal lattice will be defined on a parallelogram; however, under the symmetry of $D_3$ and $D_6$, we will show that one can use two $D_3$-symmetric triangles to create the $D_3$ tiling and one $D_6$-symmetric hexagon for the $D_6$ tiling. This is possible since triangles and hexagons form a tiling of the plane. 
 \par Along with proving the solution, we would also like to develop the theory to prove branches with its symmetry and periodic properties. Such approaches are well studied. For some of the first examples, we refer the interested reader to \cite{arioli_sug2,continue3,continue2,continue1}. Recently, a new approach was developed by the author of \cite{polynomial_chaos}. Using a constructive approach, the author first identifies an approximate branch of periodic solutions. This approximate branch can be represented using Chebyshev series. This Chebyshev representation is then used for the Newton-Kantorovich argument. A CAP can then be performed to obtain a rigorous enclosure of the true branch of periodic solutions. The approach has been applied in a variety of studies, such as \cite{continuation_1,continuation_2,continuation_3,cadiot_witham}. In \cite{cadiot_witham}, the author wrote a branch of solitons in the Whitham equation as a Chebyshev series where the dependency is in the wave speed parameter. Fourier transforms can then be used to compute the bounds for the CAP. When there are saddle-node (fold) bifurcations present on the branch, one can write the Chebyshev series depending on its pseudo-arclength. This requires one to use an augmented system and solve for a parameter along with the solution itself. This has been done in \cite{maxime_general,maxime_paper_continuation,marschal,olivier_kevin_paper,continue1} where the augmented system was considered in order to perform the proof. Note that the approach does not directly lead to a proof of the saddle-node bifurcation itself, but the augmented system allows for one to move past such bifurcations. This is possible due to the underlying mechanics being pseudo-arclength continuation. In this manuscript, we will use the latter where we expand the branch of a Chebyshev series of its pseudo-arclength. We aim to provide more of the fine details on the approach and how the objects are constructed. More specifically, we describe the numerical approach used to obtain the Chebyshev approximation of the branch and approximate inverse for the Newton-Kantorovich theorem. Furthermore, rather than stating that we use the previously obtained bounds but with the objects defined on Chebyshev series, we explicitly write the formulas for the bounds and how they can be estimated on the computer. By describing the approach in great detail, we hope to increase the usability of the approach for future studies.
 \par The rest of this paper is organized as follows. In Section \ref{sec : symmetric fourier coefficients}, we begin with a discussion on symmetric Fourier series on the hexagonal lattice. This will allow us to enforce $D_3$ and $D_6$ symmetries theoretically. Then, we discuss how one can generate the periodic tiling via triangles and hexagons. As it turns out, the result is fundamentally different for triangles and hexagons, so our discussion will require two proofs: one for $D_3$ (see Theorem \ref{th : triangle periodic}) and one for $D_6$ (see Theorem \ref{th : hexagon periodic}). Following this in Section \ref{sec : CAP}, we perform our computer-assisted analysis to prove $D_3$ and $D_6$ periodic solutions. This begins with a discussion of the numerical aspects in Section \ref{sec : numerical aspects}. We include more details on the implementation of the $D_3$ and $D_6$ symmetries in a way that is compatible with the library \cite{julia_olivier} in Section \ref{sec : dihedral.jl}. Then we obtain the required bounds to perform the CAP in Section \ref{sec : bounds}, followed by a presentation of our results in Section \ref{sec : results solutions}. In Section \ref{sec : continuation}, we develop the necessary tools for proving branches of solutions with $D_3$ and $D_6$-symmetry. This not only provides us with a proof of the branch, but also that the entire branch is symmetric. That is, each periodic pattern on this branch can be generated by either two triangles (in the case of $D_3$) or one hexagon (in the case of $D_6$). The results are presented in Section \ref{sec : results branches}. Finally, we conclude in Section \ref{sec : conclusion}, and remark on future directions.
\section{Symmetric Fourier coefficients}\label{sec : symmetric fourier coefficients}
As mentioned in the introduction, our approach heavily relies on symmetric Fourier series. From that perspective, we introduce related notations and results from \cite{JB_symmetries_1} and \cite{JB_symmetries_2}. To begin, let $D_j$ be the dihedral group of order $2j$. We will focus on $j = 3,6$ as we are interested in solutions that are periodic on a triangle and/or hexagon. Let $d$ be a constant such that $0<d <\infty$.  Then, we define  $$\tilde{n} = (\tilde{n}_1,\tilde{n}_2)\bydef \left( \frac{n_1\pi}{d},\frac{n_2\pi}{d} \right) \in \mathbb{R}^2$$ for all $(n_1,n_2,) \in \mathbb{Z}^2$. We want to restrict to Fourier series representing $D_j$-symmetric functions for $j = 3,6$. To do so, we present the following lemma.
\begin{lemma}\label{lem : G-fourier series}
Let $u$ be a Fourier series of the form 
\begin{equation}\label{usual_fourier_series}
    u(x) = \sum_{n \in \mathbb{Z}^2} u_n e^{i (\mathcal{L} \tilde{n}) \cdot x}
\end{equation}
where $\mathcal{L} \bydef \begin{bmatrix}
    1 & -\frac{1}{2} \\
    0 & \frac{\sqrt{3}}{2}
\end{bmatrix}$ is the change of basis for the hexagonal lattice. Let $D_j$ be the dihedral group of order $2j$. Therefore, each $g \in D_j$ has a unitary representation of the form
\begin{align}
    gx = C_gx + 2\pi \mathcal{L}^{-T} D_g
\end{align}
where $C_g$ is an orthogonal $2 \times 2$ matrix and $D_g \in [0,1]^2$. Next, define
\begin{align}
    \beta_g(n) \bydef \mathcal{L}^{-1} C_g \mathcal{L}n,~~\alpha_g(n) \bydef \exp(2\pi i \beta_g(n) \cdot D_g).
\end{align}
Then, it follows that $u(x) = u(gx)$ for all $g \in D_j$ and $x \in \mathbb{R}^2$ if and only if $ \alpha_g(n) u_{\beta_g(n)} = u_n$. In this case, we say that $u$ has a $D_j$-Fourier series representation.
\end{lemma}
\begin{proof}
The proof can be found in \cite{JB_symmetries_2}. 
\end{proof}
We now define
\begin{align}
    \parallelogram_0 \bydef \mathcal{L}^{-T}(-d,d)^2 = \left\{(x_1,x_2) ~ \text{such} ~ \text{that} ~\frac{x_1}{\sqrt{3}} - \frac{2}{\sqrt{3}} d \leq x_2 \leq \frac{x_1}{\sqrt{3}} + \frac{2}{\sqrt{3}} d, ~ -d \leq x_1 \leq d\right\}.\label{def : parallelogram0}
\end{align}
$\parallelogram_0$ will be the domain of periodicity. Lemma \ref{lem : G-fourier series} provides a correspondence between symmetry of functions and symmetry of the coefficients. This lemma also has an immediate corollary.
\begin{corollary}\label{corr : reduced_set}
Let $u$ be a $D_j$-Fourier series where $D_j$ is the dihedral group of order $2j$. 
Let $\mathrm{orb}_{D_j}(n)$ be the orbit of $n \in \mathbb{Z}^2$ in $D_j$ (cf. \cite{gallianalgebra}). We recall the definition of the orbit as \begin{align}
    \mathrm{orb}_{D_j}(n) = \{g \cdot n, \ \mathrm{for \ all} \ g \in D_j\}.
\end{align}
Let $\mathcal{Z}_{\mathrm{red}}(D_j)$ be defined as in \eqref{def : reduced set}. Then, the $D_j$-Fourier series of $u$ can be written with indices in $\mathcal{Z}_{\mathrm{red}}(D_j)$. More specifically,
\begin{align}
    u(x) = \sum_{n \in \mathcal{Z}_{\mathrm{red}}(D_j)} u_n \sum_{k \in \mathrm{orb}_{D_j}(n)} e^{i \mathcal{L}\tilde{k} \cdot x}.\label{def : H_fourier series}
\end{align}
\end{corollary}
\begin{proof}
The proof can be found in \cite{JB_symmetries_2}.
\end{proof}
Essentially, $\mathcal{Z}_{\mathrm{red}}(D_j)$ contains all coefficients necessary to compute a $D_j$-Fourier series. In the case of dihedral groups, this set is a strict subset of $\mathbb{Z}^2$. Therefore, we  restrict the indexing of $D_j$-symmetric functions to $\mathcal{Z}_{\mathrm{red}}(D_j)$
and construct the full series by symmetry if needed. Let us now discuss how we can obtain tilings of the plane with $D_3$ and $D_6$-symmetry.
 \subsection{Periodic tilings with Hexagons and Triangles}
 As we are attempting to prove periodic solutions with $D_3$ and $D_6$ symmetry, we are interested in demonstrating how to construct a periodic tiling using triangles and hexagons respectively. More specifically, when using Fourier series, one naturally obtains periodicity on the domain $\parallelogram_0$. When one enforces $D_3/D_6$-symmetry, then the function is symmetric about a triangle/hexagon. We wish to merge these two viewpoints into one where we can think of the function as being periodic on a triangle/hexagon.  That is, the $D_j$ periodic pattern can be described via some $D_j$-symmetric domain(s). Let us begin by stating the result for $D_3$.
 \begin{theorem}\label{th : triangle periodic}
Let $u$ be a periodic function with domain of periodicity $\parallelogram_0$ satisfying $D_3$-symmetry. Let $\Delta_1, \Delta_2$ be two equilateral triangular domains
\begin{align}
    &\Delta_1 \bydef \left\{(x_1,x_2) ~ \text{such} ~ \text{that} ~ -\frac{x_1}{\sqrt{3}} \leq x_2 \leq \frac{x_1}{\sqrt{3}} + \frac{4}{\sqrt{3}} d, ~ -2d \leq x_1 \leq 2d\right\}\\
    &\Delta_2 \bydef \left\{(x_1,x_2) ~ \text{such} ~ \text{that} ~ \frac{x_1}{\sqrt{3}} - \frac{4}{\sqrt{3}}d\leq x_2 \leq -\frac{x_1}{\sqrt{3}}, ~ -2d \leq x_1 \leq 2d\right\}.\label{def : Delta1and2}
\end{align}
Then, $u|_{\Delta_k}$ for $k = 1,2$ is $D_3$-symmetric about the centroid of $\Delta_k$, and the periodic pattern can be generated by translations of $\Delta_1$ and $\Delta_2$. 
 \end{theorem}
 \begin{proof}
Similar to $\parallelogram_0$, we introduce
\begin{align}
    \parallelogram_{2d} \bydef \left\{(x_1,x_2) ~ \text{such} ~ \text{that} ~\frac{x_1}{\sqrt{3}} - \frac{4}{\sqrt{3}} d \leq x_2 \leq \frac{x_1}{\sqrt{3}} + \frac{4}{\sqrt{3}} d, ~ -2d \leq x_1 \leq 2d\right\}.\label{def : parallelogram1}
\end{align}
Note that $\parallelogram_{2d} \bydef \mathcal{L}^{-T} (-2d,2d)^2$ is also a domain of periodicity, albeit not the minimal one. We will show that $\parallelogram_{2d}$ can be divided into two equilateral triangles, which are $\Delta_1$ and $\Delta_2$. To do so, we must prove that the opposite angles of $\parallelogram_{2d}$ measure $60^{\circ}$ and $120^{\circ}$. Its vertices are
\begin{align}
    &V_1 \bydef \mathcal{L}^{-T} \begin{bmatrix}
        2d \\ 2d
    \end{bmatrix} = \begin{bmatrix}
        2d \\ 2\sqrt{3}d
    \end{bmatrix},~V_2 \bydef \mathcal{L}^{-T} \begin{bmatrix}
        2d \\ -2d
    \end{bmatrix} = \begin{bmatrix}
        2d \\
        -\frac{2}{\sqrt{3}}d
    \end{bmatrix}, \\
    &V_3 \bydef \mathcal{L}^{-T} \begin{bmatrix}
        -2d \\ 2d
    \end{bmatrix} = \begin{bmatrix}
        -2d \\
        \frac{2}{\sqrt{3}} d
    \end{bmatrix},~V_4 \bydef \mathcal{L}^{-T} \begin{bmatrix}
        -2d \\ -2d
    \end{bmatrix} = \begin{bmatrix}
        -2d \\
        -2\sqrt{3} d
    \end{bmatrix}.
\end{align}
We will now find the angle at the vertex $V_1$. To do so, we will form two vectors, denoted $\overrightarrow{v}_1$ and $\overrightarrow{v}_2$ that go through this vertex.
\begin{align}
    &\overrightarrow{v}_1 \bydef V_1 - V_2 = \begin{bmatrix}
        2d - 2d \\
        2\sqrt{3} d - - \frac{2}{\sqrt{3}} d
    \end{bmatrix} = \begin{bmatrix}
        0 \\
        \frac{8}{\sqrt{3}} d
    \end{bmatrix}, ~\overrightarrow{v}_2 \bydef V_1 - V_3 = \begin{bmatrix}
        2d - -2d \\
        2\sqrt{3} d - \frac{2}{\sqrt{3}} d
    \end{bmatrix} = \begin{bmatrix}
        4d \\
        \frac{4}{\sqrt{3}} d
    \end{bmatrix}.
\end{align}
Now, notice that
\begin{align}
    \overrightarrow{v}_1 \cdot \overrightarrow{v}_2 = \frac{32}{3} d^2, ~ |\overrightarrow{v}_1| = \frac{8}{\sqrt{3}} d, ~ |\overrightarrow{v}_2| = \frac{8}{\sqrt{3}} d.
\end{align}
Hence, we obtain
\begin{align}
    \theta = \arccos\left(\frac{\frac{32}{3} d^2}{\frac{8}{\sqrt{3}} d \frac{8}{\sqrt{3}}d}\right) = \arccos\left(\frac{1}{2}\right) = 60^{\circ}.
\end{align}
This shows that one of the interior angles of $\parallelogram_{2d}$ is $60^{\circ}$, meaning the non-opposite angle is $120^{\circ}$. This means we can split this parallelogram into two equilateral triangles. 
\par Let us now show that the two equilateral triangles we obtain are $\Delta_1$ and $\Delta_2$. By our previous computation, we know that the angles at the vertices $V_2$ and $V_3$ are $120^{\circ}$ angles. Hence, the two equilateral triangles can be found by drawing a line between $V_2$ and $V_3$. This line has a slope of
\begin{align}
   \frac{-\frac{2}{\sqrt{3}}d - \frac{2}{\sqrt{3}}d}{2d - -2d} = \frac{-\frac{4}{\sqrt{3}}d}{2d + 2d} = \frac{\frac{4}{\sqrt{3}}d}{4d} = -\frac{1}{\sqrt{3}}.
\end{align}
With the slope computed, the line we want is given by
\begin{align}
    &x_2 + \frac{2}{\sqrt{3}}d = -\frac{1}{\sqrt{3}}(x_1 - 2d) \\
    &x_2 + \frac{2}{\sqrt{3}}d = -\frac{1}{\sqrt{3}} x_1 + \frac{2}{\sqrt{3}}d \\
    &x_2 = -\frac{1}{\sqrt{3}}x_1.
\end{align}
Then, according to the definition of $\parallelogram_{2d}$ given in \eqref{def : parallelogram1}, it is clear that we obtain $\Delta_1$ and $\Delta_2$. Since triangles tile the plane, and we have a periodic function, we can create a tiling using $\Delta_1$ and $\Delta_2$. 
\par Now that we have obtained that $\Delta_1$ and $\Delta_2$ can tile the plane, we need to show that $u$ is $D_3$-symmetric on $\Delta_k$ for $k = 1,2$. The centroids of $\Delta_1$ and $\Delta_2$, $c_1$ and $c_2$ respectively, are
\begin{align}
    c_1 \bydef \left(\frac{2d}{3},\frac{2d}{\sqrt{3}}\right), ~ c_2 \bydef \left(-\frac{2d}{3},-\frac{2d}{\sqrt{3}}\right).
\end{align}
To proceed generally, we introduce $c \bydef \left(\pm \frac{2d}{3}, \pm \frac{2d}{\sqrt{3}}\right)$. We must now show that $u$ is invariant under $D_3$-symmetry operations done about $c$. Let $g\cdot x = \mathcal{A}x$ for some $g \in D_3$ and $\mathcal{A} \in M_{2\times2}(\mathbb{R})$. Then,  
\begin{align}
    u(\mathcal{A} (x - c) + c) = u(\mathcal{A}(x - c + \mathcal{A}^{-1} c)) = u(\mathcal{A}(x + (\mathcal{A}^{-1} - I_d) c)) = u(x + (\mathcal{A}^{-1} - I_d) c)
\end{align}
where the last step followed from the fact that $u$ is $D_3$-symmetric. If $\mathcal{A} = \begin{bmatrix}
    1 & 0 \\
    0 & -1
\end{bmatrix}$, then 
\begin{align}
     u(x + (\mathcal{A}^{-1} - I_d)c) = u\left(x \mp \begin{bmatrix}
         0 \\ \frac{4d}{\sqrt{3}}
     \end{bmatrix}\right) = u(x)
\end{align}
where the last step followed from the fact that $\mathcal{L}^T \begin{bmatrix}
    0 \\ \mp \frac{4d}{\sqrt{3}}
\end{bmatrix} = \begin{bmatrix} 0 \\ \mp 2d \end{bmatrix}$. Hence, we can disregard it by periodicity of $\parallelogram_0$. Note that we require the periodicity on $\parallelogram_0$ here despite using $\parallelogram_{2d}$ to define $\Delta_k$ for $k = 1,2$. Let $R_{\theta}$ be the rotation matrix by $\theta$. If $\mathcal{A} = R_{\frac{2\pi}{3}}$ is a rotation, then
\begin{align}
    u(x + (\mathcal{A}^{-1} - I_d)c) = u\left(x + \begin{bmatrix}
        0 \\ \mp \frac{4\sqrt{3}d}{3}
    \end{bmatrix}\right) = u(x)
\end{align}
which we can again disregard by periodicity on $\parallelogram_0$. Showing that the result holds for other group elements of $D_3$ will use similar steps. This concludes the proof.
 \end{proof}
 We present the analogous result for $D_6$.
  \begin{theorem}\label{th : hexagon periodic}
Let $u$ be a periodic function satisfying $D_6$-symmetry. Let 
{\footnotesize\begin{align}
    &\parallelogram_1 \bydef \left\{ \left(x_1-d,x_2+\frac{d}{\sqrt{3}}\right) ~ | ~ (x_1,x_2) \in \parallelogram_0\right\},~\parallelogram_2 \bydef \left\{R_{\frac{2\pi}{3}} x ~ \text{for} ~ \text{all} ~ x \in \parallelogram_1\right\},~\parallelogram_3 \bydef \left\{ R_{\frac{4\pi}{3}} x ~ \text{for} ~ \text{all} ~ x \in \parallelogram_1\right\}
\end{align}}
where $R_{\theta}$ is the rotation matrix by $\theta$. Let $\varhexagon_0 \bydef \parallelogram_1 \cup \parallelogram_2 \cup \parallelogram_3$. Then, $u$ satisfies periodic boundary conditions on $\varhexagon_0$.
 \end{theorem}
 \begin{proof}
We first shift $\parallelogram_0$ to $\parallelogram_1$, which is still a domain of periodicity by definition. Now, since $u$ is $D_6$-symmetric, we are able to perform both a $120^{\circ}$ and $240^{\circ}$ rotation about $(0,0)$ and preserve the function values. That is, $u|_{\parallelogram_k}$ for $k=1,2,3$ are equivalent under $D_6$-symmetry operations.
Then, if we define $\varhexagon_0 \bydef \parallelogram_1 \cup \parallelogram_2 \cup \parallelogram_3$, we will obtain a regular hexagonal domain. Moreover, since $u$ is $D_6$-symmetric, it follows that $u|{\varhexagon_0}$ is invariant under $D_6$-symmetry operations. Since hexagons tile the plane, we can generate the periodic pattern of $u$ using copies of $\varhexagon_0$. Hence, we have proven that $u$ is periodic on $\varhexagon_0$
 \end{proof}
 \begin{remark}
 The previous two theorems both provide results on tiling the plane using triangles and hexagons; however, the result for $D_3$ differs from that of $D_6$. In the case of $D_3$, we have that $u|_{\Delta_1} \neq u|_{\Delta_2}$. This is due to the fact that $D_3$-symmetry does not contain the needed element to relate $\Delta_1$ and $\Delta_2$. In other words, $\Delta_1$ and $\Delta_2$ generate distinct sublattices of the hexagonal lattice which, when repeated periodically by translation, can generate the $D_3$ pattern. On the other hand, in the case of $D_6$, we only need to consider $\varhexagon_0$ where $u|_{\varhexagon_0}$ does respect all the $D_6$-symmetry operations. As a result, $\varhexagon_0$ is a domain of periodicity on its own, providing us with a periodic function on a hexagon. To summarize, for $D_3$, we define two distinct equilateral triangles, each of which respect the $D_3$-symmetry. When tiled together, these two triangles can generate the full $D_3$ periodic pattern. For $D_6$, we define one hexagon which respects the $D_6$-symmetry. This hexagon on its own can be considered a domain of periodicity, meaning we have a periodic function on a hexagon. In fact, when the symmetry is $D_6$, one can show that $u|_{\Delta_1} = u_{\Delta_2}$, which is another way to generate the hexagonal tiling. \end{remark} 
 Now that we have obtained the desired results for tiling the plane using triangles and hexagons when the function has $D_3$ or $D_6$-symmetry respectively, let us begin our computer assisted analysis.
 \section{Computer Assisted Analysis}\label{sec : CAP}
 In this section, we will discuss the computer assisted aspects of our approach. From now on, we will refer to $u$ as a sequence rather than a function unless other specified. First, we introduce some notation. Let $\ell^1_{j,\nu}$ denote the following Banach space
{\small\begin{align}\label{ell1Gnu}   \ell^1_{j,\nu} \bydef \left\{u = (u_n)_{n \in \mathcal{Z}_{\mathrm{red}}(D_j)}: ~ \|u\|_{1,\nu} \bydef  \sum_{n \in \mathcal{Z}_{\mathrm{red}}(D_j)} \alpha_n |u_n|\nu^n < \infty \right\}, \text{ where }(\alpha_n)_{n \in \mathcal{Z}_{\mathrm{red}}(D_j)} \bydef |\mathrm{orb}_{D_j}(n)|.
\end{align}}
Note that a sequence in  $\ell^1_{j,\nu}$ is also in the usual $\ell^1_{\nu}$. More specifically, $\ell^1_{j,\nu}$ contains the sequences indexed on $\mathcal{Z}_{\mathrm{red}}(D_j)$. Given Fourier coefficients $u = (u_n)_{n \in \mathbb{Z}^2},  v= (v_n)_{n \in \mathbb{Z}^2}$ corresponding to Fourier series of the form \eqref{usual_fourier_series}, we define the discrete convolution as 
\begin{align}\label{def : conv}
(u*v)_n = \sum_{k \in \mathbb{Z}^2} u_{n-k}v_k.
\end{align}
In the case of symmetric sequences $\ell^1_{j,\nu}$, note that by using the relations determined by $\mathcal{Z}_{\mathrm{red}}(D_j)$, one can unfold a sequence of the form \eqref{def : H_fourier series} into one of the form \eqref{usual_fourier_series}. After performing this unfolding, then one can use \eqref{def : conv} to compute the discrete convolution. Since the result is the same, we still denote $u*v$ the discrete convolution representing the product of two functions for consistency. Additionally, given $u \in \ell^1_{j,\nu}$, we define
\begin{align}\label{def : discrete conv operator}
    \mathbb{u} : \ell^1_{j,\nu} &\to \ell^1_{j,\nu} \\
    v &\mapsto  u * v
\end{align}
 the discrete convolution operator associated to $u$. To begin, we will state the main theorem we wish to apply. Following this, we will construct various quantities needed to apply this theorem. Then, we will estimate the needed bounds which we will then rigorously evaluate on the computer. Let us begin with the main theorem.
 \subsection{Newton-Kantorovich Approach}
 As we are considering the 2D Swift Hohenberg PDE, we write
 \begin{align}
    f(u) = Lu + G(u),~L \bydef (I_d + \nabla^2)^2 + \mu I_d, ~~ G(u) \bydef -\gamma u^2 + u^3.\label{def : L and G sh 2d}
\end{align}
Let $S$ be a Banach space such that $f : \ell^1_{j,\nu} \to S$. Suppose we have $\overline{u} \in \ell^1_{j,\nu}$ an approximate solution to \eqref{def : L and G sh 2d}. That is, $f(\overline{u}) \approx 0$. Then, we state the following theorem.
 \begin{theorem}\label{th : radii polynomial theorem}
Let $A^\dagger \in \mathcal{B}(\ell^1_{j,\nu},S)$. Also let $A \in \mathcal{B}(S,\ell^1_{j,\nu})$ be injective. Moreover, let $Y_0, Z_0, Z_1$ be non-negative constants, and let $Z_2(r) : (0,\infty) \to [0,\infty)$ be a non-negative function such that
\begin{align}
    &\|Af(\overline{u})\|_{1,\nu} \leq Y_0 \\
    &\|I_d - AA^\dagger\|_{\mathcal{B}(\ell^1_{j,\nu})} \leq Z_0 \\
    &\|A(Df(\overline{u}) - A^\dagger)\|_{\mathcal{B}(\ell^1_{j,\nu})} \leq Z_1 \\
    &\|A(Df(u) - Df(\overline{u}))\|_{\mathcal{B}(\ell^1_{j,\nu})} \leq Z_2(r) r ~ \text{for} ~ \text{all} ~ u \in B_r(\overline{u}).
\end{align}
If there exists $r>0$ such that
\begin{equation}\label{condition radii polynomial}
    \frac{1}{2}Z_2(r)r^2 - (1-Z_0 - Z_1)r + Y_0 <0, \ and \ Z_1 + Z_0 + Z_2(r)r < 1
 \end{equation}
then there exists a unique $\tilde{u} \in \overline{B_r(\overline{u})} \subset \ell^1_{j,\nu}$ such that $f(\tilde{u})=0$, where $B_r(\overline{u})$ is the open ball of $\ell^1_{j,\nu}$ centered at $\overline{u}$ with radius $r$.
 \end{theorem}
 \begin{proof}
 The proof can be found in \cite{van2021spontaneous}.
 \end{proof}
 \subsection{Numerical Aspects}\label{sec : numerical aspects}
In this section, we will discuss the numerical aspects of our approach. We first introduce some notation. Define
\begin{align}
    I^N \bydef \{n \in \mathcal{Z}_{\mathrm{red}}(D_j), |n_1| \leq N, |n_2| \leq N\}.\label{def : IN}
\end{align}
Let $v = (v_n)_{n \in \mathcal{Z}_{\mathrm{red}}(D_j)}$
\begin{align}
    (\pi^N v)_n \bydef \begin{cases}
        v_n & n \in I^N \\
        0 & \mathrm{else}
    \end{cases},~~ (\pi_N v)_n \bydef \begin{cases}
        0 & n \in I^N \\
        v_n & \mathrm{else}
    \end{cases}.\label{def : projection operators}
\end{align}
To obtain a numerical approximation of the solution on $\parallelogram_0$, we relied on a naive approach. We would first choose the number of Fourier coefficients $N$. Next, we randomly generated a grid of points of the corresponding size, $2N+1 \times 2N+1$. We think of this random grid as representing a random function, denoted $\overline{U}_2$. We then computed a Fourier series approximation of the function $\overline{U}_2$. That is, we computed $(\overline{a}_n)_{n \in \{-N,\dots,N\}^2}$ such that
\begin{align}
    \overline{U}_1(x) \bydef \sum_{n \in \{-N,\dots,N\}^2} \overline{a}_n e^{ i \mathcal{L}\tilde{n} \cdot x}\label{u_1bar}
\end{align}
and we should have $\overline{U}_2 \approx \overline{U}_1$.
At this step, we do not expect the coefficients $\overline{a}_n$ to satisfy the symmetry relation for $D_j$. Indeed, since $\overline{U}_2$ was a randomly chosen function, it has no reason to be $D_j$-symmetric and hence neither does its Fourier series approximation $\overline{U}_1$. As we are only interested in $D_j$-symmetric solutions for this paper, we would simply choose one element from each orbit, and defined
\begin{align}
    \overline{U}_0(x) \bydef \sum_{n \in I^N} \overline{a}_n \sum_{k \in \mathrm{orb}_{D_j}(n)} e^{ i \mathcal{L}\tilde{k} \cdot x}.\label{u_0bar}
\end{align}
Note that the choice of element from each orbit is not important to us as we are building a random initial guess. From here, we run a Newton method on $\overline{U}_0$ to refine our initial guess. More often than not, Newton's method would diverge as our guess was random and not particularly close to a solution. After a sufficient amount of attempts, our Newton method would converge. We denote the result of the convergence $\overline{u} \bydef (\overline{u}_n)_{n \in I^N}$. We then theoretically set all remaining coefficients to $0$, and write $\overline{u} \bydef (\overline{u}_n)_{n \in \mathcal{Z}_{\mathrm{red}}(D_j)}$. We use this as our approximate solution for the proof. Our approach is not guaranteed to work for any PDE and solution; however, it was sufficient for our purposes as we managed to find multiple candidates where Newton converged. 
\par Let us now construct $A$ and $A^\dagger$. These choices will consist of two parts: a finite part which can be represented as a matrix, and an infinite part which we will control. $A^\dagger$ will be an approximation of $Df(\overline{u})$. That is, for a given $h \in \ell^1_{j,\nu}$, we define
\begin{align}
    (A^\dagger h)_n \bydef \begin{cases}
        [\pi^N Df(\overline{u})\pi^N h]_n & n \in I^N \\
        ((1 + |\mathcal{L}\tilde{n}|^2)^2 + \mu) h_n & n \in \mathcal{Z}_{\mathrm{red}}(D_j) \setminus I^N
    \end{cases}.\label{def : Adagger}
\end{align}
Notice that for $n \in I^N$, we consider the action of $\pi^N Df(\overline{u})\pi^N$ on $h$. This operator can be represented as a finite matrix. We then only consider the action of the linear part in the tail. This leads us to a natural way to define $A$, which is supposed to be an approximate inverse for $Df(\overline{u})$. This is chosen so that the bound $Z_1$ in Theorem \ref{th : radii polynomial theorem} is small. In particular, we choose
\begin{align}
    A^N \approx (\pi^N Df(\overline{u})\pi^N)^{-1}.
\end{align}
As mentioned before, the operator $A^N$ can be represented a matrix, which can be stored on the computer. This is the finite part of $A$. For the tail, rather than considering the full $Df(\overline{u})^{-1}$, we only choose the inverse of the linear part, $L$. This leads us to define
\begin{align}
    (Ah)_n \bydef \begin{cases}
        (A^N \pi^N h)_{n} & n \in I^N \\
        \frac{h_n}{(1 + |\mathcal{L}\tilde{n}|^2)^2 + \mu} & n \in \mathcal{Z}_{\mathrm{red}}(D_j) \setminus I^N
    \end{cases}.
\end{align}
With this definition, observe that
\begin{align}
    \|\pi_N A\|_{\mathcal{B}(\ell^1_{j,\nu})} &\leq \max_{n \in \mathbb{Z}^2 \setminus I^N} \frac{1}{|(1 + |\mathcal{L}\tilde{n}|^2)^2 + \mu|} \bydef \frac{1}{L_N}.
\end{align}
where 
\begin{align}
    L_N \bydef \min_{n \in \mathbb{Z}^2 \setminus I^N} |(1 + |\mathcal{L}\tilde{n}|^2)^2 + \mu|\label{def : L_N}
\end{align}
We will frequently use this bound in our computations to follow. With $\overline{u}, A^\dagger$, and $A$ now detailed in construction, let us describe how we enforce the $D_3$ and $D_6$ symmetry in these objects. 
 \subsubsection{Enforcing \texorpdfstring{$D_3$}{D3} and \texorpdfstring{$D_6$}{D6} symmetries in Fourier series}\label{sec : dihedral.jl}
Our approach relies on a the ability to construct functions of the form \eqref{def : H_fourier series} and rigorous numerics. To do so, we rely primarily on two Julia packages. The first is IntervalArithmetic.jl, \cite{julia_interval}. This package allows us to perform computations on intervals, which is necessary to obtain rigorous results. The second is RadiiPolynomial.jl, \cite{julia_olivier}. This package is designed with many useful tools for computer assisted proofs. One of its main tools is the ability to store vectors as "sequence structures." By a sequence structure, we mean a vector that corresponds to a sequence of Fourier coefficients. By storing our data in its corresponding sequence structure, we are able to rely on a variety of operations built into \cite{julia_olivier}. In particular, this package allows us to perform addition, subtraction, convolutions, build multiplication operators, build linear operators, compute norms, take derivatives, etc. One can do this provided that the sequence structure with the desired symmetry they wish to prove exists in \cite{julia_olivier}. This is not the case for this project. In fact, \cite{julia_olivier} only allows one to enforce up to $D_2$ symmetry. For dihedral groups, this limitation was partially addressed by one of the authors of \cite{gs_cadiot_blanco}. Indeed, one of the authors developed a code for $D_4$-symmetric sequences. This sequence structure was then used to perform the proofs of Theorems 6.1, 6.2, 6.3, and 6.4 in Section 6 of \cite{gs_cadiot_blanco}. The code is available at \cite{dominic_D_4_julia} and is fully compatible with \cite{julia_olivier}. 
\par Since we require sequence structures for $D_3$ and $D_6$ symmetry, we extended the work done in \cite{dominic_D_4_julia} to include $D_3$ and $D_6$ sequence structures. The result, which we will call dihedral.jl can be found on Github at \cite{dominic_dihedral_julia}. This code can be viewed as an add-on to \cite{julia_olivier}. More specifically, the sequence structures implemented in \cite{dominic_dihedral_julia} are fully compatible with the tools of \cite{julia_olivier}. Our contribution was to create the sequence structures for the dihedral symmetries and define the operations necessary to perform operations such as addition, subtraction, convolution, build multiplication operators, build linear operators, compute norms, etc. One of the challenges we would like to emphasize is that we needed to find a way to store the reduced sets, $\mathcal{Z}_{\mathrm{red}}(D_j)$ for $j = 3,6$. As mentioned earlier, a sequence structure is a vector whose entries represent the indices of a sequence with a specific symmetry. Therefore, the first challenge is to find a mapping between the entries of a vector and the indices of a sequence. Mathematically speaking, we are required to find a mapping from $\mathcal{Z}_{\mathrm{red}}(D_j) \to \mathbb{N}$. In general, this task is non-trivial. By studying $\mathcal{Z}_{\mathrm{red}}(D_j)$ for $j = 3,6$, we were eventually able to find a mappings which worked. We describe the details on these mapping and how they are defined. Let 
\begin{align}
    \mathbb{N}_0^{\mathcal{K}} \bydef \{0,1,\dots,\mathcal{K}\}.
\end{align}
That is $\mathbb{N}_0^{\mathcal{K}}$ is the set of numbers from $0$ to $\mathcal{K}$. Suppose we have a sequence of size $\mathcal{N}$. That is, the indices live in $I^{\mathcal{N}}$. We must now find a mapping $\psi_j : I^{\mathcal{N}} \to \mathbb{N}_0^{|I^{\mathcal{N}}|}$ for $j = 3,6$ so that we can identify the indices of a sequence with the entries of a vector. For $D_3$, the mapping is
\begin{align}
    \psi_3 : I^{\mathcal{N}} &\to \mathbb{N}_0^{|I^{\mathcal{N}}|} \\
    (n_1,n_2) &\mapsto n_1 + n_2 \left(\frac{\mathcal{N}}{2} + 1\right) - n_2 + 1.
\end{align}
For $D_6$, it is
\begin{align}
    \psi_6 : I^{\mathcal{N}} &\to \mathbb{N}_0^{|I^{\mathcal{N}}|} \\
    (n_1,n_2) &\mapsto n_1 + n_2\frac{\mathcal{N}}{2} - \frac{(n_2 - 1)^2 + 3(n_2 - 1)}{2}.
\end{align}
The details can be found in \cite{dominic_dihedral_julia}.
Additionally, the sets $I^{\mathcal{N}}$ for $D_3$ and $D_6$ do not completely fill the indices of $\{-\mathcal{N},\dots,\mathcal{N}\}^2$. For example, $(5,1) \in \mathrm{orb}_{D_6}((4,-1))$. Therefore, in order to include $(4,-1)$, we must also include $(5,1)$. That means that $(4,-1) \notin I^4$, and one must set the index $(4,-1)$ to $0$. The reason is due to the definition of $I^{\mathcal{N}}$ in \ref{def : IN}, which states that we must only includes indices whose entries are less than $\mathcal{N}$ in absolute value. This leads to some numerical complications. More specifically, it turns out that the orbits of the entries of $I^{\mathcal{N}}$ for $D_3$ and $D_6$ only contain every index for $|n_1|, |n_2| \leq \frac{\mathcal{N}}{2}$. This means that a $D_3$ or $D_6$ Fourier series of size $\mathcal{N}$ can be viewed as a Fourier series of order $\frac{\mathcal{N}}{2}$ plus a few more indices up to order $\mathcal{N}$. To formalize this, an index $n \in \{-\mathcal{N},\dots,\mathcal{N}\}^2$ is considered a valid index for $D_3$ if 
\begin{align}
    n_2 \in \left(0,\min\left(n_1,\frac{\mathcal{N}}{2}\right)\right)~ \text{and} ~ n_1 - n_2 \in \left(-\frac{\mathcal{N}}{2},\frac{\mathcal{N}}{2}\right).
\end{align}
For $D_6$, an index is valid if
\begin{align}
    n_2 \in \left(0,\frac{\mathcal{N}}{2}\right) ~ \text{and} ~ n_1 \in \left(2n_2,\frac{\mathcal{N}}{2} + n_2\right).
\end{align}
In practice, this makes proofs which use these symmetries require additional runtime than those done with symmetries expressed on the square lattice. The primary culprit of the slow runtime is convolution since those are performed by unfolding the sequence. When one unfolds this sequence to perform convolution, it will numerically behave as though it is a sequence of size $\mathcal{N}$ despite the number of $0$ indices in this set. This does not occur on the square lattice and $D_4$ for instance since the set $\{-\mathcal{N},\dots,\mathcal{N}\}^2$ will not have automatic zeros by symmetry. As a result, every step of the convolution contributes. The benefits still outweigh this problem as we can not only enforce the symmetry, but the proofs take less memory since matrices are smaller with less coefficients.
\par With the sequence structures at \cite{dominic_dihedral_julia} available, we can indeed restrict the indices to $\mathcal{Z}_{\mathrm{red}}(D_j)$ and define $\overline{u}_0$ as in \eqref{u_0bar}. Once we run Newton on $\overline{u}_0$ to obtain $\overline{u}$, we can evaluate $\pi^N Df(\overline{u})\pi^N$ on the computer. Once we have this matrix available, we can invert it, and obtain our approximate inverse $A^N$ under the $D_j$-symmetry. With these objects now constructed numerically, we are ready to compute the bounds of Theorem \ref{th : radii polynomial theorem}. This will be the focus of the next section.
\subsection{Computing the Bounds}\label{sec : bounds}
We will now compute the bounds $Y_0, Z_0, Z_1,$ and $Z_2$. Recall \eqref{def : projection operators} as they will be frequently used throughout this section. Let us begin with $Y_0$ and $Z_0$.
\begin{lemma}\label{lem : Y0 and Z0}
Let $L_{N}$ be defined as in \eqref{def : L_N}. Let $Y_0, Z_0 > 0$ be defined as 
\begin{align}
    Y_0 \bydef \|A^N f(\overline{u})\|_{1,\nu} + \frac{1}{L_N} \|(\pi^{3N} - \pi^N) G(\overline{u})\|_{1,\nu},~Z_0 \bydef \|\pi^N - A^N \pi^N Df(\overline{u})\pi^N\|_{\mathcal{B}(\ell^1_{j,\nu})}.
\end{align}
Then, it follows that $\|Af(\overline{u})\|_{1,\nu} \leq Y_0$ and $\|I_d - AA^{\dagger}\|_{\mathcal{B}(\ell^1_{j,\nu})} = Z_0$.
\end{lemma}
\begin{proof}
For $Y_0$, we start from the definition.
\begin{align}
    \|Af(\overline{u})\|_{1,\nu} = \|\pi^N Af(\overline{u})\|_{1,\nu} + \|\pi_N A f(\overline{u})\|_{1,\nu}
    &= \|A^N f(\overline{u})\|_{1,\nu} + \|\pi_N A G(\overline{u})\|_{1,\nu} \\
    &\leq \|A^N f(\overline{u})\|_{1,\nu} + \|\pi_N A\|_{\mathcal{B}(\ell^1_{j,\nu})} \|\pi_NG(\overline{u}))\|_{1,\nu} \\
    &\leq \|A^N f(\overline{u})\|_{1,\nu} + \frac{1}{L_N} \|(\pi^{3N} - \pi^N) G(\overline{u})\|_{1,\nu}
\end{align}
as desired. For $Z_0$, notice that since the tail of $A$ and $A^\dagger$ cancel exactly, we are left with the finite part. Hence, the equality is immediate.
\end{proof}
Next, we examine the $Z_2$ bound. We first define the sequence $\delta = (\delta_n)_{n \in \mathcal{Z}_{\mathrm{red}}(D_j)}$ as
\begin{align}\label{kronecker delta}
    \delta_n \bydef \begin{cases}
        1 & n = 0 \\
        0 & \mathrm{else}
    \end{cases}.
\end{align}
We state its lemma.
\begin{lemma}\label{lem : Z2}
Let $u \in B_r(\overline{u})$. Let $q \bydef (q_n)_{n \in \mathcal{Z}_{\mathrm{red}}(D_j)} = (-2\gamma \delta_n + 6\overline{u}_n)_{n \in \mathcal{Z}_{\mathrm{red}}(D_j)}$ where $\overline{u} = (\overline{u}_n)_{n \in \mathcal{Z}_{\mathrm{red}}(D_j)}$ and $\delta_n$ is defined as in \eqref{kronecker delta}. Let $L_N$ be defined as in \eqref{def : L_N}. Now, let $Z_2(r) : (0,\infty) \to [0,\infty)$ be defined as
\begin{align}
    Z_2(r) \bydef  \max\left(\|A^N\|_{\mathcal{B}(\ell^1_{j,\nu})}, \frac{1}{L_N}\right)(\|q\|_{1,\nu} + 3r).
\end{align}
Then, $\|A(Df(u) - Df(\overline{u}))\|_{\mathcal{B}(\ell^1_{j,\nu})} \leq Z_2(r)r$.
\end{lemma}
\begin{proof}
To begin, observe that
\begin{align}
    \|A(Df(u) - Df(\overline{u}))\|_{\mathcal{B}(\ell^1_{j,\nu})} = \|A(DG(u) - DG(\overline{u}))\|_{\mathcal{B}(\ell^1_{j,\nu})}
    &\leq \|A\|_{\mathcal{B}(\ell^1_{j,\nu})} \|DG(u) - DG(\overline{u})\|_{\mathcal{B}(\ell^1_{j,\nu})} \\
    &\hspace{-2.3cm}\leq \max\left(\|A^N\|_{\mathcal{B}(\ell^1_{j,\nu})}, \frac{1}{L_N}\right) \|DG(u) - DG(\overline{u})\|_{\mathcal{B}(\ell^1_{j,\nu})}.\label{pulling_out_A_Z2}
\end{align}
Now, since $u \in B_r(\overline{u})$, there exists a $v \in B_r(0)$ such that $u = \overline{u} + v$. Note that $\|v\|_{1,\nu} \leq r$. Also let $h \in \ell^1_{j,\nu}$ such that $\|h\|_{1,\nu} \leq 1$. Then, we compute
\begin{align}
    (DG(u) - DG(\overline{u}))h = (DG(\overline{u} + v) - DG(\overline{u}))h
    &= (-2\gamma (\overline{u} + v) + 3(\overline{u} + v)^2 + 2\gamma \overline{u} - 3\overline{u}^2)h \\
    &= ((-2\gamma I_d + 6\overline{u})v + 3v^2)h \\
    &= (qv + 3v^2)h.
\end{align}
Using this, we obtain
\begin{align}
    \|DG(u) - DG(\overline{u})\|_{\mathcal{B}(\ell^1_{j,\nu})} \leq \|qvh + 3v^2h\|_{1,\nu}
    &\leq \|q\|_{1,\nu} \|v\|_{1,\nu} \|h\|_{1,\nu} + 3\|v\|_{1,\nu}^2 \|h\|_{1,\nu} \leq \|q\|_{1,\nu} r + 3r^2\label{before banach algebra}
\end{align}
where we used that $\ell^1_{j,\nu}$ (see \ref{ell1Gnu}) is a Banach algebra. Using this in \eqref{pulling_out_A_Z2} yields the result.
\end{proof}
\begin{remark}
The computation of $Z_2$ was done primarily in the interest of computational intensity. Indeed, step \eqref{pulling_out_A_Z2} is not a particularly sharp estimate. For instance, one could keep $A$ inside rather than using submultiplicativity immediately such as what was done in \cite{sh_cadiot}. We did so in order to simplify the estimates done on the computer and increase the speed of the code. This is of higher interest in Section \ref{sec : continuation}, where the runtime becomes an issue. 
\end{remark}
Before we compute the $Z_1$ bound, we define $\overline{v} \bydef (\overline{v}_n)_{n \in \mathcal{Z}_{\mathrm{red}}(D_j)} = (-2\gamma \overline{u}_n + 3(\overline{u}*\overline{u})_n)_{n \in \mathcal{Z}_{\mathrm{red}}(D_j)}$. Notice that by definition, we have $\overline{v} = \pi^{2N} \overline{v}$. We will now state the lemma for $Z_1$.
\begin{lemma}\label{lem : Z1}
Recall that $\overline{v} \bydef \pi^{2N} \overline{v}$. Let $\overline{V} = (\overline{V}_n)_{n \in \mathcal{Z}_{\mathrm{red}}(D_j)},\phi = (\phi_n)_{n \in \mathcal{Z}_{\mathrm{red}}(D_j)} \in \ell^1_{j,\nu}$ where
\begin{align}
    \overline{V}_n \bydef \begin{cases}
        0 & n = (0,0) \\
        \overline{v}_n & \mathrm{else}
    \end{cases}, ~~\phi_n \bydef \begin{cases} \frac{\|\overline{V}\|_{\infty}}{\nu^{N+1}} & n \in I^N \\
    0 & \mathrm{else} 
    \end{cases}.
\end{align}
Let $Z_1 > 0$ be defined as
\begin{align}
    &Z_1 \bydef  \|A^N \phi\|_{1,\nu} + \frac{1}{L_N} \|\overline{v}\|_{1,\nu} .
\end{align}
Then, it follows that $\|A(Df(\overline{u}) - A^\dagger)\|_{\mathcal{B}(\ell^1_{j,\nu})} \leq Z_1$.
\end{lemma}
\begin{proof}
To begin, given $h \in \ell^1_{j,\nu}, \|h\|_{1,\nu} \leq 1$, let
\begin{align}
    z \bydef (Df(\overline{u}) - A^\dagger) h.
\end{align}
Then, observe that we can write
\begin{align}
    \|A(Df(\overline{u}) - A^\dagger)\|_{\mathcal{B}(\ell^1_{j,\nu})} = \|A(Df(\overline{u}) - A^\dagger)h\|_{1,\nu} = \|Az\|_{1,\nu} = \|A^N \pi^N z\|_{1,\nu} + \|\pi_N A z\|_{1,\nu}.\label{split Z1}
\end{align}
We now have two terms in \eqref{split Z1} to bound. For the first term, notice that for $n \in I^N$ we can write
\begin{align}
    (\pi^N z)_n = \sum_{m \in \mathbb{Z}^2} \overline{v}_{n-m} h_n - \sum_{m \in I^N} \overline{v}_{n-m} h_{m} = \sum_{m \in \mathbb{Z}^2 \setminus I^N} \overline{v}_{n-m} h_{m} = (\overline{v} *\pi_N h)_n
\end{align}
Now, we claim that $|(\overline{v}* \pi_N h)_n| \leq \phi_n$ for all $n \in I^N$. First, let $W_1(n) \subset \mathbb{Z}^2$ be the set of $m$ for which $\overline{v}_{n - m} \neq 0$. Also let $W_2 \bydef \mathbb{Z}^2 \setminus I^N,$ which is the set of $m$ for which $\pi_N h \neq 0$. Then, observe that
\begin{align}
    |(\overline{v} *\pi_N h)_n| \leq \left| \sum_{m \in \mathbb{Z}^2} \overline{v}_{n -m} (\pi_N h)_m\right| &\leq \sum_{m \in W_1(n) \cap W_2} \left(\frac{|\overline{v}_{n - m}|}{\nu^{|m|}}\right) |h_m| \nu^{|m|} \\
    &\leq \sum_{m \in W_1(n) \cap W_2} \left(\sup_{k \in W_1(n) \cap W_2} \frac{|\overline{v}_{n-k}|}{\nu^{|k|}}\right) |h_m| \nu^{|m|} \\
    &\leq \sup_{m \in W_1(n) \cap W_2} \frac{|\overline{v}_{n - m}|}{\nu^{|m|}} \|h\|_{1,\nu}.\label{gabriel_stopping_point}
\end{align}
Since $v_{k} = 0$ whenever $k \in \mathbb{Z}^2 \setminus I^{2N}$ and we only need $n \in I^N$, we obtain the overestimate that $W_1(n) \subset I^{3N}$. Additionally, since $(\pi_N h)_n = 0$ for all $n \in I^N$, we get
\begin{align}
    |(\overline{v} * \pi_N h)_n| \leq \max_{m \in I^{3N} \setminus I^N} \frac{|\overline{v}_{n-m}|}{\nu^{|m|}} \leq \frac{1}{\nu^{N+1}} \max_{m \in I^{3N} \setminus I^N} |\overline{v}_{n-m}| \leq \frac{1}{\nu^{N+1}} \max_{k \in I^N, m \in I^{3N} \setminus I^N} |\overline{v}_{k-m}|.\label{maximum_step}
\end{align}
Now, since $I^N$ and $I^{3N} \setminus I^N$ are disjoint, the index of $(0,0)$ is never attained when computing the maximum in \eqref{maximum_step}. As a result, we can write
\begin{align}
    \max_{k \in I^N, m \in I^{3N} \setminus I^N} |\overline{v}_{k-m}| = \max_{k \in I^N, m \in I^{3N} \setminus I^N} |\overline{V}_{k-m}| = \|\overline{V}\|_{\infty}.
\end{align}
Hence, we indeed see that $|(\overline{v} * \pi_N h)_n| \leq \phi_n$ for all $n \in I^N$. This means we can say
\begin{align}
    \|A^N \pi^Nz\|_{1,\nu} = \|A^N(\overline{v}\pi_N h)\|_{1,\nu} \leq \|A^N \phi\|_{1,\nu}.\label{finite part Z1}
\end{align}
Finally, we examine the second term of \eqref{split Z1}. 
\begin{align}
    \|\pi_N A z\|_{1,\nu} \leq \|\pi_N A \|_{\mathcal{B}(\ell^1_{j,\nu})} \|\pi_n z\|_{1,\nu} \leq \frac{1}{L_N} \|z\|_{1,\nu} \leq \frac{1}{L_N} \|\overline{v}\|_{1,\nu} \|h\|_{1,\nu} \leq \frac{1}{L_N} \|\overline{v}\|_{1,\nu}\label{infinite part Z1}
\end{align}
where we used the Banach algebra property and the fact that $\|h\|_{1,\nu} \leq 1$. Therefore, we can combine \eqref{finite part Z1} and \eqref{infinite part Z1} to get
\begin{align}
    \|A(Df(\overline{u}) - A^\dagger)\|_{\mathcal{B}(\ell^1_{j,\nu})} \leq \|A^N \phi\|_{1,\nu} + \frac{1}{L_N} \|\overline{v}\|_{1,\nu}  \bydef Z_1
\end{align}
as desired.
\end{proof}
\begin{remark}
The steps to estimate \eqref{finite part Z1} were similar to those used by the authors of \cite{gabriel_pfc}. Indeed, if we had stopped at \eqref{gabriel_stopping_point}, we would have had the equivalent result for our problem. We decided to go further and introduce $\overline{V}$ so we could obtain a uniform estimate for all $n$. This will be of interest in Section \ref{sec : continuation} where the uniform estimate is less computationally intensive for the estimates done on the computer. If we had not introduced $\overline{V}$, we could have obtained an estimate similar to that done in \cite{gabriel_pfc} which is a sharper version of $Z_1$. Another sharper way to compute $Z_1$ is demonstrated in \cite{miguel_soliton}. In this framework, the authors do not use $A^{\dagger}$ and directly estimate $\|I_d - A Df(\overline{u})\|$. By using non-square matrices, one can obtain a sharper estimate for $Z_1$. For our purposes, we chose not to use this approach due to the fact that it requires a larger rectangular matrix. As we are primarily interested in efficient computations for proving branches of periodic solutions later in this paper, we would prefer not to do computations with this larger matrix. Another possibility is the approach used by the authors of \cite{sh_cadiot}, which uses the Hilbert space properties of $\ell^2$ to remove the need for rectangular matrices. The tradeoff with this approach is that it requires one to do a proof in $\ell^2$, which is not a Banach algebra. Regardless, the computation performed in Lemma \ref{lem : Z1} is the most computationally efficient as it uses $\overline{V}$ and has no matrix multiplications whatsoever. Hence, it was sufficient for our purposes. 
\end{remark}
With the bounds now computed, we can present our results. 
\subsection{Results}\label{sec : results solutions}
In this section, we prove the existence of various periodic solutions with $D_3$-symmetry and $D_6$-symmetry. We construct our solution $\overline{u} \in \ell^1_{j\nu}$ on the parallelogram $\parallelogram_0$. By using Theorems \ref{th : triangle periodic} and \ref{th : hexagon periodic}, we obtain that the corresponding function's periodic tiling can be generated by $\Delta_1$ and $\Delta_2$ in the case of $D_3$, and $\varhexagon_0$ for the case of $D_6$. We present these results along with a demonstration of the periodicity. For all proofs in this section, we fix the following:
\begin{align}
    N = 80, ~ \nu = 1.09, ~ \text{and} ~ r_0 \bydef 3 \times 10^{-8}.
\end{align}
These values are used for each of the proofs.
\begin{theorem}[\bf The First Triangular Solution]\label{th : 1_1}
Let $\mu = 0.01, \gamma = 1.6, d = 10$. Then there exists a unique solution $\tilde{u}$ to \eqref{def : L and G sh 2d} in $\overline{B_{r_0}(\overline{u})} \subset \ell^1_{3,1.09}$ and we have that $\|\tilde{u}-\overline{u}\|_{1,1.09} \leq r_0$. Moreover, the periodic pattern $\tilde{u}$ can be generated by a tiling of $\Delta_1$ and $\Delta_2$.
\end{theorem}
\begin{proof}
We perform the full construction described in Section \ref{sec : numerical aspects} to build $\overline{u}$ and $A^N$. Using \cite{julia_blanco_D3D6,dominic_dihedral_julia}, we obtain
\begin{align}
    \|A^N\|_{\mathcal{B}(\ell^1_{3,1.09})} \leq 92.86 \text{,}~Y_0 \bydef 1.92 \times 10^{-8} \text{,}~Z_{2}(r_0) \bydef 10830.18   \text{,}~Z_1 \bydef 0.2315,~Z_0 \bydef 1.541 \times 10^{-9}.
    \end{align}
We prove that these values satisfy Theorem \ref{th : radii polynomial theorem}. 
\end{proof}
\begin{figure}[H]
\centering
 \begin{minipage}{.4\linewidth}
  \centering\epsfig{figure=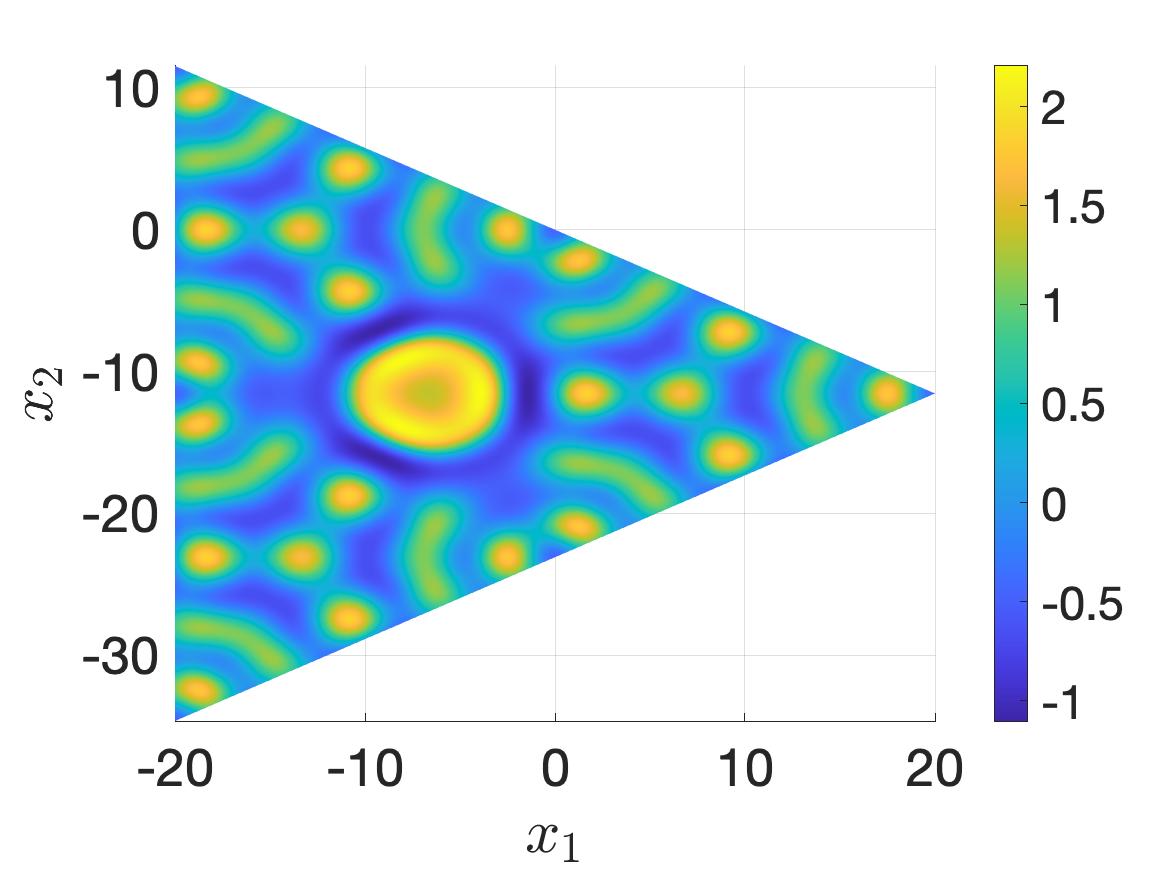,width=\linewidth}
 \end{minipage} %
 \begin{minipage}{.4\linewidth}
  \centering\epsfig{figure=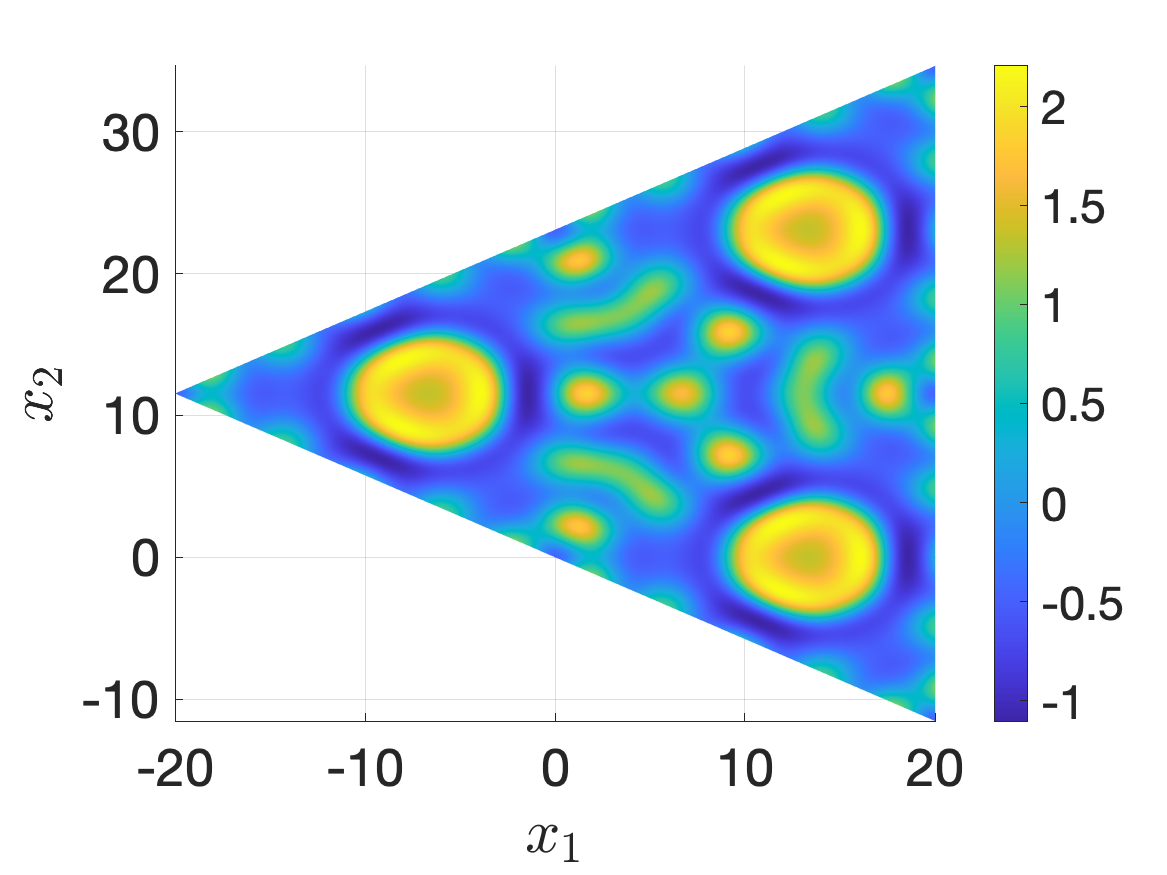,width=\linewidth}
 \end{minipage} 
 \caption{Plot of $\overline{u}$ of an approximation of a $D_3$ Pattern on $\Delta_1$ (left) and $\Delta_2$ (right)used in  Theorem \ref{th : 1_1}}\label{fig : th1_1}
 \end{figure}%
\begin{theorem}[\bf The Second Triangular Solution]\label{th : 1_2}
Let $\mu = 0.01, \gamma = 1.6, d = 5$. Then there exists a unique solution $\tilde{u}$ to \eqref{def : L and G sh 2d} in $\overline{B_{r_0}(\overline{u})} \subset \ell^1_{3,1.09}$ and we have that $\|\tilde{u}-\overline{u}\|_{1,1.09} \leq r_0$. Moreover, the periodic pattern $\tilde{u}$ can be generated by a tiling of $\Delta_1$ and $\Delta_2$.
\end{theorem}
\begin{proof}
We perform the full construction described in Section \ref{sec : numerical aspects} to build $\overline{u}$ and $A^N$. Using \cite{julia_blanco_D3D6,dominic_dihedral_julia}, we obtain
\begin{align}
    \|A^N\|_{\mathcal{B}(\ell^1_{3,1.09})} \leq 33.237\text{,}~Y_0 \bydef 3.455 \times 10^{-11} \text{,}~Z_{2}(r_0) \bydef 383.11   \text{,}~Z_1 \bydef 0.0194,~Z_0 \bydef 4.366 \times 10^{-11}.
    \end{align}
We prove that these values satisfy Theorem \ref{th : radii polynomial theorem}. 
\end{proof}
\begin{figure}[H]
\centering
 \begin{minipage}{.4\linewidth}
  \centering\epsfig{figure=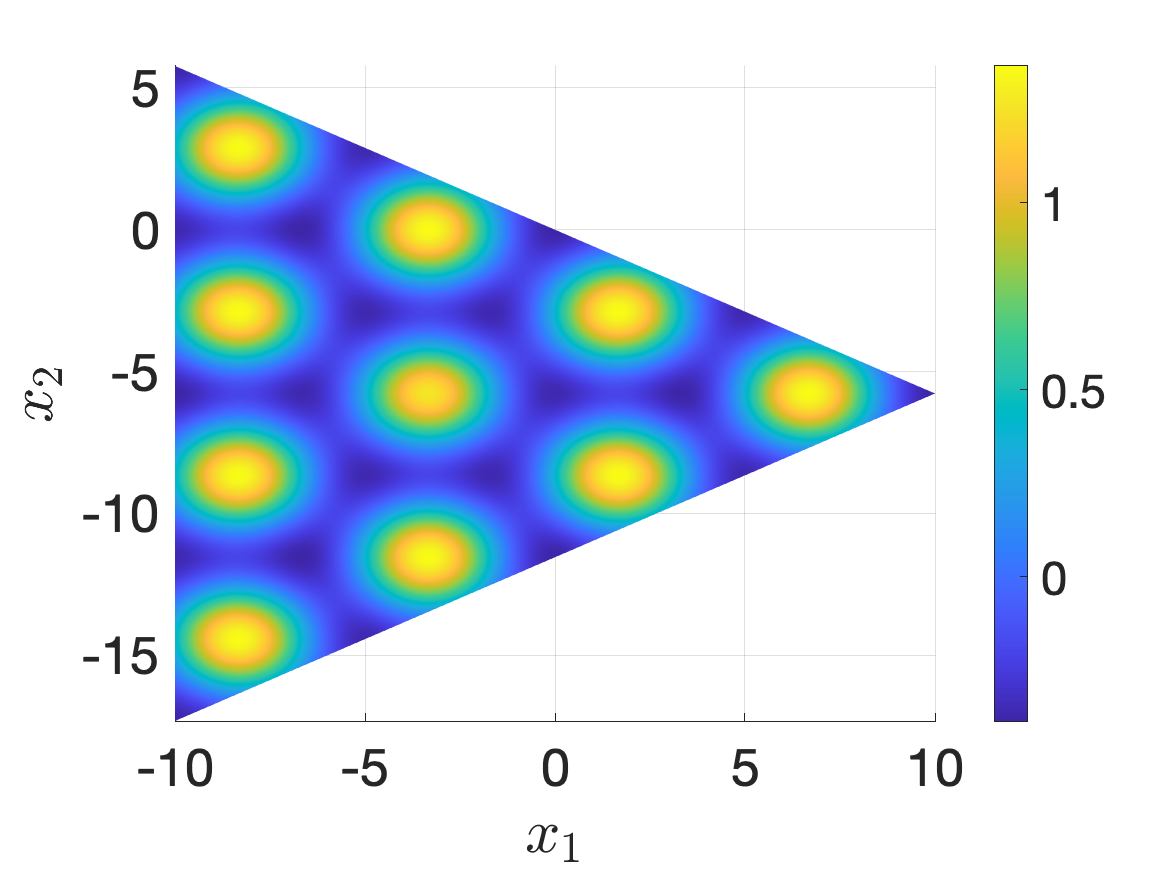,width=\linewidth}
 \end{minipage} %
 \begin{minipage}{.4\linewidth}
  \centering\epsfig{figure=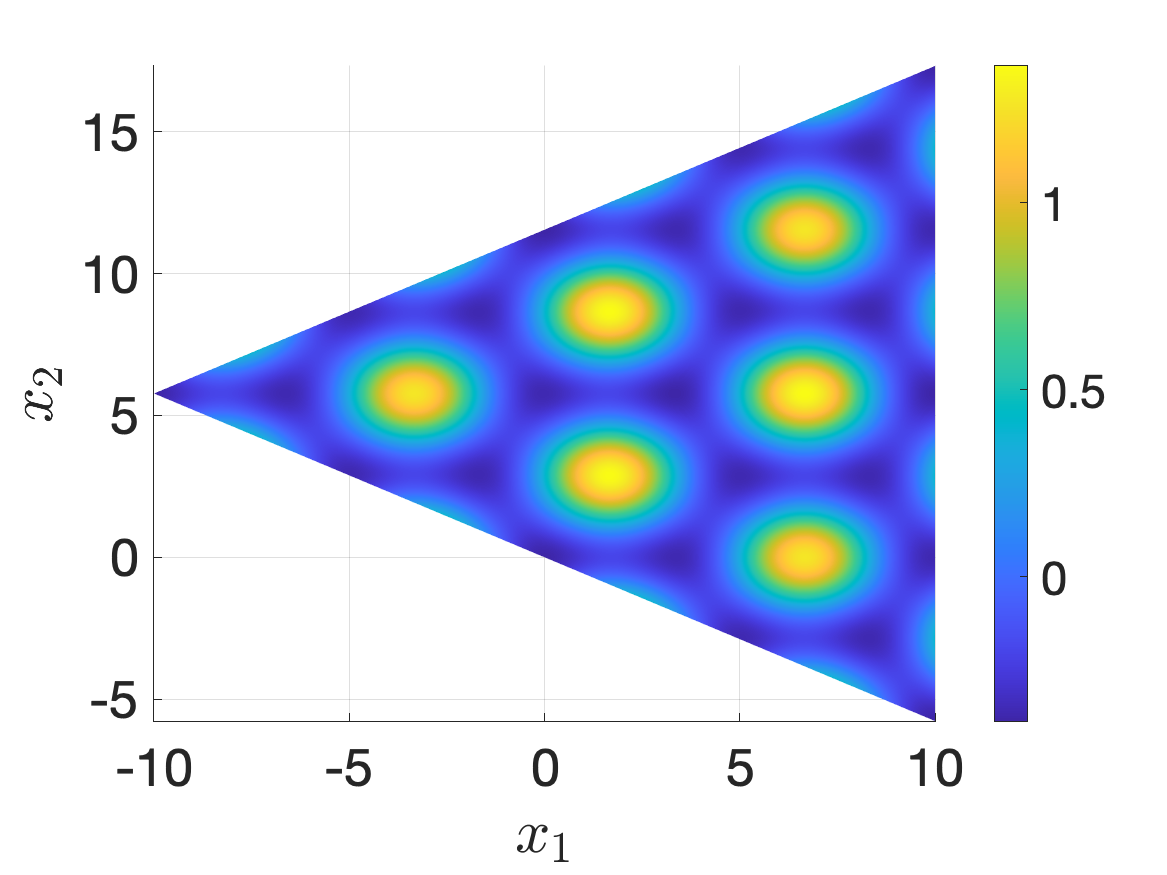,width=\linewidth}
 \end{minipage} 
 \caption{Plot of $\overline{u}$ of an approximation of a $D_3$ Pattern on $\Delta_1$ (left) and $\Delta_2$ (right) used in  Theorem \ref{th : 1_2}.}\label{fig : th1_2}
 \end{figure}%
\begin{theorem}[\bf The Third Triangular Solution]\label{th : 1_3}
Let $\mu = -0.01, \gamma = 1.7, d = 5$. Then there exists a unique solution $\tilde{u}$ to \eqref{def : L and G sh 2d} in $\overline{B_{r_0}(\overline{u})} \subset \ell^1_{3,1.09}$ and we have that $\|\tilde{u}-\overline{u}\|_{1,1.09} \leq r_0$. Moreover, the periodic pattern $\tilde{u}$ can be generated by a tiling of $\Delta_1$ and $\Delta_2$.
\end{theorem}
\begin{proof}
We perform the full construction described in Section \ref{sec : numerical aspects} to build $\overline{u}$ and $A^N$. Using \cite{julia_blanco_D3D6,dominic_dihedral_julia}, we obtain
\begin{align}
    \|A^N\|_{\mathcal{B}(\ell^1_{3,1.09})} \leq 6.981 \text{,}~Y_0 \bydef 3.84 \times 10^{-11} \text{,}~Z_{2}(r_0) \bydef 133.286   \text{,}~Z_1 \bydef 0.03,~Z_0 \bydef 2.1 \times 10^{-11}.
    \end{align}
We prove that these values satisfy Theorem \ref{th : radii polynomial theorem}. 
\end{proof}
\begin{figure}[H]
\centering
 \begin{minipage}{.4\linewidth}
  \centering\epsfig{figure=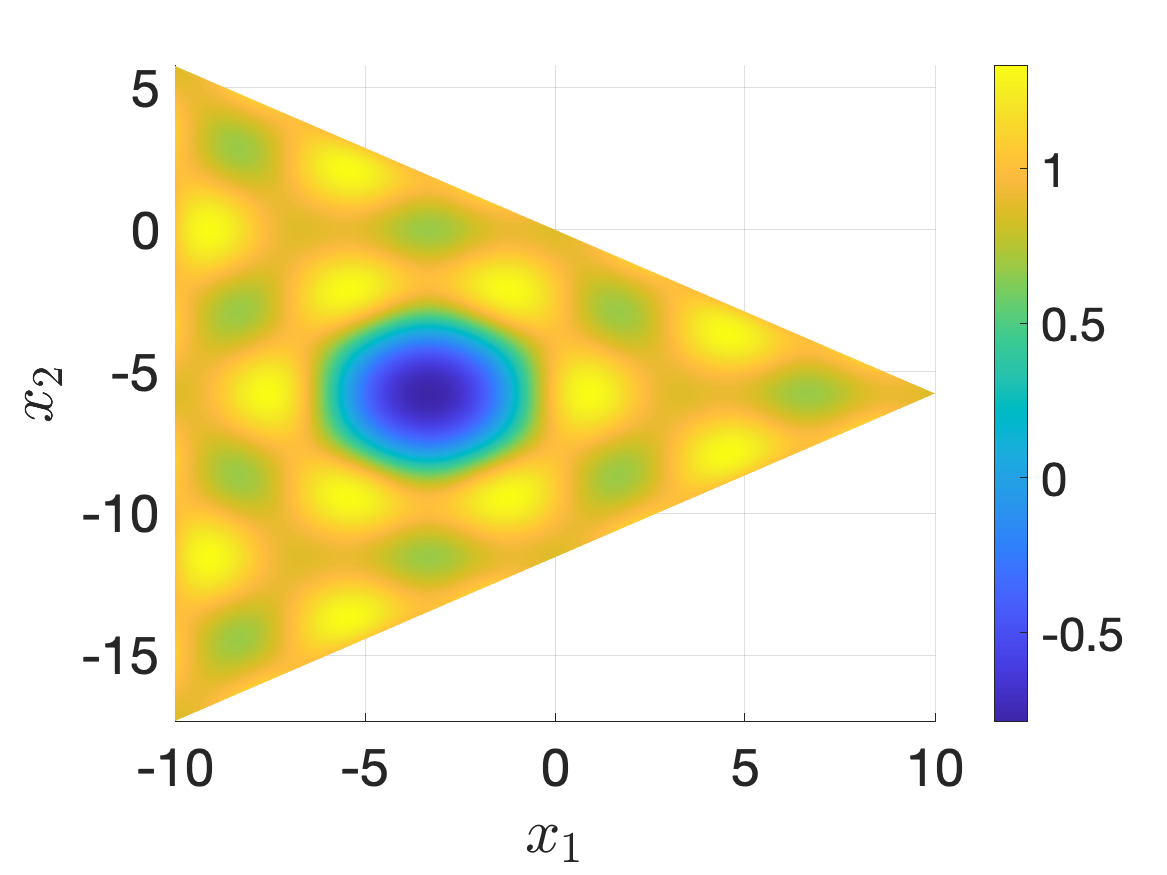,width=\linewidth}
 \end{minipage} %
 \begin{minipage}{.4\linewidth}
  \centering\epsfig{figure=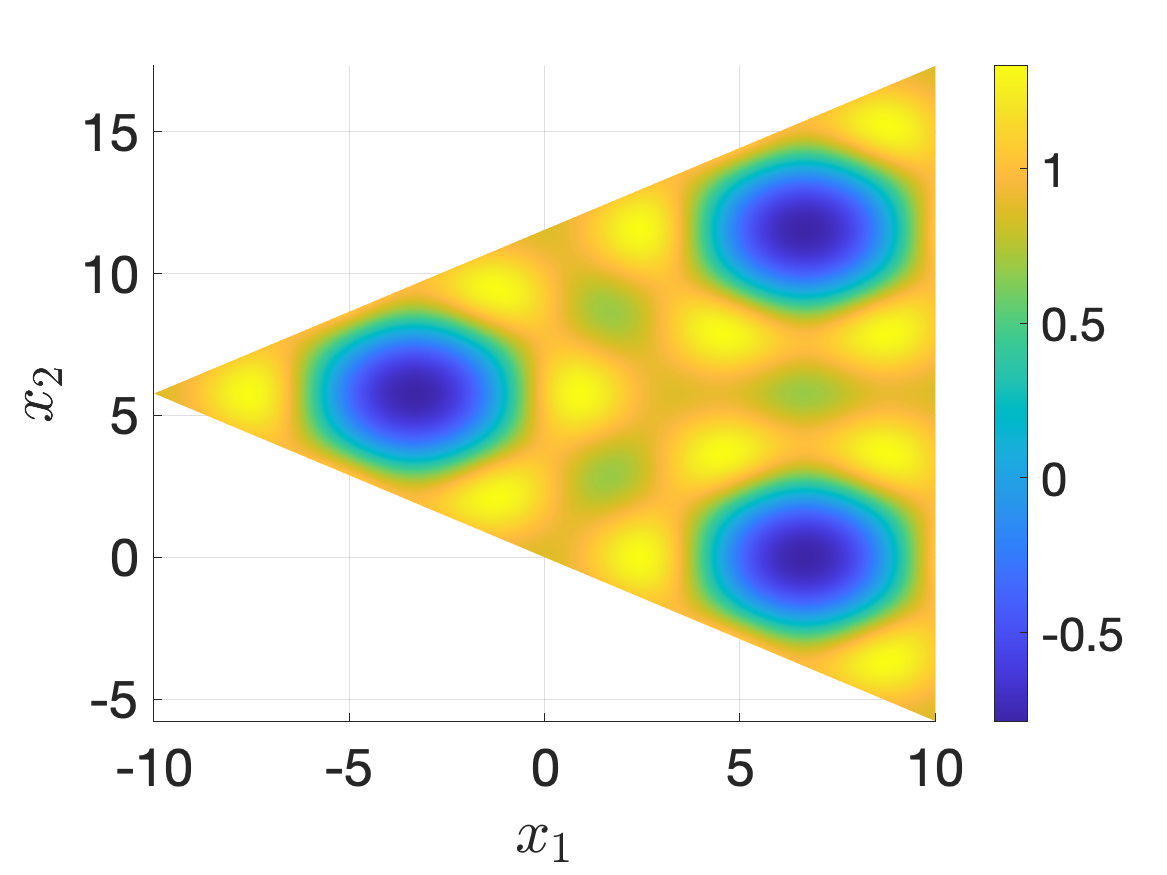,width=\linewidth}
 \end{minipage} 
 \caption{Plot of $\overline{u}$ of an approximation of a $D_3$ Pattern on $\Delta_1$ (left) and $\Delta_2$ (right) used in  Theorem \ref{th : 1_3}.}\label{fig : th1_3}
 \end{figure}%
\begin{theorem}[\bf The Fourth Triangular Solution]\label{th : 1_4}
Let $\mu = -0.2, \gamma = 2, d = 5$. Then there exists a unique solution $\tilde{u}$ to \eqref{def : L and G sh 2d} in $\overline{B_{r_0}(\overline{u})} \subset \ell^1_{3,1.09}$ and we have that $\|\tilde{u}-\overline{u}\|_{1,1.09} \leq r_0$. Moreover, the periodic pattern $\tilde{u}$ can be generated by a tiling of $\Delta_1$ and $\Delta_2$.
\end{theorem}
\begin{proof}
We perform the full construction described in Section \ref{sec : numerical aspects} to build $\overline{u}$ and $A^N$. Using \cite{julia_blanco_D3D6,dominic_dihedral_julia}, we obtain
\begin{align}
    \|A^N\|_{\mathcal{B}(\ell^1_{3,1.09})} \leq 6.95 \text{,}~Y_0 \bydef 6.32 \times 10^{-11} \text{,}~Z_{2}(r_0) \bydef 96.202   \text{,}~Z_1 \bydef 0.0112,~Z_0 \bydef 2.93 \times 10^{-11}.
    \end{align}
We prove that these values satisfy Theorem \ref{th : radii polynomial theorem}. 
\end{proof}
\begin{figure}[H]
\centering
 \begin{minipage}{.4\linewidth}
  \centering\epsfig{figure=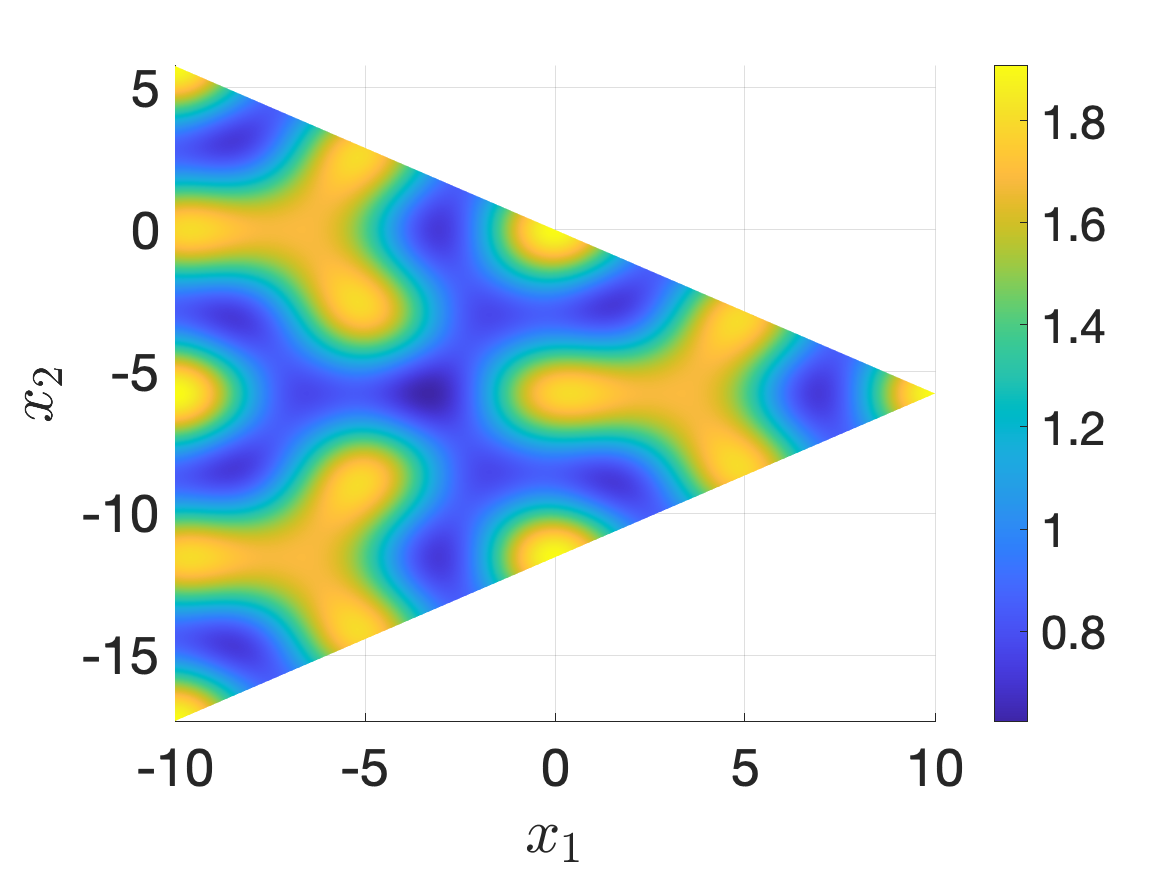,width=\linewidth}
 \end{minipage} %
 \begin{minipage}{.4\linewidth}
  \centering\epsfig{figure=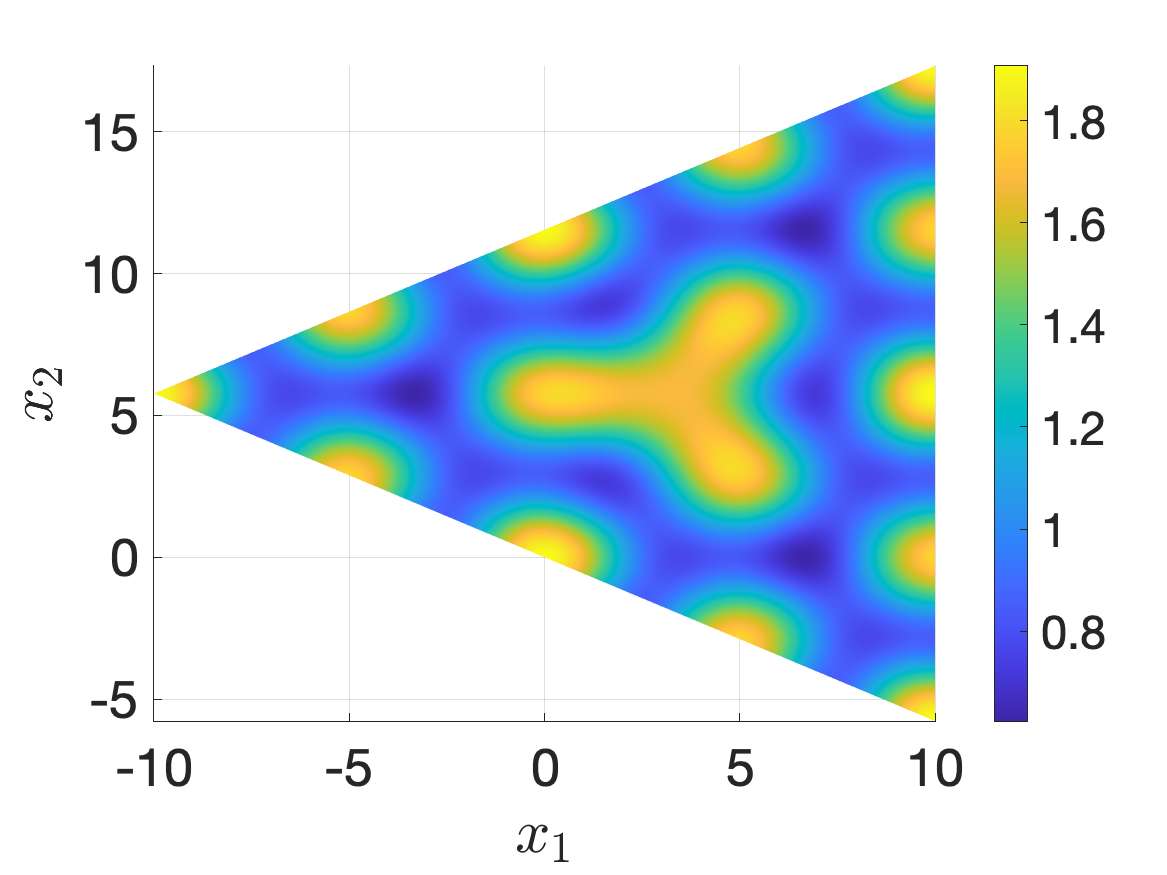,width=\linewidth}
 \end{minipage} 
 \caption{Plot of $\overline{u}$ of an approximation of a $D_3$ Pattern on $\Delta_1$ (left) and $\Delta_2$ (right) used in  Theorem \ref{th : 1_4}.}\label{fig : th1_4}
 \end{figure}%
\begin{theorem}[\bf The First Hexagonal Solution]\label{th : 2_1}
Let $\mu = -0.01, \gamma = 1.6, d = 10$. Then there exists a unique solution $\tilde{u}$ to \eqref{def : L and G sh 2d} in $\overline{B_{r_0}(\overline{u})} \subset \ell^1_{6,1.09}$ and we have that $\|\tilde{u}-\overline{u}\|_{1,1.09} \leq r_0$. Moreover, $\tilde{u}$ is periodic on $\varhexagon_0$.
\end{theorem}
\begin{proof}
We perform the full construction described in Section \ref{sec : numerical aspects} to build $\overline{u}$ and $A^N$. Using \cite{julia_blanco_D3D6,dominic_dihedral_julia}, we obtain
\begin{align}
    \|A^N\|_{\mathcal{B}(\ell^1_{6,1.09})} \leq 7.816 \text{,}~Y_0 \bydef 8.003 \times 10^{-12} \text{,}~Z_{2}(r_0) \bydef 131.28   \text{,}~Z_1 \bydef 0.092,~Z_0 \bydef 1.83 \times 10^{-11}.
    \end{align}
We prove that these values satisfy Theorem \ref{th : radii polynomial theorem}. 
\end{proof}
\begin{figure}[H]
\centering
\epsfig{figure=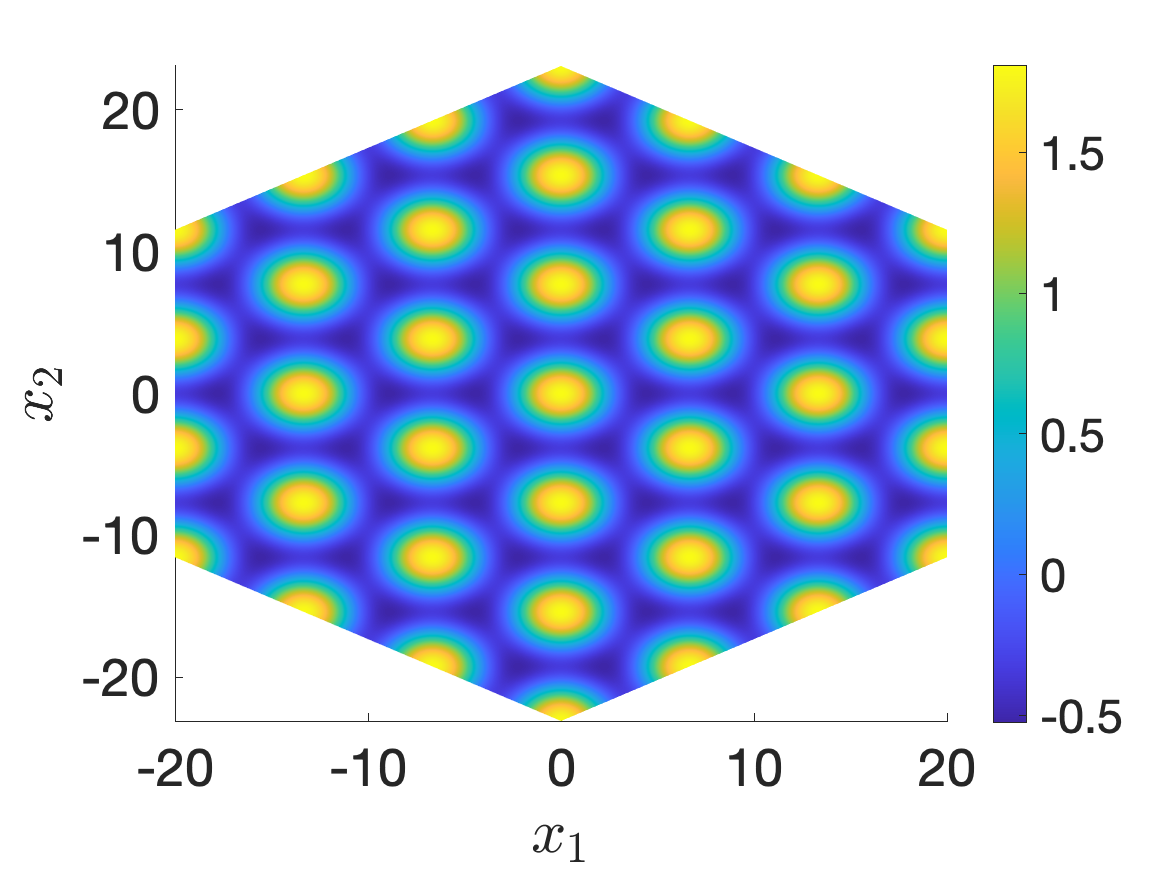,width=0.6\textwidth}
 \caption{Plot of $\overline{u}$ on $\varhexagon_0$ used in  Theorem \ref{th : 2_1}.}\label{fig : th2_1}
 \end{figure}%
\begin{theorem}[\bf The Second Hexagonal Solution]\label{th : 2_2}
Let $\mu = -0.1, \gamma = 2, d = 10$. Then there exists a unique solution $\tilde{u}$ to \eqref{def : L and G sh 2d} in $\overline{B_{r_0}(\overline{u})} \subset \ell^1_{6,1.09}$ and we have that $\|\tilde{u}-\overline{u}\|_{1,1.09} \leq r_0$. Moreover, $\tilde{u}$ is periodic on $\varhexagon_0$.
\end{theorem}
\begin{proof}
We perform the full construction described in Section \ref{sec : numerical aspects} to build $\overline{u}$ and $A^N$. Using \cite{julia_blanco_D3D6,dominic_dihedral_julia}, we obtain
{\small\begin{align}
    \|A^N\|_{\mathcal{B}(\ell^1_{6,1.09})} \leq 542.523 \text{,}~Y_0 \bydef 3.182 \times 10^{-11} \text{,}~Z_{2}(r_0) \bydef 3112.49   \text{,}~Z_1 \bydef 0.5015,~Z_0 \bydef 9.133 \times 10^{-10}.
    \end{align}}
We prove that these values satisfy Theorem \ref{th : radii polynomial theorem}. 
\end{proof}
\begin{figure}[H]
\centering
\epsfig{figure=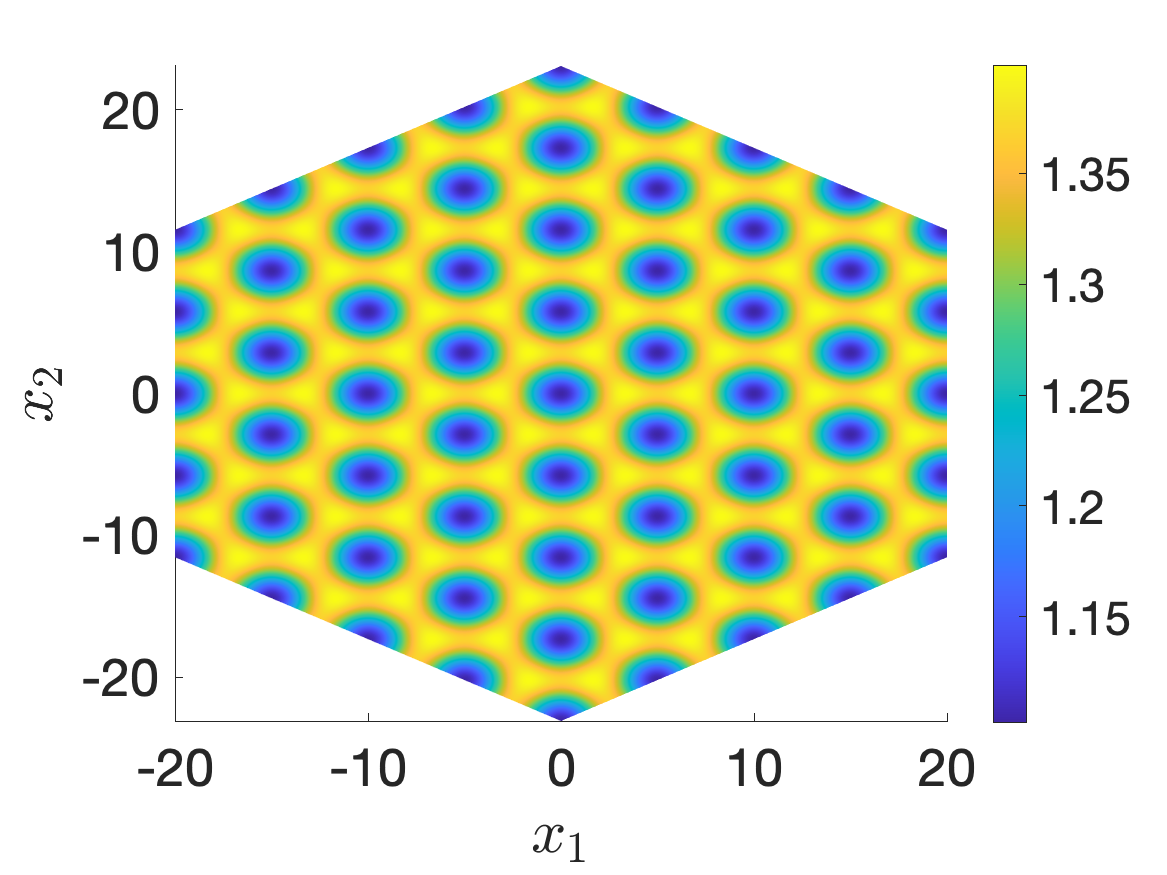,width=0.6\textwidth}
 \caption{Plot of $\overline{u}$ on $\varhexagon_0$ used in  Theorem \ref{th : 2_2}.}\label{fig : th2_2}
 \end{figure}%
\begin{theorem}[\bf The Third Hexagonal Solution]\label{th : 2_3}
Let $\mu = 0.3, \gamma = 2.1, d = 5$. Then there exists a unique solution $\tilde{u}$ to \eqref{def : L and G sh 2d} in $\overline{B_{r_0}(\overline{u})} \subset \ell^1_{6,1.09}$ and we have that $\|\tilde{u}-\overline{u}\|_{1,1.09} \leq r_0$. Moreover, $\tilde{u}$ is periodic on $\varhexagon_0$.
\end{theorem}
\begin{proof}
We perform the full construction described in Section \ref{sec : numerical aspects} to build $\overline{u}$ and $A^N$. Using \cite{julia_blanco_D3D6,dominic_dihedral_julia}, we obtain
\begin{align}
    \|A^N\|_{\mathcal{B}(\ell^1_{6,1.09})} \leq 44.52 \text{,}~Y_0 \bydef 1.799 \times 10^{-10} \text{,}~Z_{2}(r_0) \bydef 813.455   \text{,}~Z_1 \bydef 0.0308,~Z_0 \bydef 7.962 \times 10^{-11}.
    \end{align}
We prove that these values satisfy Theorem \ref{th : radii polynomial theorem}. 
\end{proof}
\begin{figure}[H]
\centering
\epsfig{figure=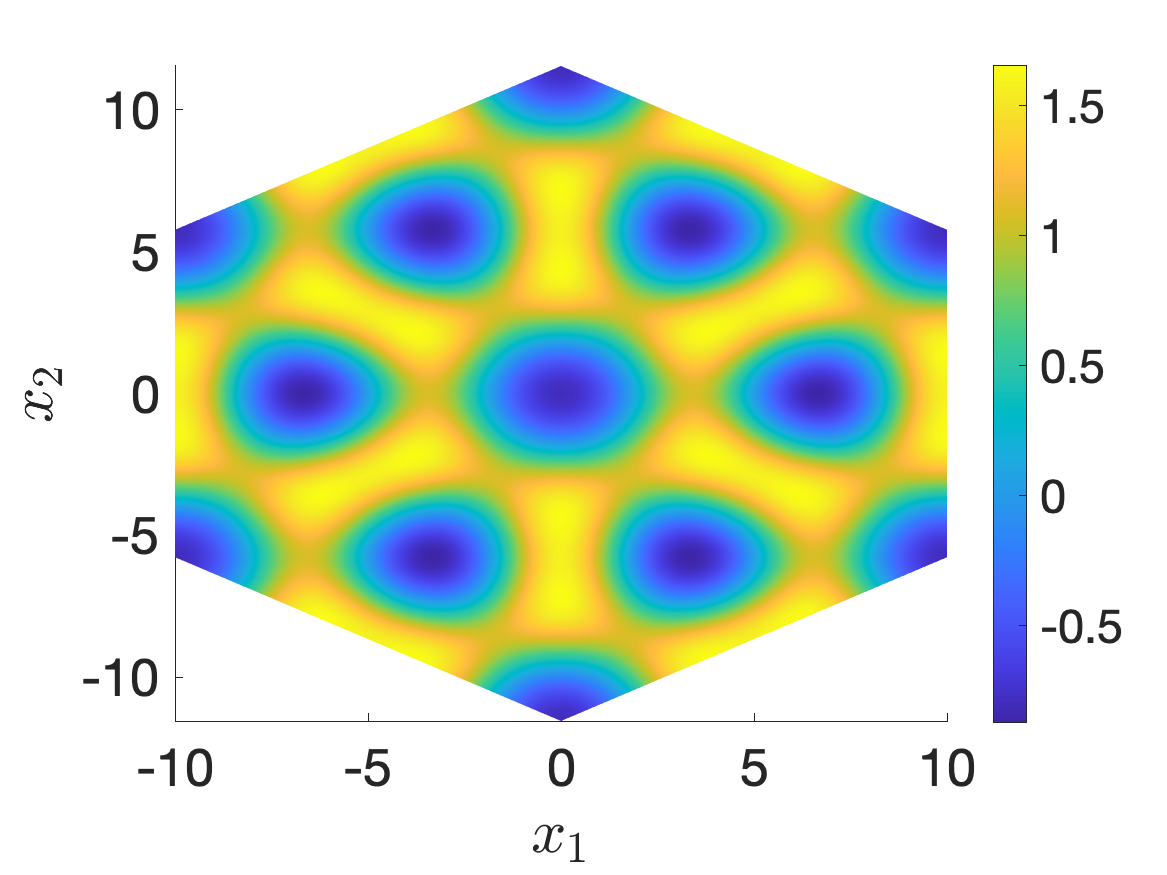,width=0.6\textwidth}
 \caption{Plot of $\overline{u}$ on $\varhexagon_0$ used in  Theorem \ref{th : 2_3}.}\label{fig : th2_3}
 \end{figure}%
\begin{theorem}[\bf The Fourth Hexagonal Solution]\label{th : 2_4}
Let $\mu = 0.25, \gamma = 2, d = 15$. Then there exists a unique solution $\tilde{u}$ to \eqref{def : L and G sh 2d} in $\overline{B_{r_0}(\overline{u})} \subset \ell^1_{6,1.09}$ and we have that $\|\tilde{u}-\overline{u}\|_{1,1.09} \leq r_0$. Moreover, $\tilde{u}$ is periodic on $\varhexagon_0$.
\end{theorem}
\begin{proof}
We perform the full construction described in Section \ref{sec : numerical aspects} to build $\overline{u}$ and $A^N$. Using \cite{julia_blanco_D3D6,dominic_dihedral_julia}, we obtain
\begin{align}
    \|A^N\|_{\mathcal{B}(\ell^1_{6,1.09})} \leq 44.624 \text{,}~Y_0 \bydef 7.35 \times 10^{-9} \text{,}~Z_{2}(r_0) \bydef 3773.58   \text{,}~Z_1 \bydef 0.2204,~Z_0 \bydef 3.13 \times 10^{-10}.
    \end{align}
We prove that these values satisfy Theorem \ref{th : radii polynomial theorem}. 
\end{proof}
\begin{figure}[H]
\centering
\epsfig{figure=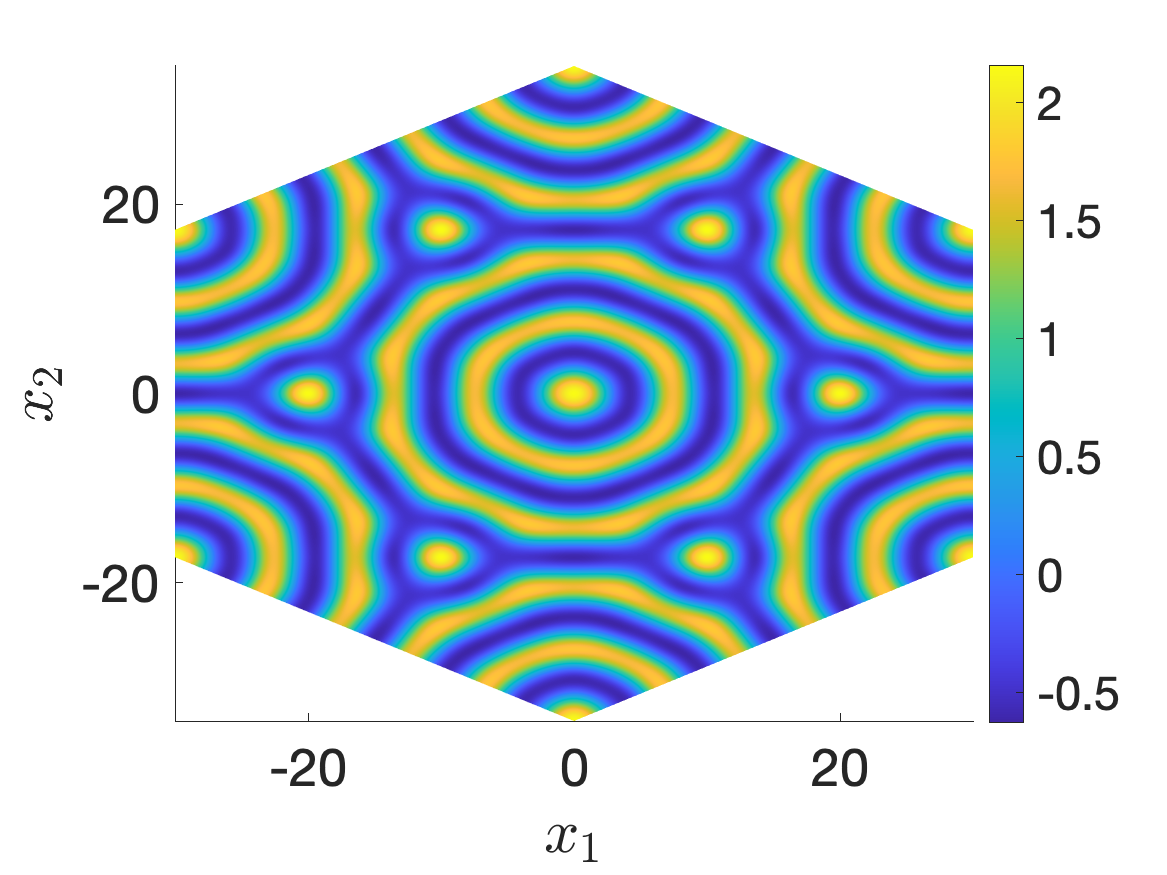,width=0.6\textwidth}
 \caption{Plot of $\overline{u}$ on $\varhexagon_0$ used in  Theorem \ref{th : 2_4}.}\label{fig : th2_4}
 \end{figure}%
\section{Branches of Hexagonal and Triangular Solutions}\label{sec : continuation}
In this section, we wish to prove the existence of a branch of periodic solutions whose periodic patterns can be generated by two triangles or one hexagon depending on the symmetry. There is a strong interest in branches of solutions in 2D Swift Hohenberg. In particular, the authors of \cite{jason_spot_paper} and \cite{jason_ring_paper} provide various approximate branches of solutions. Some of these branches emerge from a Turing bifurcation that lead to homoclinic snaking. By following the snaking curve on a bounded domain, one can obtain periodic patterns as shown in \cite{hexagon2008}. Additionally, a rich variety of purely periodic patterns exist in \eqref{eq : swift_hohenberg} with both hexagonal and triangular symmetry. We would like to provide rigorous proofs of the symmetry of these branches of solutions. That is, we would like to prove that the branch itself is $D_j$-symmetric. This means that each solution along the branch is a periodic pattern that can be generated by two triangles (in the case of $D_3$) or one hexagon (in the case of $D_6$). Note that we are not directly proving the branches given in the aforementioned papers. Our goal is to provide the theory to prove branches of periodic solutions with triangular and hexagonal symmetry. Our approach could then be used to prove specific branches such as the ones identified in these papers; however, we did not do so. Our approach will be computer-assisted and was developed by the authors of \cite{polynomial_chaos}. The approach is based on Chebyshev series, and allows us to parameterize the branch as a Chebyshev series expansion. It has been applied before in previous works, and we refer the interested reader to \cite{cadiot_witham,marschal} for examples.
\par For our purposes, we will be performing continuation in $\mu$. We will use pseudo-arclength continuation to perform our rigorous proof. The benefit of pseudo-arclength continuation is that we can pass through saddle node (or fold) bifurcations. We will require this in order to prove Theorems \ref{th : hexagon branch 2} and \ref{th : hexagon branch 3}. To begin, we define $X_{j,\nu} \bydef \mathbb{R} \times \ell^1_{j,\nu}$ and write $w \bydef (\mu,u) \in X_{j,\nu}$. We define its norm to be 
\begin{align}
    \|w\|_{X_{j,\nu}} \bydef |\mu| + \|u\|_{1,\nu}.
\end{align}
We also define $\mathbb{N}_0 \bydef \mathbb{N} \cup \{0\}$. Now, we will expand $\mu$ and $u$ as Chebyshev series dependent on the pseudo-arclength, $s$. That is, we write
\begin{align}
    \mu(s) \bydef \mu_0 + 2\sum_{k \in \mathbb{N}_0} \mu_k T_k (s) ~~\text{and}~~u(s) \bydef u_0 + 2\sum_{k \in \mathbb{N}_0 } u_k T_k (s), ~ s \in [-1,1]
\end{align} 
where $T_k: [-1,1] \to \mathbb{R}$ are the Chebyshev polynomials of the first kind and $\mu_k \in \mathbb{R}, u_k \in \ell^1_{j,\nu}$ for all $k \in \mathbb{N}_0$. We now introduce some additional spaces. The first is $\mathbb{R}_{con}$, which will consistent of Chebyshev sequences $\mu(s) \bydef (\mu_k)_{k \in 0,\dots,K}$ and each $\mu_k \in \mathbb{R}$. We now define its norm, and an immediate inequality which we will rely on
\begin{align}
     \|\mu(s)\|_{\mathbb{R}_{con}} \bydef \sup_{s \in [-1,1]} |\mu(s)|
     &= \sup_{s \in [-1,1]} \left| \mu_0 + 2\sum_{k \in \mathbb{N}_0} \mu_k T_k(s)\right| \\
     &\leq |\mu_0| + 2\sup_{s \in [-1,1]} \sum_{k \in \mathbb{N}_0} |\mu_k T_k(s)|  \\
     &\leq |\mu_0| + 2\sum_{k \in \mathbb{N}_0} |\mu_k|.
    \end{align}
We rely on the inequality above for every estimate we perform. As a result, we will introduce the notation
\begin{align}
    \|\mu(s)\|_{\mathbb{R}_{con,ineq}} \bydef |\mu_0| + 2\sum_{k \in \mathbb{N}_0} |\mu_k|.\label{Rconineq}
\end{align}
The second space we define is $\ell^1_{j,\nu,con}$, which will consist of Chebyshev sequences $u(s) \bydef (u_k)_{k \in 0,\dots,K}$ and each $u_k \in \ell^1_{j,\nu}$. Similarly to what was done for $\mathbb{R}_{con}$, we have
\begin{align}
    \|u(s)\|_{\ell^1_{j,\nu,con}} \bydef \sup_{s \in [-1,1]} \|u(s)\|_{1,\nu} &= \sup_{s \in [-1,1]} \left\| u_0 + 2\sum_{k \in \mathbb{N}_0} u_k T_k(s)\right\|_{1,\nu}\\
    &\leq \|u_0\|_{1,\nu} + 2\sup_{s \in [-1,1]} \sum_{k \in \mathbb{N}_0} \|u_k\|_{1,\nu}  |T_k(s)| \\
    &\leq \|u_0\|_{1,\nu} + 2\sum_{k \in \mathbb{N}_0} \|u_k\|_{1,\nu}.
    \end{align}
We again will exclusively use the inequality when doing our analysis. Hence, we write
\begin{align}
    \|u(s)\|_{\ell^1_{j,\nu,con,ineq}} \bydef \|u_0\|_{1,\nu} + 2\sum_{k \in \mathbb{N}_0} \|u_k\|_{1,\nu}.\label{ell1conineq}
\end{align}
The third space is $X_{j,\nu,con}$. 
 We write $w(s) \bydef (\mu(s),u(s)) \in \mathbb{R}_{con} \times \ell^1_{j,\nu,con} \bydef X_{j,\nu,con}$ and its norm
\begin{align}
    \|w(s)\|_{X_{j,\nu,con}} \bydef \sup_{s \in [-1,1]} \|w(s)\|_{X_{j,\nu}} &= \sup_{s \in [-1,1]} \left\| w_0 + 2\sum_{k \in \mathbb{N}_0} w_k T_k(s)\right\|_{X_{j,\nu}}\\
    &\leq \|w_0\|_{X_{j,\nu}} + 2\sup_{s \in [-1,1]}\sum_{k\in \mathbb{N}_0} \|w_{k} \|_{X_{j,\nu}} |T_k(s)| \\
    &\leq  \|w_0\|_{X_{j,\nu}} + 2\sum_{k \in \mathbb{N}_0} \|w\|_{X_{j,\nu}}.
\end{align}
Once again, we write
\begin{align}
    \|w(s)\|_{X_{j,\nu,con,ineq}} \bydef \|w_0\|_{X_{j,\nu}} + 2\sum_{k \in \mathbb{N}_0} \|w\|_{X_{j,\nu}}.\label{Xjnuconineq}
\end{align}
With the spaces above and their norm estimates defined, we then define the map $F : X_{j,\nu,con} \to \mathcal{S}$ for $\mathcal{S}$ another Banach space as
\begin{align}
    F(w(s)) \bydef \begin{bmatrix}
        (u(s) - \overline{u}(s),\dot{u}(s))_2 \\
        f(\mu(s),u(s))
    \end{bmatrix}\label{eq : psuedo arclength system}
\end{align}
where $\dot{u}(s)$ is the second component of the \emph{tangent vector}, $\dot{w}(s)$. We can then construct $\overline{w}(s) \bydef (\overline{\mu}(s),\overline{u}(s))$ but with finitely many Chebyshev coefficients $K$. This will be the topic of the next section. Furthermore, we can use the same construction for the approximate inverse of $DF(\overline{w}(s))$. Firstly, since $\overline{w}(s)$ has finitely many Chebyshev coefficients, so does $DF(\overline{w}(s))$. Hence, it makes sense to define the approximate inverse with finitely many coefficients as well. We denote it by $B(s) : \mathcal{S} \to X_{j,\nu,con}$ 
\begin{align}
    B(s) = B_0 + 2 \sum_{k = 1}^{K} B_k T_k (s).
\end{align}
Then, the following theorem makes use of a uniform contraction argument to allow us to rigorously prove a branch of solutions. 
\begin{theorem}[Newton-Kantorovich Theorem for Branches]
\label{th: radii polynomial continuation}
Let $B^{\dagger}(s) \in \mathcal{B}(X_{j,\nu,con},\mathcal{S})$. Also let $B(s) \in \mathcal{B}(\mathcal{S},X_{j,\nu,con})$ be injective. Moreover, let $Y_0,Z_0 Z_1$ be non-negative constants and let  $Z_2 : (0, \infty) \to [0,\infty)$ be a non-negative function  such that
  \begin{align}\label{eq: definition Y0 Z1 Z2 continuation}
    \sup_{s \in [-1,1]}\|B(s)F(\overline{w}(s))\|_{X_{j,\nu,con}} & \le Y_0\\
    \sup_{s \in [-1,1]}\|I_d - B(s)B^{\dagger}(s))\|_{\mathcal{B}(X_{j,\nu,con})} &\le Z_0\\
    \sup_{s \in [-1,1]} \|B(s)(DF(\overline{w}(s)) - B^{\dagger}(s))\|_{\mathcal{B}(X_{j,\nu,con})} &\leq Z_1 \\
    \sup_{s \in [-1,1]}\|B(s)\left({D}F(w(s) - DF(\overline{w}(s))\right)\|_{\mathcal{B}(X_{j,\nu,con})} & \le Z_2(r)r, ~~ \text{for all } w(s) \in B_r(w_0(s))
\end{align}  
If there exists $r>0$ such that 
\begin{equation}\label{condition radii polynomial continuation}
    \frac{1}{2}Z_2(r)r^2 - (1- Z_0 - Z_1)r + Y_0 <0, \ and \ Z_0 + Z_1 + Z_2(r)r < 1
 \end{equation}
 is satisfied, 
then for every $s \in [-1,1]$, there exists a unique $\tilde{w}(s) \in \overline{B_r(\overline{w}(s))} \subset X_{j,\nu,con}$ such that $F(\tilde{w}(s))=0$. Moreover, the function $s \to \tilde{w}(s)$ is of class $C^{\infty}.$
\end{theorem}
\begin{proof}
The proof can be found in \cite{continuation_1}, \cite{continuation_2}, and \cite{continuation_3} for instance. 
\end{proof}
\begin{remark}
Note that the constants $Y_0, Z_0, Z_1,$ and $Z_2$ from Theorem \ref{th: radii polynomial continuation} are not the same as those defined in \ref{th : radii polynomial theorem}. The constants in Theorem \ref{th: radii polynomial continuation} are computed for branches of periodic solutions, whereas from Theorem \ref{th : radii polynomial theorem} they are for a single periodic solution. We have abused notation and given them the same names for uniformity with previous literature in CAPs.   
\end{remark}
Let us now discuss the numerical aspects of this approach in order to apply Theorem \ref{th: radii polynomial continuation}. 
\subsection{Construction of \texorpdfstring{$\overline{w}(s),B^{\dagger}(s)$}{WBdag} and \texorpdfstring{$B(s)$}{B}}\label{sec : continuation construction}
In this section, we will discuss the computation of $\overline{w}(s), B^{\dagger}(s),$ and $B(s)$. Our approach will rely on numerical pseudo-arclength continuation and some bijective transformation $\mathcal{T}$ which will allow us to move from grid points to Chebyshev coefficients and vise versa. To compute each $\overline{w}_k \bydef (\overline{\mu}_k,\overline{u}_k)$ for the chosen number of Chebyshev coefficients $K$, we first fix the number of grid points, denoted $K_{grid}$. Following this, we fix an arclength, denoted $s_{fix} \in \mathbb{R}^+$, and create a grid of points
\begin{align}
    s_{grid,p} \bydef \frac{1}{2} s_{fix} - \frac{1}{2}s_{fix} \cos \left(\frac{\pi k}{K_{grid}}\right) ~~ \text{for} ~~ p = 0,\dots,K_{grid}-1.
\end{align}
The grid of $s_{grid,p}$'s is meant to give us the step-sizes for our pseudo-arclength continuation. That is, the step-size of each pseudo-arclength continuation step will be $ds_p \bydef s_{grid,p} - s_{grid,p-1}$. To explain further, suppose we have obtained an approximate solution, $\overline{w}_{grid,p} \bydef (\overline{\mu}_{grid,p},\overline{u}_{grid,p})$ to $f(\overline{w}_{grid,p}) \approx 0$ and computed the tangent vector to the branch, denoted $\dot{w}_{grid,p}$. Then, we define
\begin{align}
    \overline{w}_{grid,p+1} \bydef \overline{w}_{grid,p} \pm ds_k \dot{w}_{grid,p}
\end{align}
where we choose the sign of $ds_k$ depending on the direction we wish to continue. We do this for all $k = 0,\dots,K_{grid} - 1$.
At this point, we have a grid of approximate solutions $\overline{w}_{grid,p}$ for all $k = 0,\dots,K_{grid}$ to $f(\overline{w}_{grid,p}) \approx 0$. We then use our inverse transformation $\mathcal{T}^{-1}$ (which exists since $\mathcal{T}$ is bijective) to compute the approximate branch of solutions $\overline{w}(s) \bydef (\overline{w}_n)_{n \in \{0,\dots,K\}}$. That is, we use the inverse transform on the grid points to recover the Chebyshev coefficients $\overline{w}(s)$. Note that we can go back and forth between $\overline{w}(s)$ and the grid points $\overline{w}_{grid,p}$ using $\mathcal{T}$ and $\mathcal{T}^{-1}$ respectively.
To demonstrate the connection, we would approximately have
\begin{align}
\overline{w}\left(\frac{2(s_{grid,p} - s_{fix})}{s_{fix}} + 1\right) \approx \overline{w}_{grid,p}, ~~ \text{for} ~~ p = 0,\dots,K_{grid} - 1
\end{align}
when evaluating at the grid points $s_{grid,p}$. 
\par Let us now move to the discussion of $B^{\dagger}(s)$ and $B(s)$. To begin, we would like to construct a grid of $B^{\dagger}_{grid,p}$ and approximate inverses $B_{grid,p}$ for $p = 0,\dots, K_{grid} - 1$. That is, we would like to have $B^{\dagger}_{grid,p} \approx DF(\overline{w}_{grid,p})$ and $B_{grid,p} \approx DF(\overline{w}_{grid,p})^{-1}$ for each $p$. We first define some notations. We introduce $\Pi^N, \Pi_N$ and $\mathscr{h}_n$ for $\mathscr{h} \bydef (\eta,h) \in X_{j,\nu}$. These are
\begin{align}
    \Pi^N \bydef \begin{bmatrix}
        1 & 0 \\
        0 & \pi^N 
    \end{bmatrix}, ~~ \Pi_N \bydef \begin{bmatrix}
        0 & 0 \\
        0 & \pi_N 
    \end{bmatrix}, ~~(\mathscr{h})_n \bydef (\eta,h_n).
\end{align}
Then, we define the action of $B^{\dagger}_{grid,p}$ as
\begin{align}
    (B^{\dagger}_{grid,p}\mathscr{h})_n \bydef \begin{cases}
        \left[\Pi^N DF(\overline{w}_{grid,p}) \Pi^N\mathscr{h}\right]_n & n \in I^N \\
        \begin{bmatrix}
            0 \\
            ((1 + |\mathcal{L} \tilde{n}|^2)^2) h_n
        \end{bmatrix} & n \in \mathcal{Z}_{\mathrm{red}}(D_j) \setminus I^N
    \end{cases} 
\end{align}
Note that we have excluded the $\overline{\mu}_{grid,p}$ in the tail for simplicity.
Following this, we choose $B^N_{grid,p} \approx (\Pi^N DF(\overline{w}_{grid,p})\Pi^N)^{-1}$ and then each $B_{grid,p}$ is defined as
\begin{align}
    (B_{grid,p}\mathscr{h})_n \bydef \begin{cases}
        (B^N_{grid,p}\Pi^N \mathscr{h})_n & n \in I^N \\
        \begin{bmatrix}
            0 \\
            \frac{h_n}{(1 + |\mathcal{L}\tilde{n}|^2)^2}
        \end{bmatrix} & n \in \mathcal{Z}_{\mathrm{red}}(D_j) \setminus I^N
    \end{cases}.
\end{align}
At this point, we now use $\mathcal{T}^{-1}$ again to compute the approximate branch of $B^{\dagger}(s)$ and  approximate inverses $B(s)$. As before, we would approximately have
{\small\begin{align}
    &B^{\dagger}\left(\frac{2(s_{grid,p} - s_{fix})}{s_{fix}} + 1\right) \approx B^{\dagger}_{grid,p},~B\left(\frac{2(s_{grid,p} - s_{fix})}{s_{fix}} + 1\right) \approx B_{grid,p}, ~~ \text{for} ~~ p = 0,\dots,K_{grid} - 1
\end{align}}
when evaluating at the grid points $s_{grid,p}$. Finally, we need an estimate for $\|\Pi_N B(s)\|_{\mathcal{B}(X_{j,\nu,con})}$. To begin, let $z(s) \in X_{j,\nu,\mathrm{con}}$. Notice that since $B(s)$ has $K$ many Chebyshev coefficients, we get
\begin{align}
    \|\Pi_N B(s)\|_{\mathcal{B}(X_{j,\nu,con})} &= \|\Pi_N B(s)z(s)\|_{X_{j,\nu,con}} \\
    &\leq \|\Pi_N B_0 z_0\|_{X_{j,\nu}} + 2\sum_{k = 1}^{K} \|\Pi_N B_k z_k \|_{X_{j,\nu}} \\
    &\leq \|\Pi_N B_0\|_{\mathcal{B}(X_{j,\nu})} \|z_0\|_{X_{j,\nu}} + 2\sum_{k = 1}^{K} \|\Pi_N B_k \|_{\mathcal{B}(X_{j,\nu})} \|z_k\|_{X_{j,\nu}}.
\end{align}
Now, we have
{\small\begin{align}
    \|\Pi_N B_k\|_{\mathcal{B}(X_{j,\nu})} \leq \max_{m \in \mathbb{Z}^2 \setminus I^N} \frac{1}{|(1+|\mathcal{L}|\tilde{m}|^2)^2|} \leq  \max_{m \in \mathbb{Z}^2 \setminus I^N} \frac{1}{|(1 + |\mathcal{L} \tilde{m}|^2)^2|} \bydef \frac{1}{L_{N,0}},
\end{align}}
where
\begin{align}
    L_{N,0} \bydef \min_{m \in \mathbb{Z}^2 \setminus I^N} |(1 + |\mathcal{L} \tilde{m}|^2)^2|.\label{def : L_NK}
\end{align}
The final estimate is uniform which is why we excluded $\mu$ in the definition of the tail for $B_{grid,p}$ and $B^{\dagger}_{grid,p}$. Indeed, we obtain
\begin{align}
    \|\Pi_N B(s) \|_{\mathcal{B}(X_{j,\nu,con})} &\leq \|\Pi_N B_0\|_{\mathcal{B}(X_{j,\nu})} \|\Pi_N z_0\|_{X_{j,\nu}} + 2\sum_{k = 1}^{K} \|\Pi_N B_k \|_{\mathcal{B}(X_{j,\nu})} \|\Pi_N z_k \|_{X_{j,\nu}} \\
    &\leq \frac{1}{L_{N,0}} \|z_0\|_{X_{j,\nu}} + 2 \sum_{k = 1}^{K} \frac{1}{L_{N,0}}\|z_k \|_{X_{j,\nu}}  \\
    &= \frac{1}{L_{N,0}} \left( \|z_0\|_{X_{j,\nu}} + 2\sum_{k = 1}^K \|z_k \|_{X_{j,\nu}}\right) \\
    &\leq \frac{1}{L_{N,0}} \left(\|z(s)\|_{X_{j,\nu,con}} + 2\sum_{k = 1}^K \|z(s)\|_{X_{j,\nu,con}}\right) \\
    &=\frac{2K+1}{L_{N,0}}.\label{B(s)_norm_estimate}
\end{align}
Note that the factor $2K+1$ appears as we had to use the bound  
$ \|z_p\|_{X_{j,\nu}} \leq \|z(s)\|_{X_{j,\nu}}$ for each $p$ in order to obtain a uniform estimate. We would have no reason to expect the sum of the first $K$ terms to be bounded by $\|z(s)\|_{X_{j,\nu}}$ in general. Note that in certain cases, we can avoid this additional factor as we can use 
\begin{align}
    \frac{1}{L_{N,0}} \left( \|z_0\|_{X_{j,\nu}} + 2\sum_{k = 1}^K \|z_k \|_{X_{j,\nu}}\right) &\leq  \frac{1}{L_{N,0}} \left( \|z_0\|_{X_{j,\nu}} + 2\sum_{k \in \mathbb{N}_0} \|z_k \|_{X_{j,\nu}}\right) \\
    &= \frac{1}{L_{N,0}} \|z(s)\|_{X_{j,\nu,con,ineq}}.\label{tail ineq estimate}
\end{align}
This result still depends on $z(s)$; however, in cases where its $X_{j,\nu,con,ineq}$ norm is computable, it can provide us with a sharper result. As mentioned previously, we use the $ineq$-norm estimates in practice whenever the norm is computable; hence, this estimate will be useful for our computations.
This is a description of the numerical tools needed in order to apply Theorem \ref{th: radii polynomial continuation}. Let us now discuss the computation of the bounds.
\subsection{Computing the Bounds for Branches of Solutions}
In this section, we will compute the bounds $Y_0, Z_0, Z_1,$ and $Z_2(r)$. These computations will be similar to those done in Section \ref{sec : bounds}, but involve the use of the transformation $\mathcal{T}$ and its inverse. First, let us introduce some notation. We will choose $\mathcal{T}$ to the Fourier Transform, $\mathcal{F}$. For our purposes, we will use the Fast Fourier Transform (FFT) algorithm. This means we would choose $K_{grid} \bydef \frac{K_{\mathrm{FFT}}}{2}$ for some  $K_{\mathrm{FFT}}$ which is a power of $2$ larger than $K$. In other words, $K_{\mathrm{FFT}}$ is at least the next power of $2$ after $K$. Let 
$\mathcal{F}_{\mathcal{K}}[\cdot]$ denote the Fourier transform with $\mathcal{K}$ coefficients of a Chebyshev sequence. That is, for a chosen $\mathcal{K}$ we will compute the next power of $2$ and zero pad the Chebyshev sequence if necessary. Then we take the Chebyshev FFT. We similarly define $\mathcal{F}_{\mathcal{K}}^{-1}[\cdot]$ to be the analogous definition for the Inverse Fast Fourier Transform (IFFT). At this point, we also define pointwise multiplication, denoted $.*$ as
\begin{align} 
    w(s).*z(s) \bydef (w_0z_0,w_1z_1,\dots,w_{\mathcal{K}}z_{\mathcal{K}})~~ \text{where} ~~ (w(s))_{k \in \{0,\dots,\mathcal{K}\}}, (z(s))_{k \in \{0,\dots,\mathcal{K}\}}.
\end{align}
We are now ready compute $Y_0$ and $Z_0$.
\begin{lemma}\label{lem : Y0s}
Let $L_{N,0}$ be defined as in \eqref{def : L_NK}. Let $Y_{0}, Z_0$ be defined as
{\footnotesize\begin{align}
    Y_0 \bydef  \|\mathcal{F}_{4K}^{-1}[\mathcal{F}_{4K}[B^N(s)].*F(\mathcal{F}_{4K}[\overline{w}(s)])\|_{X_{j,\nu,con}}+ \frac{1}{L_{N,0}}\left\|\mathcal{F}_{3K}^{-1}[(\pi^{3N} - \pi^N)G(\mathcal{F}_{3K}[\overline{w}(s)])]\right\|_{\ell^1_{j,\nu,con,ineq}}.
\end{align}}
\vspace{-0.5cm}
\begin{align}
    Z_0 \bydef \left\| \mathcal{F}^{-1}_{2K}\left[ \Pi^N - \mathcal{F}_{2K}[B^N(s)].*\mathcal{F}_{2K}[\Pi^NDF(\overline{w}(s))\Pi^N]\right]\right\|_{\mathcal{B}(X_{j,\nu,con})}.
\end{align}
Then, we have $\sup_{s \in [-1,1]} \|B(s)F(\overline{w}(s))\|_{X_{j,\nu,con}} \leq Y_0$ and $\|I_d - B(s)B^{\dagger}(s)\|_{\mathcal{B}(X_{j,\nu,con})} = Z_0$.
\end{lemma}
\begin{proof}
Beginning with $Y_0$, observe that we can write
{\small\begin{align}
    \|B(s) F(\overline{w}(s))\|_{X_{j,\nu,con}} &\leq \left\| \Pi^N B(s) F(\overline{w}(s))\right\|_{X_{j,\nu,con}} + \left\| \Pi_N B(s) F(\overline{w}(s))\right\|_{X_{j,\nu,con}} \\
    &\leq \|B^N(s)F(\overline{w}(s))\|_{X_{j,\nu,con}} + \left\| \Pi_NB(s)\right\|_{\mathcal{B}(X_{j,\nu,con})} \left\| \begin{bmatrix}
        0 \\
        \pi_N f(\overline{w}(s))
    \end{bmatrix}\right\|_{X_{j,\nu,con,ineq}} \\
    &= \|B^N(s)F(\overline{w}(s))\|_{X_{j,\nu,con}} + \frac{1}{L_{N,0}} \left\| (\pi^{3N} - \pi^N)G(\overline{w}(s))\right\|_{\ell^1_{j,\nu,con,ineq}}    \end{align}}
where we used \eqref{tail ineq estimate} on the final step. Note that the use of \eqref{Xjnuconineq} is sufficient as the term $G(\overline{w}(s))$ is computable. We must now compute the pointwise products, which we will do this using the FFT. With respect to $s$, $B^N(s)$ is a polynomial of order $K$ and $F(\overline{w}(s))$ is a polynomial of order $3K$. Hence, it follows that $B(s)F(\overline{w}(s))$ is a polynomial of order $4K$ with respect to $s$. Therefore, we will use $4K$ Chebyshev coefficients when computing our Fourier transforms for this term. Moreover, $G(\overline{w}(s))$ is a polynomial of order $3K$, so we can use $3K$ for this Fourier transform. Once we have found this, we can directly apply the result of Lemma \ref{lem : Y0 and Z0} for $Y_0$ to obtain the estimate.
\par For $Z_0$, similarly as in Lemma \ref{lem : Y0 and Z0}, the tails of $B(s)$ and $B^{\dagger}(s)$ cancel exactly. This leave us with with the finite parts. Note that $B^N(s)$ and $\Pi^N DF(\overline{w}(s))\Pi^N$ are polynomials of order $K$. Hence, it follows that $B^N(s)\Pi^N DF(\overline{w}(s))\Pi^N$ is a polynomial of order $2K$ with respect to $s$. Therefore, we use $2K$ Chebyshev coefficients when computing our Fourier transforms. Once we have this, we directly apply the result of Lemma \ref{lem : Y0 and Z0} to obtain the estimate. 
\end{proof}
Next, we must compute $Z_2(r)$. Let us state the lemma for it.
\begin{lemma}\label{lem : Z2s}
Let $q(s) = (q_n)_{n \in \{0,\dots,K\}}$ and each $q_n$ is a Fourier series with $kth$ coefficient $(q_n)_k = -2\gamma \delta_k + 6(\overline{u}_n)_k)$ where $\overline{u}(s) = (\overline{u}_n)_{n \in \{ 0,\dots,K\}}$, $(\overline{u}_n)_k$ is the $kth$ Fourier coefficient, and  $\delta_k$ is defined as in \eqref{kronecker delta}. Let $L_{N,0}$ be defined as in \eqref{def : L_NK}. Now, let $Z_2(r) : (0,\infty) \to [0,\infty)$ be defined as
\begin{align}
    Z_2(r) \bydef \max\left(\|B^N(s)\|_{\mathcal{B}(X_{j,\nu,con})}, \frac{2K+1}{L_{N,0}}\right)\left(1 + \|q(s)\|_{\ell^1_{j,\nu,con}} + 3r\right).
    \end{align}
Then, it follows that $\sup_{s \in [-1,1]} \|B(s)(DF(w(s)) - DF(\overline{w}(s)))\|_{\mathcal{B}(X_{j,\nu,con})} \leq Z_2(r)r$.
\end{lemma}
\begin{proof}
To begin, since $w(s) \in B_r(\overline{w}(s))$, there exists a $v(s) \in B_r(0)$ such that $w(s) = \overline{w}(s) + v(s)$. Note that $\|v(s)\|_{\ell^1_{j,\nu,con}} \leq r$. The proof now follows similar steps to those used in Lemma \ref{lem : Z2}. That is, we obtain
{\footnotesize\begin{align}
    \|B(s)(DF(w(s)) - DF(\overline{w}(s)))\|_{\mathcal{B}(X_{j,\nu,con})} &\leq \|B(s)\|_{\mathcal{B}(X_{j,\nu,con})}\left\|\begin{bmatrix}
        0 & 0 \\
        -v(s) & \mathbb{q}(s)\mathbb{v}(s) + 3\mathbb{v}(s)^2
    \end{bmatrix}\right\|_{\mathcal{B}(X_{j,\nu,con})} \\
    &\hspace{-5cm}\leq \max\left(\|B^N(s)\|_{\mathcal{B}(X_{j,\nu,con})},\frac{2K+1}{L_{N,0}}\right)\left(\|v(s)\|_{\mathcal{B}(\mathbb{R}_{con},\ell^1_{k,\nu,con})} + \|\mathbb{q}(s)\mathbb{v}(s)
    \|_{\mathcal{B}(\ell^1_{j,\nu,con})} + 3 \|\mathbb{v}^2\|_{\mathcal{B}(\ell^1_{j,\nu,con})}\right) \\
    &\hspace{-4cm}\leq \max\left(\|B^N(s)\|_{\mathcal{B}(X_{j,\nu,con})}, \frac{2K+1}{L_{N,0}}\right)\left(r + \|q(s)\|_{\ell^1_{j,\nu,con}} r + 3r^2\right) \bydef Z_{2}(r)r
\end{align}}
where $\mathbb{q}$ and $\mathbb{v}$ are multiplication operators as in \eqref{def : discrete conv operator} and we used that $\ell^1_{j,\nu,con}$ is a Banach algebra and similar steps to those used in Lemma \ref{lem : Z2} to obtain the estimate on $\|B(s)\|_{\mathcal{B}(X_{j,\nu,con})}$. Note that here we were forced to use the estimate with $2K+1$; however, in practice, the maximum will be $\|B^N(s)\|_{\mathcal{B}(X_{j,\nu,con})}$. As a result, the estimate does not usually factor in. This concludes the proof.
\end{proof}
Finally, we compute the $Z_1$ bound. First, we will define $\overline{v}(s) \in \ell^1_{j,\nu,con}$ as 
\begin{align}
    \overline{v}(s) \bydef -2\gamma \overline{u}(s) + 3\mathcal{F}_{2K}^{-1}[\mathcal{F}_{2K}[\overline{u}(s)].*\mathcal{F}_{2K}[\overline{u}(s)]].
\end{align}
Note that $\overline{u}(s)$ is a polynomial of order $K$ with respect to $s$. Hence, we expect $\overline{u}(s)^2$ to be a polynomial of order $2K$ with respect to $s$, so we choose our FFT size appropriately. We now state the result for $Z_1$ as a lemma.
\begin{lemma}\label{lem : Z1s}
Let $\overline{V}(s) \bydef (\overline{V}_k)_{k \in \{0,\dots,2K\}},\phi(s) \bydef (\phi_k)_{k \in \{0,\dots,2K\}} \in \ell^1_{j,\nu,con}$ where
\begin{align}
    (\overline{V}_k)_n \bydef \begin{cases}
        0 & n = (0,0) \\
        (\overline{v}_k)_n & \mathrm{else}
    \end{cases}, ~~(\phi_k)_n \bydef \begin{cases} \frac{\|\overline{V}_k\|_{\infty}}{\nu^{N+1}} & n \in I^N \\
    0 & \mathrm{else} 
    \end{cases}.
\end{align}
Let $L_{N,0}$ be defined as in \eqref{def : L_NK}. Let $Z_{1} > 0$ be defined as
\begin{align}
    Z_1 \bydef \left\| B^N(s) \begin{bmatrix}
        0 \\
        \phi(s)
    \end{bmatrix}\right\|_{X_{j,\nu,con}} + \frac{1}{L_{N,0}} \|\overline{v}(s)\|_{\ell^1_{j,\nu,con,ineq}}.
\end{align}
Then, it follows that $\sup_{s \in [-1,1]}\|B(s)(DF(\overline{w}(s)) - B^{\dagger}(s))\|_{\mathcal{B}(X_{j,\nu,con})} \le Z_1$. 
\end{lemma}
\begin{proof}
To begin, let $\mathscr{h}(s) \bydef (\eta(s),h(s)) \in X_{j,\nu,con}$. Observe that
\begin{align}
    (DF(\overline{w}(s)) - B^{\dagger}(s))\mathscr{h}(s) &= \begin{bmatrix}
        (h(s),\dot{u}(s))_2 \\
        \eta(s)\overline{u}(s) + Df(\overline{w}(s))h(s)
    \end{bmatrix} - \begin{bmatrix}
        (h(s),\dot{u}(s)\pi^N)_2 \\
        \eta(s) \pi^N \overline{u}(s) + \pi^N Df(\overline{w}(s))\pi^N h(s)
    \end{bmatrix} \\
    &= \begin{bmatrix}
        (h(s),\dot{u}(s) - \dot{u}(s) \pi^N)_2 \\
        \eta(s)(\overline{u}(s) - \pi^N \overline{u}(s)) + (Df(\overline{w}(s)) - \pi^N Df(\overline{w}(s))\pi^N)h(s) 
    \end{bmatrix} \\
    &= \begin{bmatrix}
        0 \\
        (Df(\overline{w}(s)) - \pi^N Df(\overline{w}(s))\pi^N)h(s)
    \end{bmatrix}
\end{align}
where the last step followed from the fact that $\pi^N \overline{u}(s) = \overline{u}(s), \dot{u}(s) \pi^N = \dot{u}(s)$ by definition. Hence, we introduce $z(s) \bydef (Df(\overline{w}(s)) - \pi^N Df(\overline{w}(s))\pi^N)h(s)$ and write
\begin{align}
    \|B(s)(DF(\overline{w}(s)) - B^{\dagger}(s))\|_{\mathcal{B}(X_{j,\nu,con})} &= \left\| B(s) \begin{bmatrix}
        0 \\
        z(s)
    \end{bmatrix}\right\|_{X_{j,\nu,con}} \\
    &=\left\| B^N(s) \begin{bmatrix}
        0 \\
        z(s)
    \end{bmatrix}\right\|_{X_{j,\nu,con}} + \left\|\Pi_N B(s) \begin{bmatrix}
        0 \\
        z(s)
    \end{bmatrix}\right\|_{X_{j,\nu,con}} \\
    &\hspace{-3.1cm}\leq \left\| B^N(s) \begin{bmatrix}
        0 \\
        \overline{v}(s) \pi_N h(s)
    \end{bmatrix}\right\|_{X_{j,\nu,con}} + \|\Pi_N B(s)\|_{\mathcal{B}(X_{j,\nu,con})} \left\| \Pi_N \begin{bmatrix}
        0 \\ z(s)
    \end{bmatrix}\right\|_{X_{j,\nu,con,ineq}} \\
    &\leq \left\| B^N(s) \begin{bmatrix}
        0 \\
        \phi(s)
    \end{bmatrix}\right\|_{X_{j,\nu,con}} + \frac{1}{L_{N,0}} \|\overline{v}(s)\|_{\ell^1_{j,\nu,con,ineq}} \bydef Z_1
\end{align}
where we used \eqref{tail ineq estimate} since $z(s)$ is computable and we would use it anyway. This concludes the proof.
\end{proof}
This concludes the computation of the bounds needed to apply Theorem \ref{th: radii polynomial continuation}. We are now able to prove the existence of symmetric branches of solutions periodic on a triangle and hexagon.
\subsection{Hexagonal and Triangular Branches of Periodic Solutions}\label{sec : results branches}
In this section, we present the proofs of five branches of solutions in \eqref{eq : swift_hohenberg}. The first two will possess $D_3$-symmetry, meaning each periodic pattern on the branch can be generated by tiling with $\Delta_1$ and $\Delta_2$ according to Theorem \ref{th : triangle periodic}. The latter three will possess $D_6$-symmetry, meaning each solution on the branch is periodic on $\varhexagon_0$ according to Theorem \ref{th : hexagon periodic}. We now present the results.
\begin{theorem}[\bf The First Triangular Branch]\label{th : triangle branch 1}
Let $s_{fix} = 0.05, \gamma = 1.6, j = 3, \nu = 1.1$. Moreover, let $r_0 \bydef 2 \times 10^{-5}, $. Then there exists a unique solution $\tilde{w}(s)$ to \eqref{eq : psuedo arclength system} in $\overline{B_{r_0}(\overline{w}(s))} \subset X_{3,1.1,con}$ and we have that $\sup_{s \in [-1,1]} \|\tilde{w}(s)-\overline{w}(s)\|_{X_{3,1.1,con}} \leq r_0$. This corresponds to a branch of periodic solutions in \eqref{eq : swift_hohenberg} under which the pattern can be generated by a tiling of $\Delta_1$ and $\Delta_2$.     
\end{theorem}
\begin{proof}
Choose $N = 40, d = 5, K = 3$. Then, we perform the full construction described in Section \ref{sec : continuation construction} to build $\overline{w}(s)$ and $B^N(s)$. Using \cite{julia_blanco_D3D6,dominic_dihedral_julia}, we choose $r_0 \bydef 2 \times 10^{-5} $ and obtain
{\small\begin{align}
    \|B^N(s)\|_{\mathcal{B}(X_{3,1.1,con})} \leq  113.916\text{,}~Y_0 \bydef 7.152 \times 10^{-6} \text{,}~Z_{2}(r_0) \bydef 11174.51   \text{,}~Z_1 \bydef 0.4737,~Z_0 \bydef 0.02671.
    \end{align}}
We prove that these values satisfy Theorem \ref{th: radii polynomial continuation}. 
\end{proof}
\begin{figure}[H]
\centering
 \begin{minipage}{.33\linewidth}
  \centering\epsfig{figure=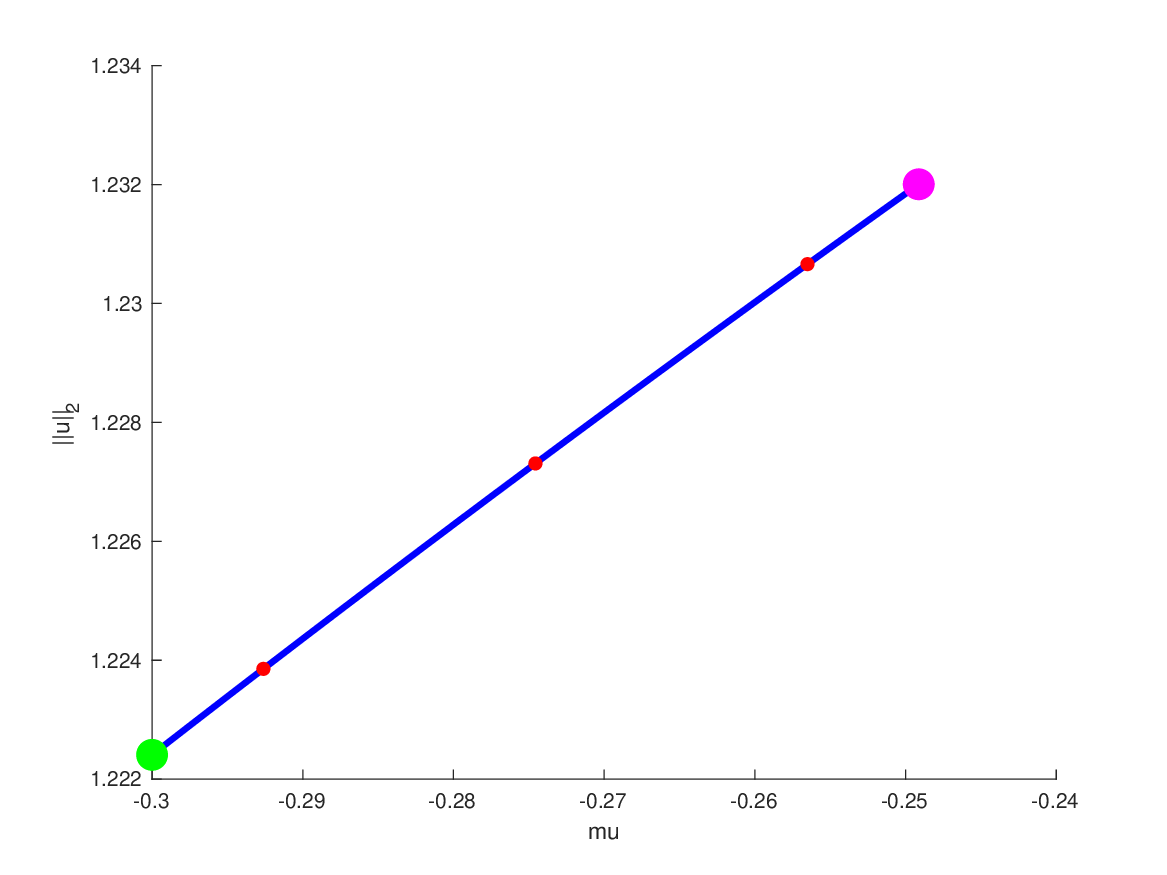,width=\linewidth}
  \end{minipage}%
 \begin{minipage}{.33\linewidth}
  \centering\epsfig{figure=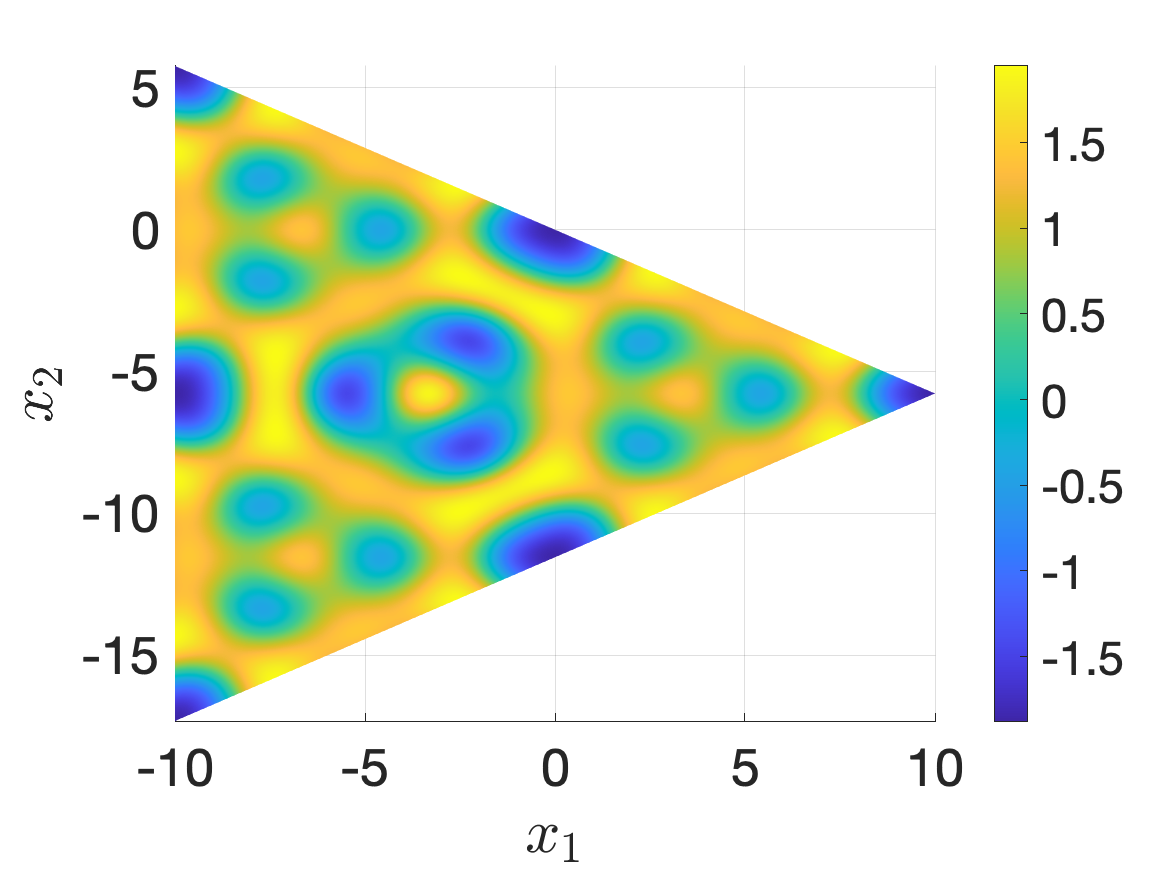,width=\linewidth}
 \end{minipage} 
\begin{minipage}{.33\linewidth}
  \centering\epsfig{figure=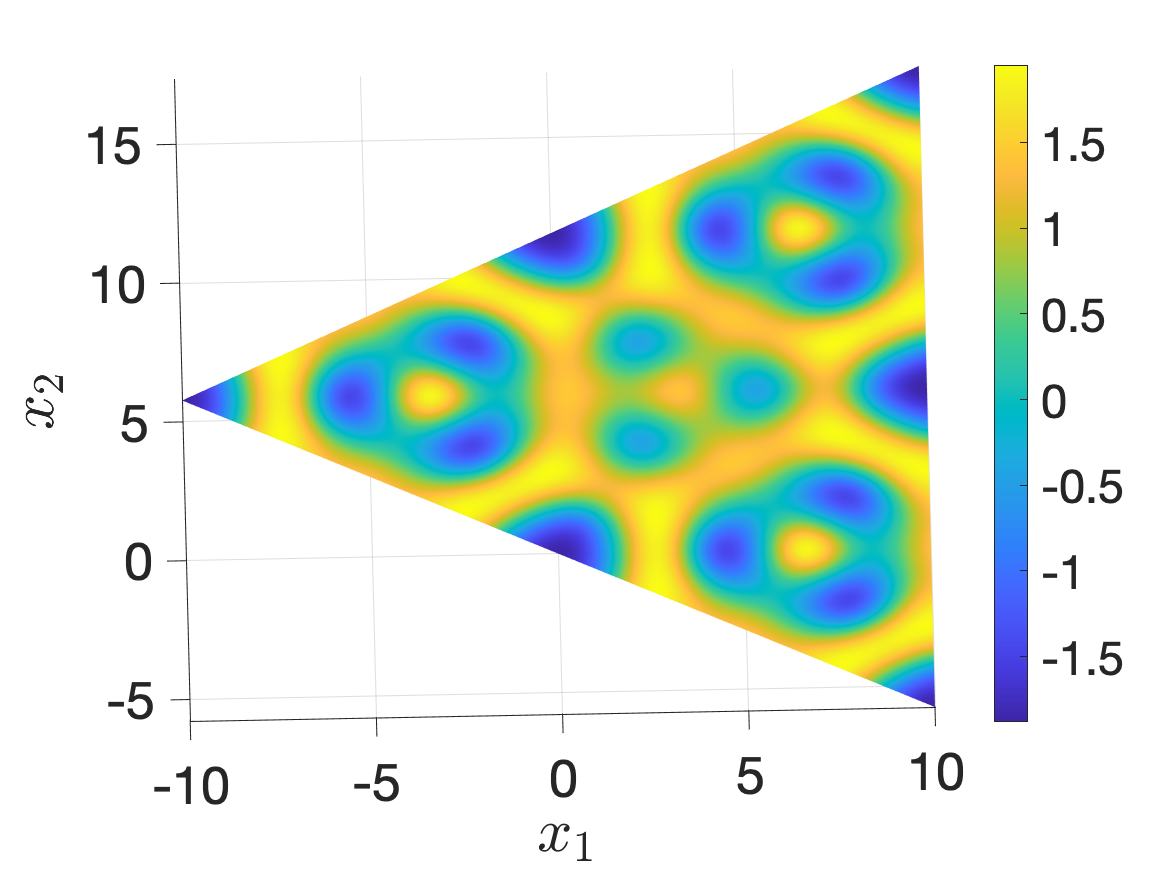,width=\linewidth}
 \end{minipage}
 \caption{Plot of the branch of approximate solutions proven in Swift Hohenberg with $D_3$-symmetry (left) used in  Theorem \ref{th : triangle branch 1}. The approximation on the branch (magenta point) when $\mu \approx -0.24913$ on $\Delta_1$ (center) and $\Delta_2$ (right). Note that the green point is plotted in Figure \ref{fig : intro} (center) and (right).}\label{fig : th3_1}
 \end{figure}%
\begin{theorem}[\bf The Second Triangular Branch]\label{th : triangle branch 2}
Let $s_{fix} = 0.5, \gamma = 0.3, j=3, \nu = 1.4$. Moreover, let $r_0 \bydef 2 \times 10^{-6}, $. Then there exists a unique solution $\tilde{w}(s)$ to \eqref{eq : psuedo arclength system} in $\overline{B_{r_0}(\overline{w}(s))} \subset X_{3,1.4,con}$ and we have that $\sup_{s \in [-1,1]} \|\tilde{w}(s)-\overline{w}(s)\|_{X_{3,1.4,con}} \leq r_0$. This corresponds to a branch of periodic solutions in \eqref{eq : swift_hohenberg} under which the pattern can be generated by a tiling of $\Delta_1$ and $\Delta_2$.     
\end{theorem}
\begin{proof}
Choose $N = 24, d = 5, K = 31$. Then, we perform the full construction described in Section \ref{sec : continuation construction} to build $\overline{w}(s)$ and $B^N(s)$. Using \cite{julia_blanco_D3D6,dominic_dihedral_julia}, we choose $r_0 \bydef 2 \times 10^{-6} $ and obtain
{\small\begin{align}
    \|B^N(s)\|_{\mathcal{B}(X_{3,1.4,con})} \leq  3241.65\text{,}~Y_0 \bydef 1.789 \times 10^{-7} \text{,}~Z_{2}(r_0) \bydef 257070.1   \text{,}~Z_1 \bydef 0.102,~Z_0 \bydef 8.063 \times 10^{-3}.
    \end{align}}
We prove that these values satisfy Theorem \ref{th: radii polynomial continuation}. 
\end{proof}
\begin{figure}[H]
\centering
 \begin{minipage}{.33\linewidth}
  \centering\epsfig{figure=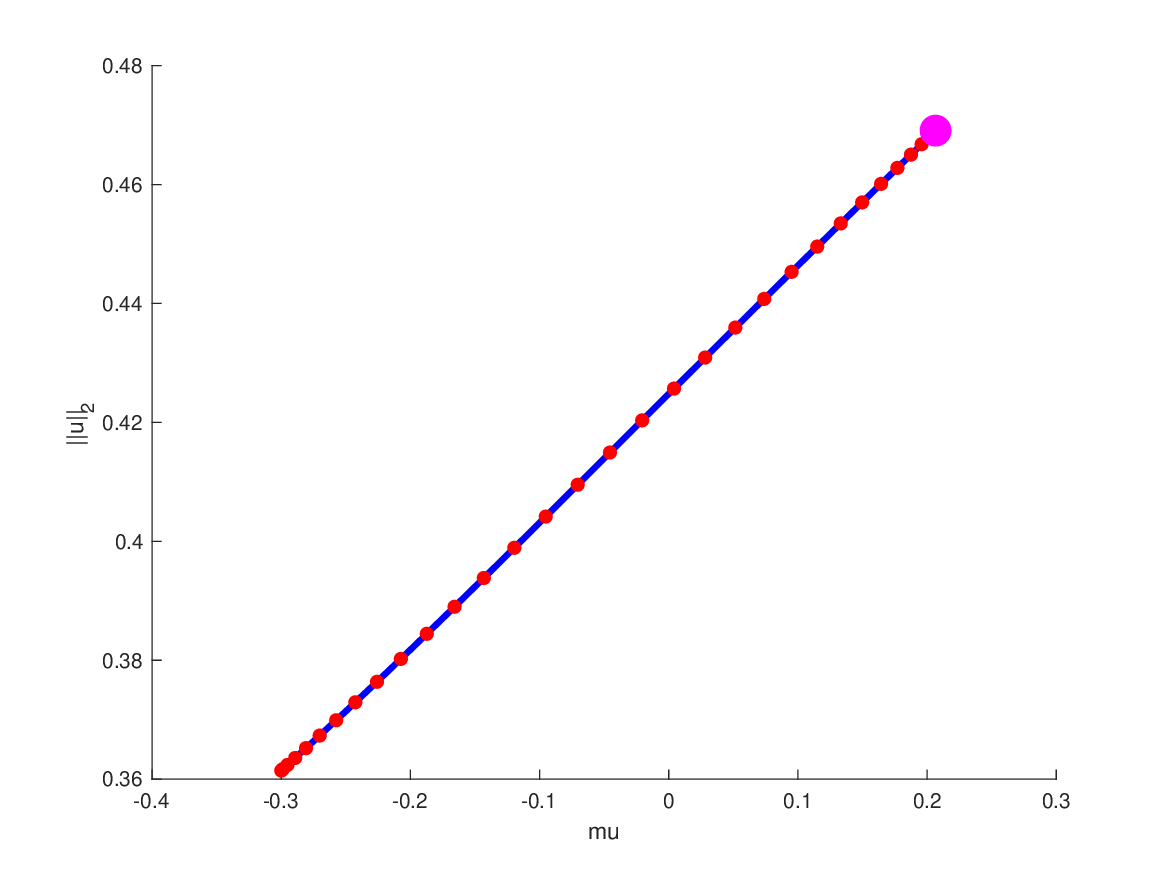,width=\linewidth}
  \end{minipage}%
 \begin{minipage}{.33\linewidth}
  \centering\epsfig{figure=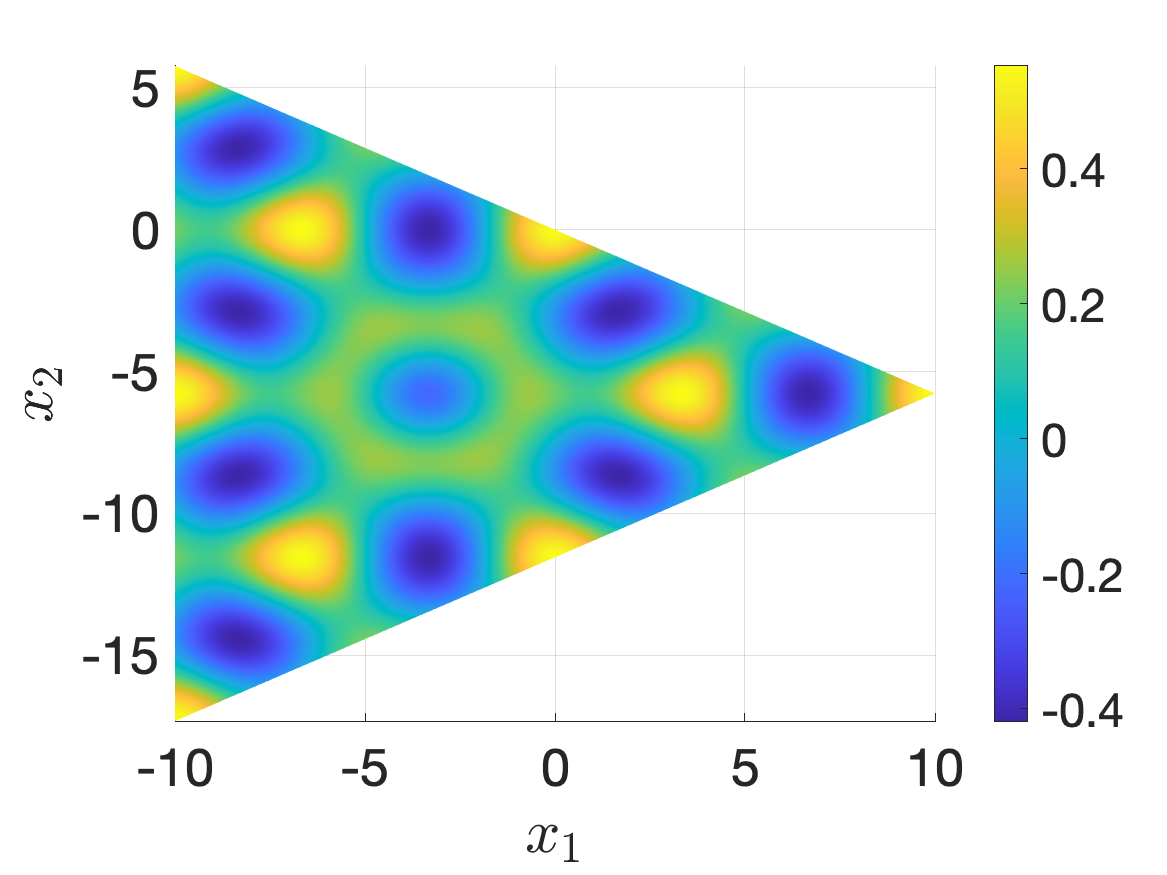,width=\linewidth}
 \end{minipage} 
 \begin{minipage}{.33\linewidth}
  \centering\epsfig{figure=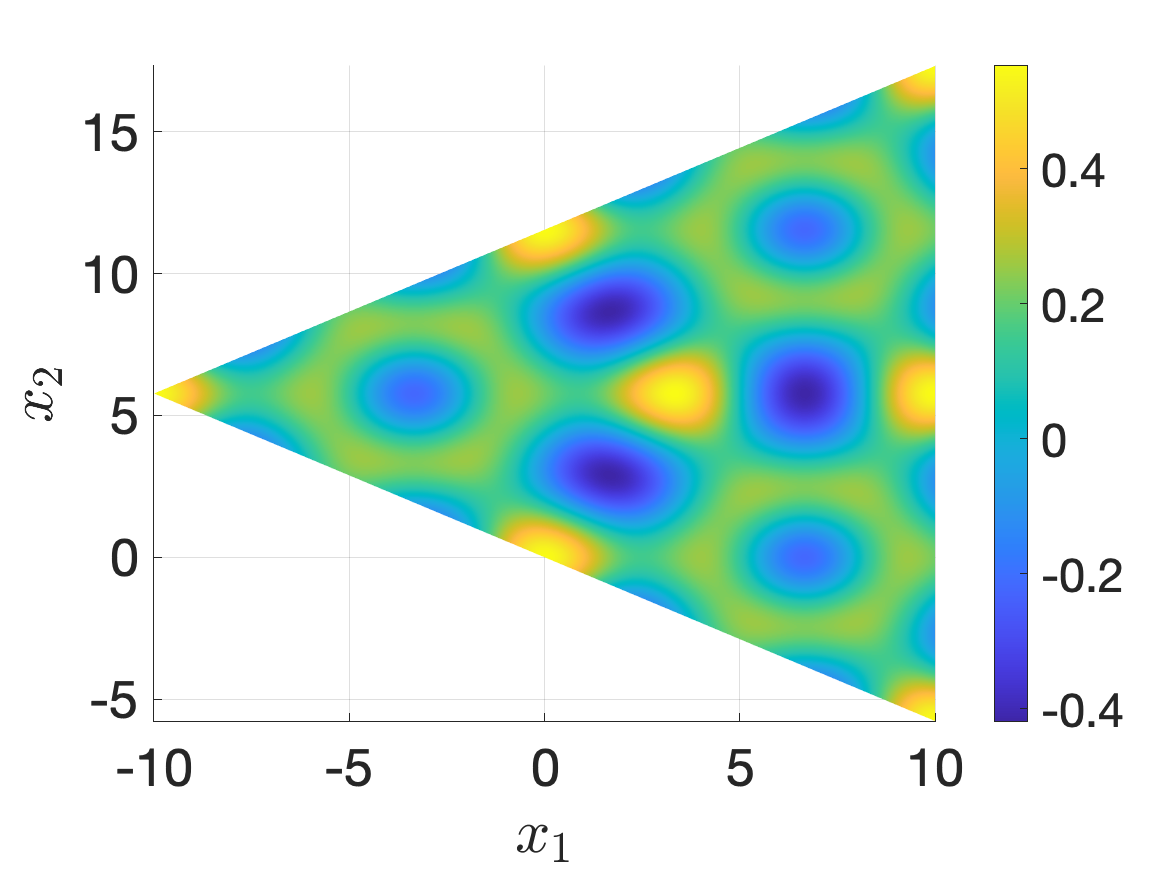,width=\linewidth}
 \end{minipage}
 \caption{Plot of the branch of approximate solutions proven in Swift Hohenberg with $D_3$-symmetry (left) used in  Theorem \ref{th : triangle branch 2}. The approximation on the branch (magenta point) when $\mu \approx 0.206565$ on $\Delta_1$ (center) and $\Delta_2$ (right).}\label{fig : th3_2}
 \end{figure}%
We now present the proof of the hexagonal branch of solutions.
\begin{theorem}[\bf The First Hexagonal Branch]\label{th : hexagon branch}
Let $s_{fix} = 0.4, \gamma = 1.6, j = 6, \nu = 1.1$. Moreover, let $r_0 \bydef, 1 \times 10^{-4}$. Then there exists a unique solution $\tilde{w}(s)$ to \eqref{eq : psuedo arclength system} in $\overline{B_{r_0}(\overline{w}(s))} \subset X_{6,1.1,con}$ and we have that $\sup_{s \in [-1,1]} \|\tilde{w}(s)-\overline{w}(s)\|_{X_{6,1.1,con}} \leq r_0$. This corresponds to a branch of periodic solutions on $\varhexagon_0$ to \eqref{eq : swift_hohenberg}. 
\end{theorem}
\begin{proof}
Choose $N = 60, d = 10, K = 7$. Then, we perform the full construction described in Section \ref{sec : continuation construction} to build $\overline{w}(s)$ and $B^N(s)$. Using \cite{julia_blanco_D3D6,dominic_dihedral_julia}, we choose $r_0 \bydef 1 \times 10^{-4}$ and obtain
{\small\begin{align}
    \|B^N(s)\|_{\mathcal{B}(X_{6,1.1,con})} \leq 45.28 \text{,}~Y_0 \bydef 4.999 \times 10^{-5} \text{,}~Z_{2}(r_0) \bydef 3957.77   \text{,}~Z_1 \bydef 0.2327,~Z_0 \bydef 1.005 \times 10^{-2}.
    \end{align}}
We prove that these values satisfy Theorem \ref{th: radii polynomial continuation}. 
\end{proof}
\begin{figure}[H]
\centering
 \begin{minipage}{.5\linewidth}
  \centering\epsfig{figure=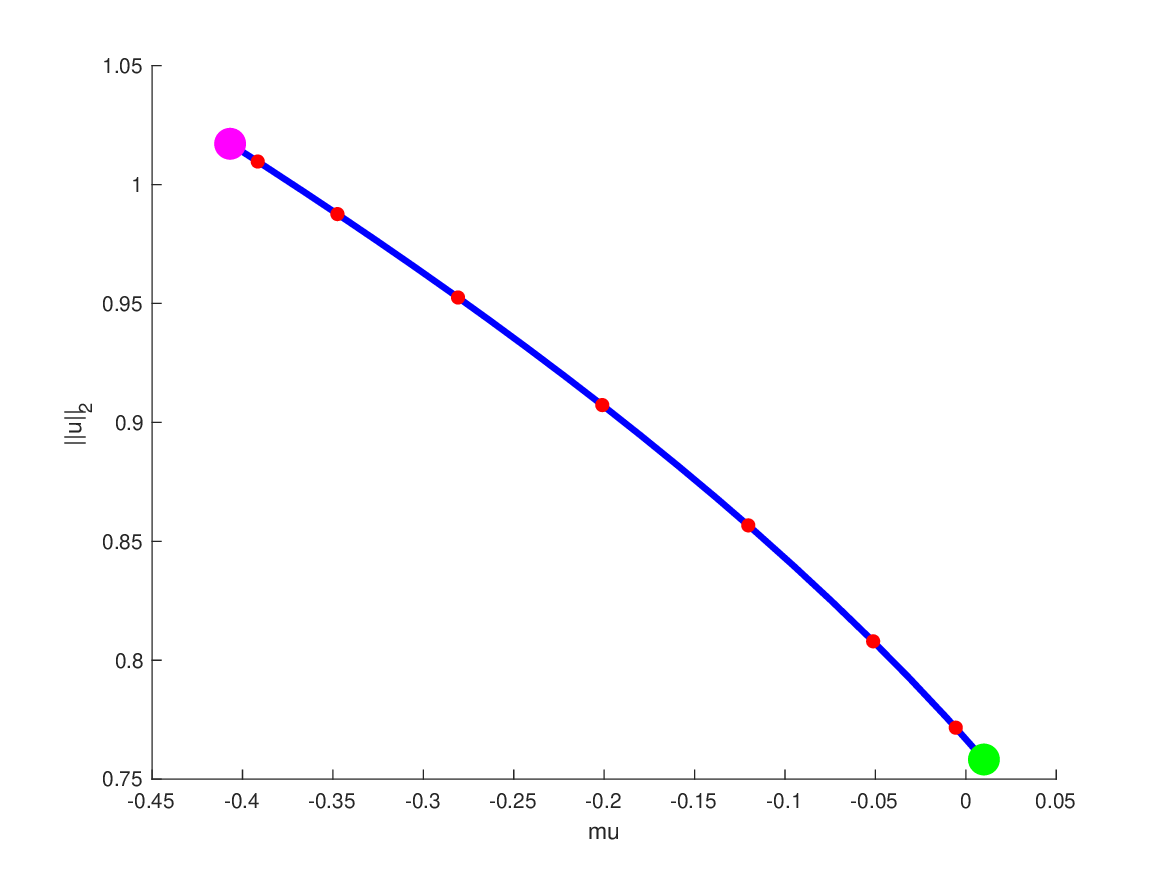,width=\linewidth}
  \end{minipage}%
 \begin{minipage}{.5\linewidth}
  \centering\epsfig{figure=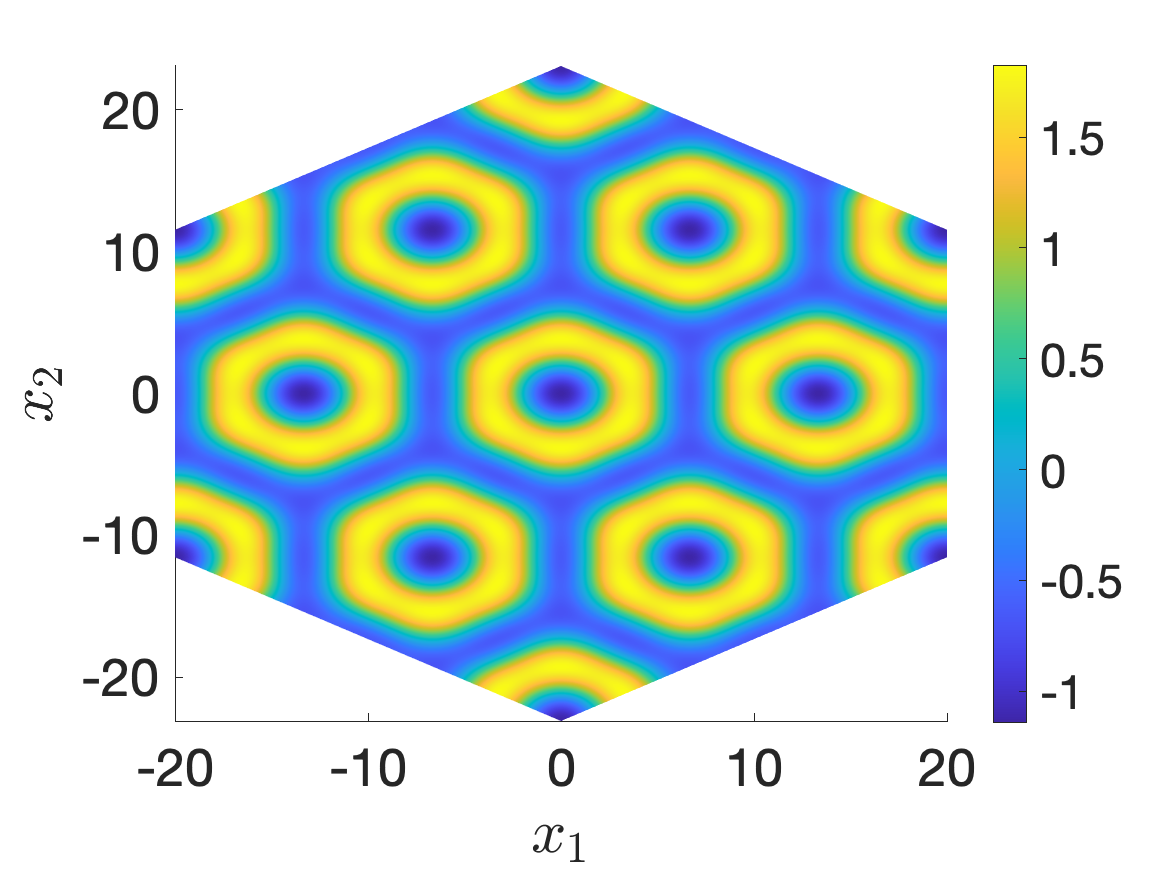,width=\linewidth}
 \end{minipage} 
 \caption{Plot of the branch of approximate solutions proven in Swift Hohenberg with $D_6$-symmetry (left) used in  Theorem \ref{th : hexagon branch}. An approximate solution on the branch (magenta point) when $\mu \approx -0.407$ (right). Note that the green point is plotted in Figure \ref{fig : intro} (left).}\label{fig : th4_1}
 \end{figure}%
\begin{theorem}[\bf The Second Hexagonal Branch]\label{th : hexagon branch 2}
Let $s_{fix} = 0.18, \gamma = 1.6, j = 6, \nu = 1.25$. Moreover, let $r_0 \bydef, 5 \times 10^{-5}$. Then there exists a unique solution $\tilde{w}(s)$ to \eqref{eq : psuedo arclength system} in $\overline{B_{r_0}(\overline{w}(s))} \subset X_{6,1.25,con}$ and we have that $\sup_{s \in [-1,1]} \|\tilde{w}(s)-\overline{w}(s)\|_{X_{6,1.25,con}} \leq r_0$. This corresponds to a branch of periodic solutions on $\varhexagon_0$ to \eqref{eq : swift_hohenberg}. 
\end{theorem}
\begin{proof}
Choose $N = 20, d = 5, K = 31$. Then, we perform the full construction described in Section \ref{sec : continuation construction} to build $\overline{w}(s)$ and $B^N(s)$. Using \cite{julia_blanco_D3D6,dominic_dihedral_julia}, we choose $r_0 \bydef 5 \times 10^{-5}$ and obtain
\begin{align}
    \|B^N(s)\|_{\mathcal{B}(X_{6,1.25,con})} \leq 36.122 \text{,}~Y_0 \bydef 1.854 \times 10^{-5} \text{,}~Z_{2}(r_0) \bydef 8235.94   \text{,}~Z_1 \bydef 0.2195,~Z_0 \bydef 0.03361.
    \end{align}
We prove that these values satisfy Theorem \ref{th: radii polynomial continuation}. 
\end{proof}
\begin{figure}[H]
\centering
 \begin{minipage}{.33\linewidth}
  \centering\epsfig{figure=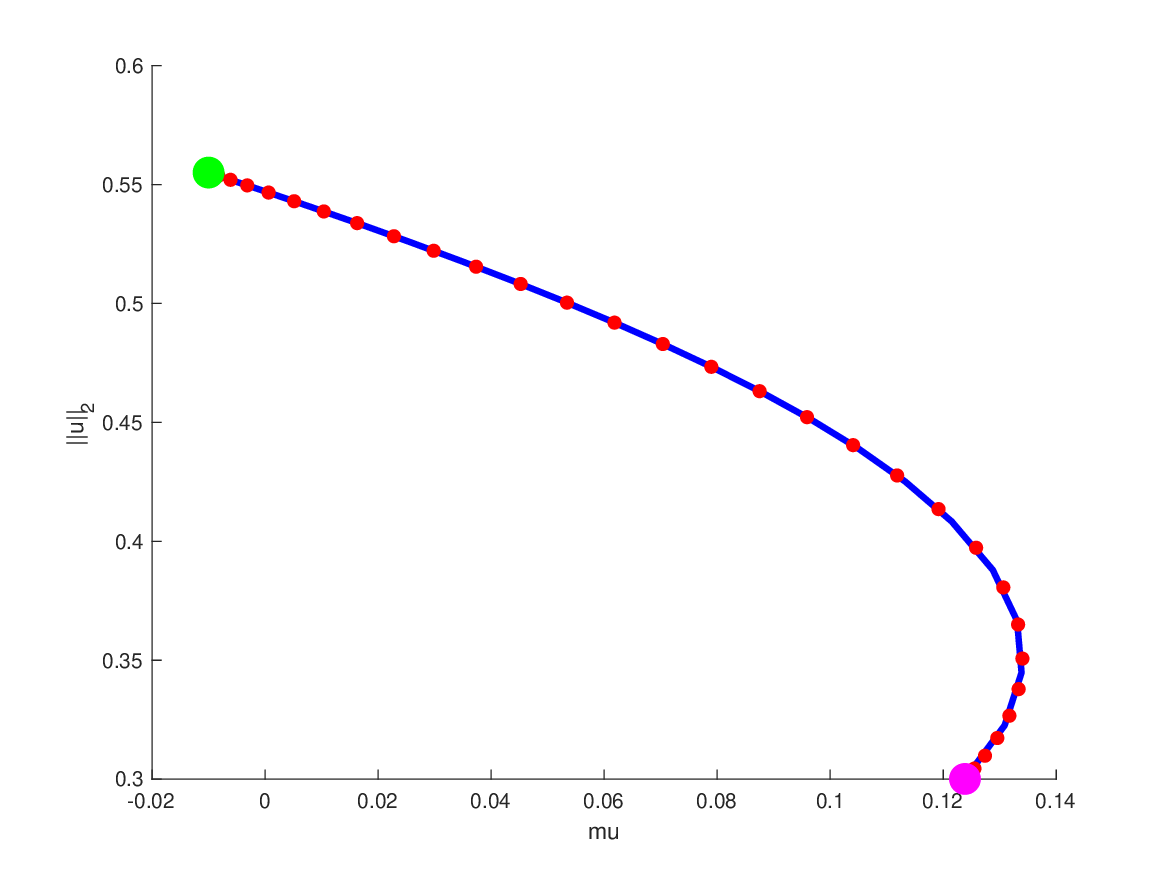,width=\linewidth}
  \end{minipage}%
 \begin{minipage}{.33\linewidth}
  \centering\epsfig{figure=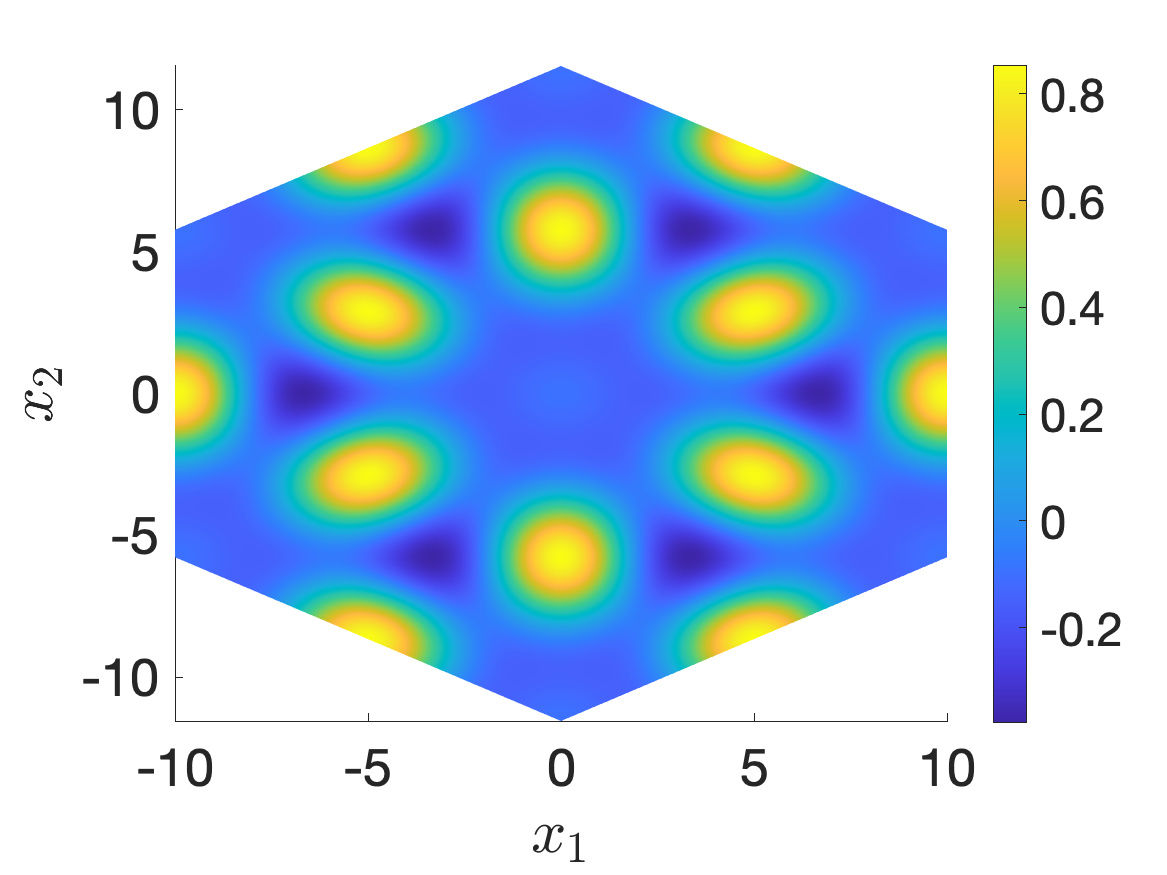,width=\linewidth}
 \end{minipage} 
 \begin{minipage}{.33\linewidth}
  \centering\epsfig{figure=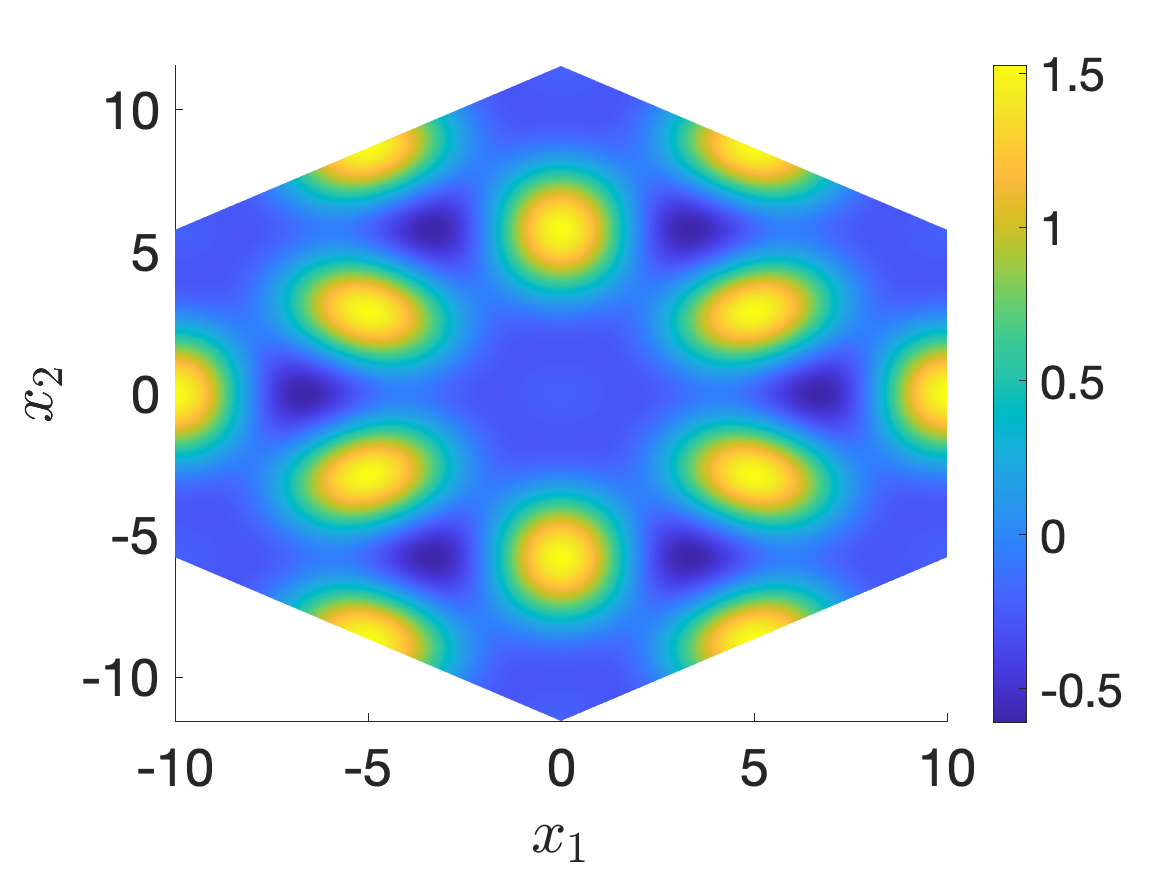,width=\linewidth}
 \end{minipage} 
 \caption{Plot of the branch of approximate solutions proven in Swift Hohenberg with $D_6$-symmetry (left) used in  Theorem \ref{th : hexagon branch 2}. An approximate solution on the branch (magenta point) when $\mu \approx 0.12385$ (center). An approximate solution on the branch (green point) when $\mu \approx -0.01$ (right).}\label{fig : th4_2}
 \end{figure}%
\begin{theorem}[\bf The Third Hexagonal Branch]\label{th : hexagon branch 3}
Let $s_{fix} = 0.23, \gamma = 1.6, j = 6, \nu = 1.25$. Moreover, let $r_0 \bydef 3 \times 10^{-5}$. Then there exists a unique solution $\tilde{w}(s)$ to \eqref{eq : psuedo arclength system} in $\overline{B_{r_0}(\overline{w}(s))} \subset X_{6,1.25,con}$ and we have that $\sup_{s \in [-1,1]} \|\tilde{w}(s)-\overline{w}(s)\|_{X_{6,1.25,con}} \leq r_0$. This corresponds to a branch of periodic solutions on $\varhexagon_0$ to \eqref{eq : swift_hohenberg}. 
\end{theorem}
\begin{proof}
Choose $N = 20, d = 5, K = 15$. Then, we perform the full construction described in Section \ref{sec : continuation construction} to build $\overline{w}(s)$ and $B^N(s)$. Using \cite{julia_blanco_D3D6,dominic_dihedral_julia}, we choose $r_0 \bydef 3 \times 10^{-5}$ and obtain
\begin{align}
    \|B^N(s)\|_{\mathcal{B}(X_{6,1.25,con})} \leq 44.04 \text{,}~Y_0 \bydef 1.49 \times 10^{-6} \text{,}~Z_{2}(r_0) \bydef 6065.28   \text{,}~Z_1 \bydef 0.4403,~Z_0 \bydef 0.1857.
    \end{align}
We prove that these values satisfy Theorem \ref{th: radii polynomial continuation}. 
\end{proof}
\begin{figure}[H]
\centering
 \begin{minipage}{.25\linewidth}
  \centering\epsfig{figure=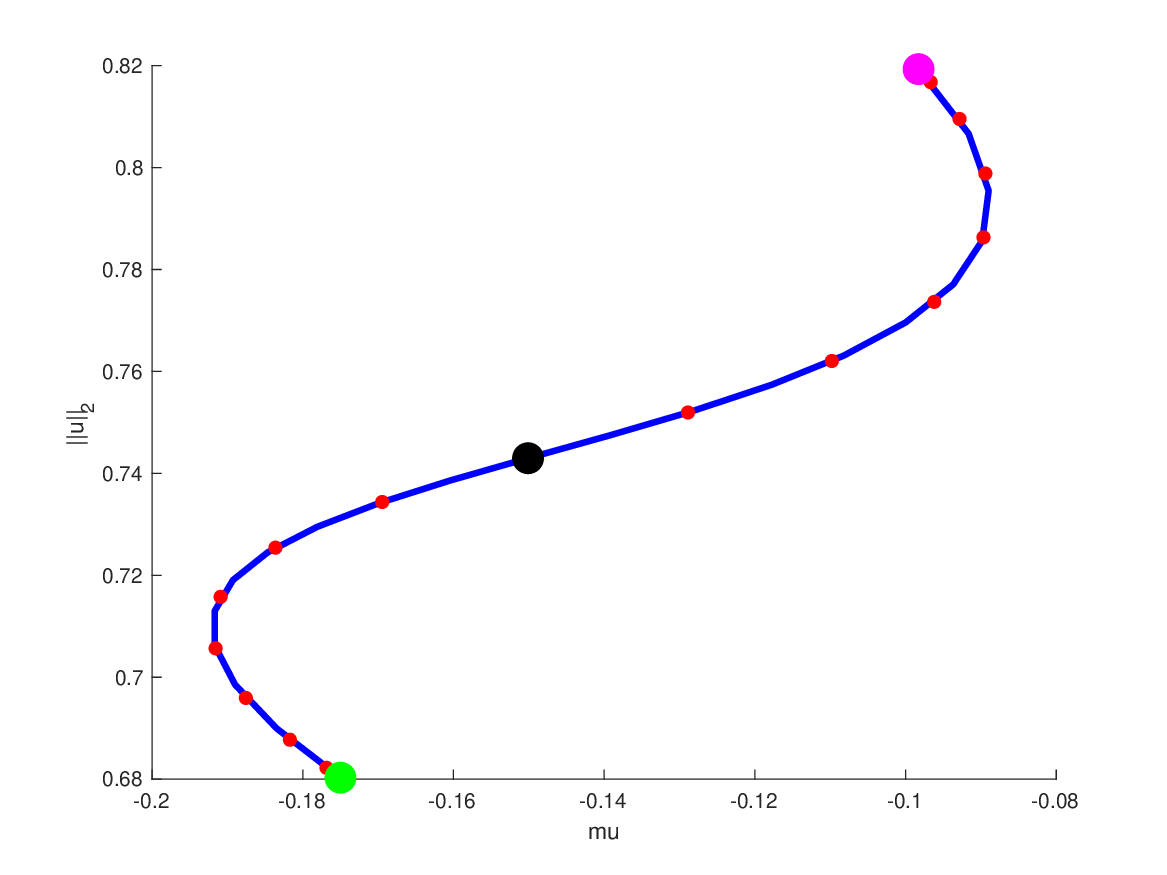,width=\linewidth}
  \end{minipage}%
 \begin{minipage}{.24\linewidth}
  \centering\epsfig{figure=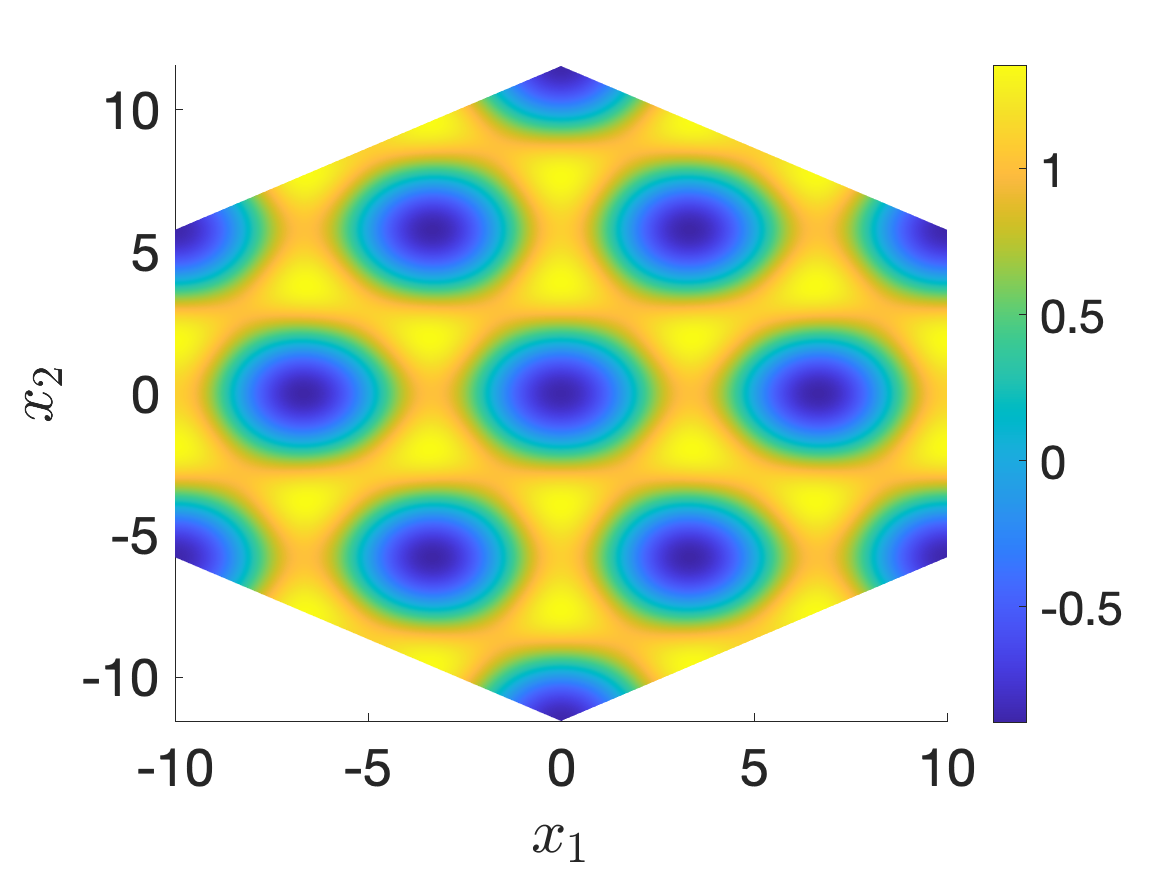,width=\linewidth}
 \end{minipage} 
 \begin{minipage}{.24\linewidth}
  \centering\epsfig{figure=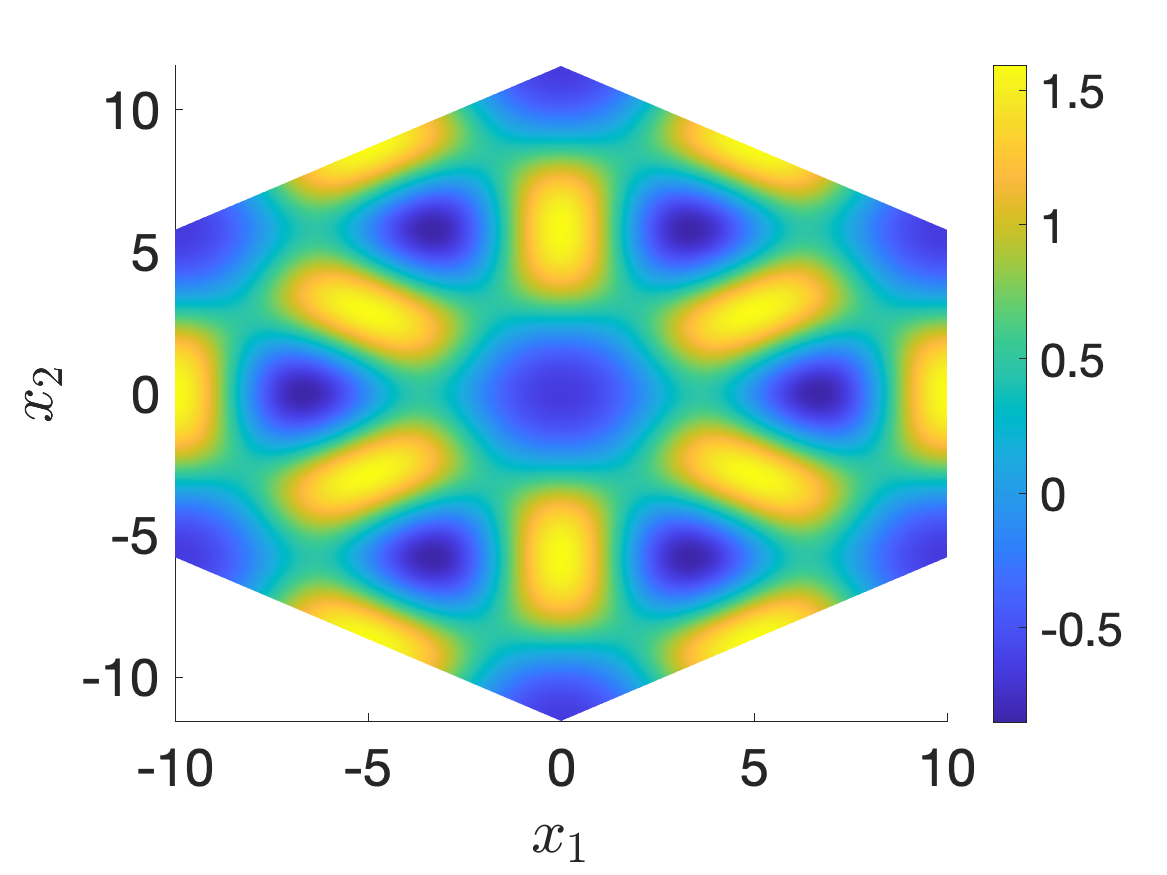,width=\linewidth}
 \end{minipage}
 \begin{minipage}{.24\linewidth}
  \centering\epsfig{figure=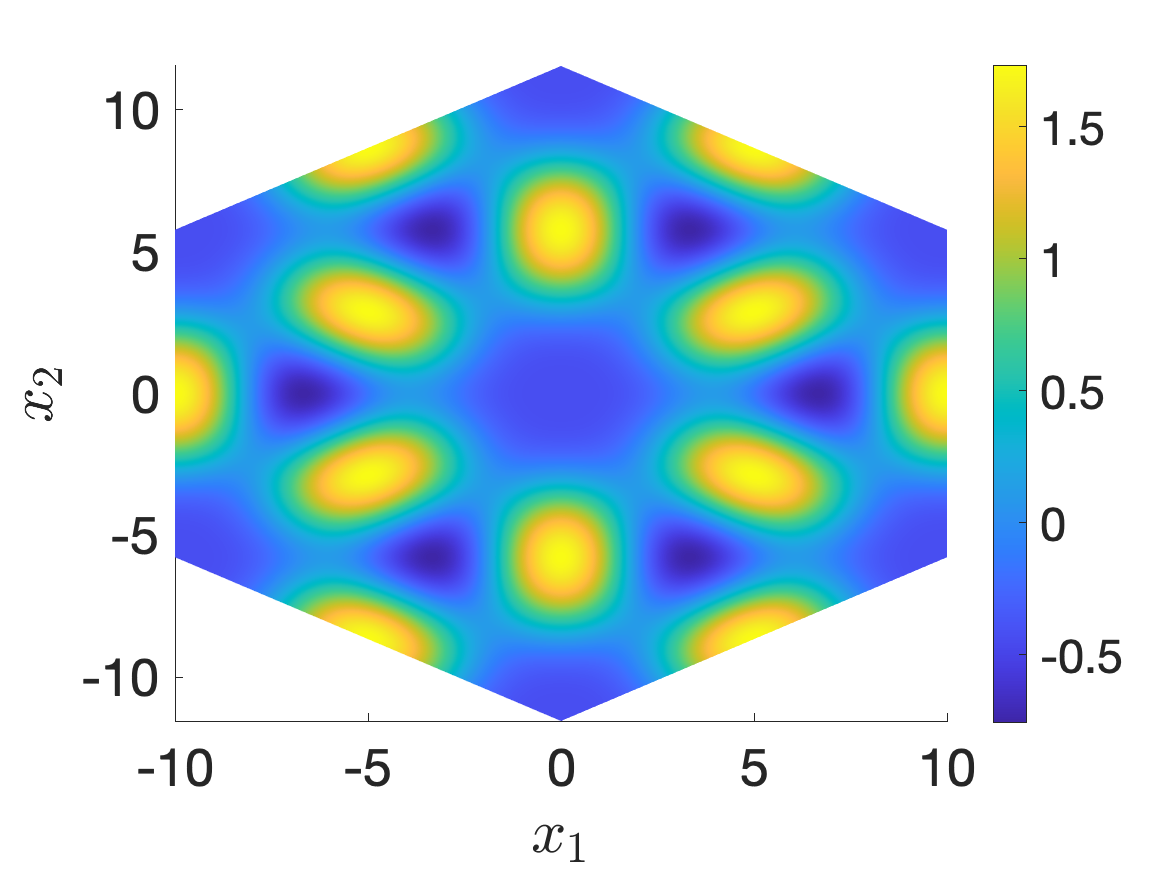,width=\linewidth}
 \end{minipage}
 \caption{Plot of the branch of approximate solutions proven in Swift Hohenberg with $D_6$-symmetry (left) used in  Theorem \ref{th : hexagon branch 3}. An approximate solution on the branch (magenta point) when $\mu \approx -0.0983$ (middle left). An approximate solution on the branch (black point) when $\mu \approx -0.1501$ (middle right). An approximate solution on the branch (green point) when $\mu \approx -0.175$ (right).}\label{fig : th4_3}
 \end{figure}%
\section{Conclusion}\label{sec : conclusion}
In this paper, we provided an approach for rigorously proving the existence, local uniqueness, and $D_j$ for $j = 3,6$ symmetry of periodic solutions to the 2D Swift Hohenberg PDE. We provided the necessary tools via \cite{dominic_dihedral_julia} for performing the computer assisted proof using \cite{julia_olivier}. Additionally, for $D_3$, we showed that we can generate the periodic tiling via two triangles invariant under $D_3$-symmetry. For $D_6$, we showed that the same result can be obtained with only one hexagon leading to a periodic function defined on a hexagon.
\par There are many future works that could build on our approach. One could be to use the reduced set for $D_3$ and $D_6$ to prove localized patterns with those symmetries. This is related to the approach developed by the authors of \cite{symmetry_blanco_cadiot}. In their methodology, it is a requirement that the Fourier series of $u_0$ is on the square lattice. The problematic is that this method requires one to consider the domain $\parallelogram_0$ and extend the function by $0$ on $\mathbb{R}^2$ rather than periodically. Since the domain $\parallelogram_0$ itself is not invariant under $D_3$ nor $D_6$-symmetry, one cannot construct an approximate solution of a localized pattern this way. This means the authors cannot utilize the useful symmetry reduction that hexagonal lattice symmetries provide, which would be particularly helpful when proving localized patterns as they are often more computationally intensive. Furthermore, it was shown by the authors of \cite{sh_cadiot} that along with the localized pattern, we obtain a proof of a branch of periodic solutions converging to the localized pattern. Currently, this branch of periodic solutions would be on the square due to the fact that the method developed in \cite{symmetry_blanco_cadiot} requires a Fourier series posed on the square lattice. Since hexagons tile the plane, and we only need one hexagon to generate the periodic tiling, taking the limit as the period tends to infinity, one can expect related localized patterns. This would make it of great interest to combine the approaches, and we consider it future work to obtain a branch of periodic solutions on a hexagon converging to a $D_6$ localized pattern. 
\par Another possible interest is computational. In its current form, the sequence structures for $D_3$ and $D_6$ have many indices which are of zero value. Since many operations, particularly convolution, are done by looping over all of $I^N$, there are many extra loop iterations being performed that are unnecessary. The ability to improve this would make the use of \cite{dominic_dihedral_julia} of higher interest for more computationally intensive problems.
\bibliographystyle{abbrv}
\bibliography{biblio}
\end{document}